\documentclass[11pt]{amsart}

\usepackage[leqno]{amsmath}
\usepackage{amsfonts,amsthm,amssymb,mathrsfs,amscd,mathabx,mathscinet}
\usepackage{epsfig}
\usepackage{graphicx}
\usepackage{subfigure}
\usepackage{longtable}
\usepackage{cite,url,stackrel}
\usepackage{hyperref,enumitem}
\hypersetup{
    colorlinks,
    linkcolor={blue!50!black},
    citecolor={red!50!black},
    urlcolor={blue!80!black}
}
\usepackage{todonotes}

\usepackage{color}
\usepackage{xspace}

\usepackage{tikz}
\usetikzlibrary{positioning}
\usetikzlibrary{patterns}
\usetikzlibrary{math, calc}

\newcommand {\mr}{\mathrm}

\def\HH{\mathscr H}
\renewcommand{\AA}{\mathscr A}

\newcommand{\PF}{\mathscr P_{\operatorname{fin}}} 

\newcommand{\covset}{\ensuremath{\mathscr{L}}} 
\newcommand{\covpos}{\ensuremath{\mathscr{F}}} 
\newcommand{\flatset}{\ensuremath{{L}}} 
\newcommand{\flatpos}{\ensuremath{{F}}} 
\newcommand{\central}{\ensuremath{\mathcal{K}}} 
\newcommand{\CC}{\central} 
\newcommand{\parclass}[1]{\pi(#1)}
\renewcommand{\SS}{\mathscr S} 
\newcommand{\ground}{E} 
\newcommand{\LL}{\mathcal{L}} 
\newcommand{\rk}{\ensuremath{\operatorname{rk}}} 
\newcommand{\rl}{\rk_\LL} 
\newcommand{\rc}{\rk} 
\newcommand{\cl}{\operatorname{cl}} 
\newcommand{\nn}{\mathbb N}
\def\ssq{\subseteq}
\newcommand{\separ}[2]{\delta_{#1} (#2)} 

\newcommand{\AC}{\texttt{AC}}
\newcommand{\SmC}{\texttt{SCAT}}
\newcommand{\POS}{\texttt{POS}}
\newcommand{\Ob}[1]{\operatorname{Ob}(#1)}
\newcommand{\Mor}[1]{\operatorname{Mor}(#1)}
\usepackage{stmaryrd}
\newcommand{\qc}{\sslash} 

\newcommand{\pareti}{B^{\square}}

\newcommand{\pcp}{sliding\xspace}

\renewcommand{\top}{\widehat{1}}
\renewcommand{\bot}{\widehat{0}}
\newcommand{\onemore}[1]{{#1^{(e)}}}

\newcommand{\Bext}{\widetilde{B}}
\newcommand{\stab}{\operatorname{stab}}

\newcommand{\wtop}[1]{
#1^{\wedge}
}
\newcommand{\wbot}[1]{
#1^{\vee}}
\newcommand{\wtb}[1]{
#1^{\wedge\!\!\vee}
}

\newcommand{\FF}{\mathcal F}

\def\signs{\{+,-,0\}} 

\newcommand{\fkz}{\covpos} 
\newcommand{\fkzb}{\covpos_{\Bext}} 
\newcommand{\fkzbsub}{(\fkzb)} 

\newcommand{\ze}[1]{{\operatorname{ze}(#1)}} 

\newcommand{\rcw}{K} 

\newtheorem{lemma}{Lemma}[section]
\newtheorem{proposition}[lemma]{Proposition}
\newtheorem{theorem}[lemma]{Theorem}
\newtheorem{theorem-definition}[lemma]{Theorem-Definition}
\newtheorem{corollary}[lemma]{Corollary}
\newtheorem{corollary-definition}[lemma]{Corollary-Definition}

\theoremstyle{definition}
\newtheorem{definition}[lemma]{Definition}

\newtheorem{remark}[lemma]{Remark}

\newtheorem{examplenew}[lemma]{Example}

\newtheorem{notation}[lemma]{Notation}

\title{Finitary affine oriented matroids}

\author{Emanuele Delucchi}
\address{E. Delucchi, University of Applied Arts and Sciences of Southern Switzerland (SUPSI), Lugano, Switzerland}
\email{emanuele.delucchi@supsi.ch}
\author{Kolja Knauer}
\address{K. Knauer, Aix Marseille Univ, Universit\'e de Toulon, CNRS, LIS, Marseille, France \\ Departament de Matem\`atiques i Inform\`atica,
Universitat de Barcelona, Spain}
\email{}

\begin{document}

\maketitle

\begin{abstract}
We initiate the axiomatic study of affine oriented matroids (AOMs) on arbitrary ground sets, obtaining fundamental notions such as minors, reorientations and a natural embedding into the frame work of Complexes of Oriented Matroids. The restriction to the finitary case (FAOMs) allows us to study tope graphs and covector posets, as well as to view FAOMs as oriented finitary semimatroids. We show shellability of FAOMs and single out  the FAOMs that are affinely homeomorphic to $\mathbb{R}^n$.
Finally, we study group actions on AOMs, whose quotients in the case of FAOMs are a stepping stone towards a general theory of affine and toric pseudoarrangements. Our results include applications of the multiplicity Tutte polynomial of group actions of semimatroids, generalizing enumerative properties
of toric arrangements to a combinatorially defined class of arrangements of submanifolds. This answers partially a question by Ehrenborg and Readdy.
\end{abstract}

\setcounter{tocdepth}{1}

\tableofcontents

\noindent\textbf{Keywords:} Affine Oriented Matroids, Semimatroids, finitary and toric arrangements

\noindent\textbf{MSC code:} 52C40, 52C30, 06A12, 57S12

\section{Introduction}

\subsection{Subject, results and structure of the paper}
In this paper we establish the natural generalization of finite affine oriented matroids to arbitrary ground sets and derive several results about their axiomatics, topology and geometry. Our motivation is twofold: on the one hand we aim at advancing the structural theory of oriented matroids and arithmetic matroids, on the other hand we have in mind applications to linear and toric arrangements, which we discuss below in \S \ref{2mot}, as well as to general manifold arrangements (see Remark \ref{rem:EhRe}). Let us here summarize our main results.

\begin{itemize}
\item We present axiom systems for covectors of {\bf Affine Oriented Matroids} ({AOM}s) over arbitrary ground sets (Section \ref{sec:AOMs}). These support canonical operations such as reorientation and taking minors (\S\ref{ssec:minors}). 
In particular our axiomatization, derived from results of Baum and Zhu~\cite{Bau-16}, allows us to see AOMs as part of the theory of Complexes of Oriented Matroids (COMs) -- a recent common generalization of oriented matroids and lopsided sets~\cite{Ban-18}. However, the extension to arbitrary non-finite ground sets is novel and many of our results extend to general COMs, shedding a first light into this direction.
Furthermore, we introduce a natural and geometrically meaningful notion of parallelism that defines an equivalence relation on the elements of the AOM (\S\ref{ssec:par}). It is crucial for the development of the subsequent results.

\item In order to obtain a theory that more closely encapsulates some of the geometric features of finitary affine hyperplane arrangements, in Section \ref{sec:FAOM} we state axioms for {\bf Finitary Affine Oriented Matroids} ({FAOM}s). These are AOMs with some local cardinality restrictions. A main theoretical feature of this restricted setting is that FAOMs are ``orientations of finitary semimatroids'', i.e.: the zero sets of covectors of an FAOM constitute the geometric semilattice of flats of a finitary semimatroid 
(e.g., in the sense of \cite{DeluRiedel}, generalizing the finite notion developed by Wachs and Walker \cite{WW} and by Ardila \cite{Ardila} and Kawahara \cite{Kawahara}).
We carry out a basic study of {\bf tope graphs and covector posets} of FAOMs (\S\ref{subsec:topes}) and then we focus on topological properties. We prove that order complexes of covector posets of FAOMs are shellable (\S\ref{ssec:shellings}) and explicitly describe their homeomorphism type (\S\ref{ssec:topcov}). Moreover, we derive some order-theoretic properties of the geometric parallelism relation in FAOMs (\S\ref{ssec:ordertype}). This allows us to single out a special class of FAOMs whose covector poset is affinely homeomorphic to Euclidean space $\mathbb R^n$ (see Section \ref{sec:frames}). 

\item  In Section \ref{sec:gas} we take FAOMs as a stepping stone in order to extend the theory of arrangements of pseudospheres (and -planes) beyond the Euclidean setting, towards {\bf pseudoarrangements in the torus}. See \S\ref{2mot} below for some motivating context from arrangements theory. In order to accomplish this we study {\bf group actions} on AOMs and, in particular, a class of group actions for which the quotient of the covector poset is homeomorphic to a torus. In this torus, the quotients of all one-element contractions of the given FAOM determine an arrangement of tamely embedded tori. Notice that such ``toric pseudoarrangements'' are strictly more general than toric arrangements defined by level sets of characters (which we call ``stretchable" extrapolating the Euclidean terminology), see Figure \ref{fig2}.  In any case, stretchable or not, the faces of the corresponding dissection of the torus are enumerated by the Tutte polynomial associated in \cite{DeluRiedel} to the induced group action on the underlying semimatroid, generalizing enumerative results by Moci on arithmetic Tutte polynomials associated to toric arrangements, see \cite{MoDa} and Remark \ref{rem:ArMat}. We also mention that Pagaria in \cite{PagOM} put forward a notion of orientable arithmetic matroid, asking for an interpretation in terms of pseudoarrangements on the torus. See \S\ref{tomsp} for how our work contributes to this line of research.

\begin{remark}\label{rem:EhRe} Ehrenborg and Readdy ask in \cite{EhReManifold} for a natural class of submanifold arrangements where an ``arithmetic" Tutte polynomial can be meaningfully defined. Our first answer to this question is the class of arrangements in Euclidean space or in tori obtained from (possibly trivial) ``sliding" group actions on FAOMs (Definition \ref{def:sliding}). Theorem~\ref{MT:toric} shows that the Tutte polynomial of the associated group action on the underlying semimatroid provides the desired topological enumeration, together with the algebraic-combinatorial properties studied in \cite{DeluRiedel}. In the case of ``standard'' toric arrangements we recover the arithmetic Tutte polynomial.
\end{remark}
\end{itemize}

The multi-pronged nature of our foundational work shows that infinite affine oriented matroids are at the crossroads of several topics in structural, algebraic and topological combinatorics. Thus AOMs offer new tools for existing {\bf open problems}, and create some new ones in their own right: we outline some of these connections and research directions in Section \ref{sec:OQ}.

\medskip

In order to make the paper reasonably self-contained we include an Appendix where we briefly summarize the topological and algebraic-combinatorial tools we need.

\subsection{Two motivating examples} \label{2mot}

We outline some of the motivation for our  work, and explain our contribution in these contexts.

\subsubsection{Arrangements in Euclidean space}\label{intro:hyperplanes}
Let $\AA:=\{H_e\}_{e\in E}$ be an {\em arrangement of hyperplanes}, i.e., a family of codimension $1$ affine subspaces of the Euclidean space $\mathbb R^d$. We call such an arrangement ``oriented'' if  for every $e\in E$ we are given a labeling by $H_e^+$ and $H_e^-$ of the two connected components of $\mathbb R^d\setminus H_e$. 

\begin{definition}
	Given an oriented arrangement $\AA:=\{H_e\}_{e\in E}$ of affine hyperplanes in $\mathbb R^d$ define, for every $x\in \mathbb R^d$ a sign vector  $\Sigma_x\in \{+,0,-\}^E$ as follows.
	$$
	\Sigma_x(e):=
	\left\{
	\begin{array}{ll}
	+ &\textrm{ if } x\in H_e^+ \\
		0 &\textrm{ if } x\in H_e \\
			- &\textrm{ if } x\in H_e^-
	\end{array}
	\right.
	$$
	Let then $\covset(\mathscr A):=\{\Sigma_x \mid x\in \mathbb R^d\}$.
\end{definition}

\def\dimensione{\footnotesize}

\begin{figure}[h!]
\includegraphics[scale=0.8]{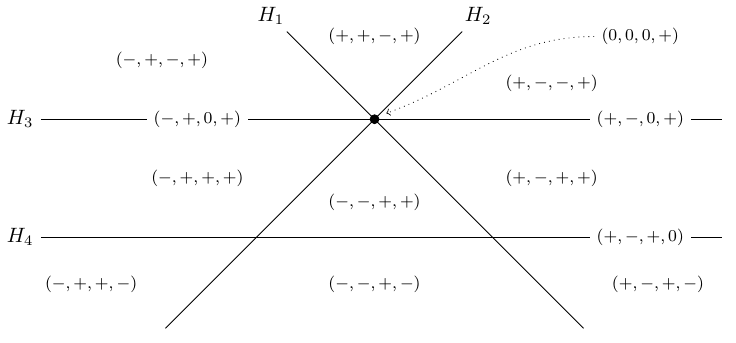}
\caption{An arrangement of hyperplanes in $\mathbb R^2$ with some cells labeled by the respective sign vector. }\label{fig:arrangement}
\end{figure}

The covector axioms of {\em oriented matroids} abstract some of the properties of  $\covset(\AA)$ in the case where $\AA$ is finite and $\cap \AA\neq\emptyset$, see \S\ref{FOMs}. Conversely, while not every oriented matroid arises from such an arrangement of hyperplanes, the powerful ``Topological Representation Theorem" of Folkman and Lawrence asserts that the system of covectors of every oriented matroid is the set of sign vectors determined by some  arrangement of oriented pseudospheres in the sphere (obtained as the boundary of the order complex of the covector poset, see~\cite[Chapter 5]{bjvestwhzi-93}).

If $\AA$ is finite, but $\cap\AA$ is not necessarily non-empty, then $\covset(\AA)$ is the set of covectors of a {\em finite affine oriented matroid}. Finite affine oriented matroids can be defined either intrinsically or as subsets of covector sets of oriented matroids, see~\cite{Kar-92,Bau-16}. The latter point of view allows us to interpret every finite affine oriented matroid as an arrangement of pseudoplanes in Euclidean space, again via the order complex of its covector poset, but it is an open problem to characterize which arrangements arise from finite affine oriented matroids, see~\cite{FoZa} and \S\ref{oq_pseudo} .
\begin{figure}[h]\label{fig2}
\centering
\includegraphics[scale=0.8]{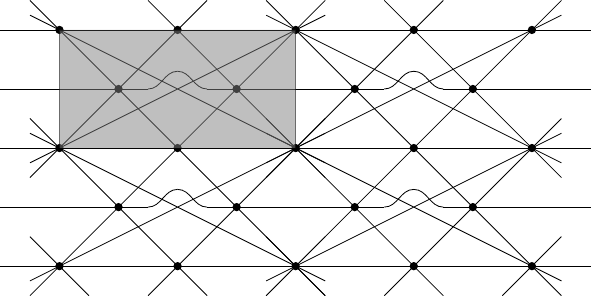}
\caption{A non-stretchable line arrangement with an action of $\mathbb Z^2$ defined by letting a lattice basis act as translations by the two sides of the shaded rectangle. (The picture should be thought of as being repeated in vertical and horizontal direction). Any orientation of it gives rise to a FAOM. 
}
\label{fig:pseudoperiodic}
\end{figure}

More generally, if $\AA$ is only assumed to be {\em finitary}, meaning that every $x\in \mathbb R^d$ has a neighborhood meeting finitely many $H_e$, then every  element of $\covset(\AA)$ indexes an open cell in $\mathbb R^d$. These open cells are the relative interiors of the  {\em faces} of the polyhedral subdivision of $\mathbb R^d$ induced by $\AA$. The faces of a polyhedral complex are naturally ordered by inclusion, and this partial order corresponds to the (abstract) natural order among sign vectors (see Definition~\ref{def:ursprung}). 

\begin{itemize}
\item Our ``Finitary Affine Oriented Matroids" axiomatize properties of the polyhedral stratification of Euclidean space induced by finitary hyperplane arrangements. Not every FAOM is realizable as $\covset(\AA)$ for a finitary arrangement. Still, some familiar geometric and topological features generalize nicely to the non-realizable case as well.
\item We view our topological representation of FAOMs as a step towards the currently open problem of a topological characterization of affine pseudoarrangements (see \S\ref{oq_pseudo}).
\end{itemize}

\subsubsection{Toric arrangements} \label{intro:toric} 
Let now $\AA$ be a finite family of level sets of characters of the compact torus $T=(S^1)^d$. Such {\em toric arrangements} have been in the focus of recent research originally motivated by work of De Concini, Procesi and Vergne on partition functions and splines, see \cite{DeCP}. A toric arrangement defines a polyhedral CW-structure $\rcw(\AA)$ on the torus. The face category of this cell complex is central in the study of the topology of the associated arrangement in the complex torus~\cite[\S2]{DaDeJEMS} and of arrangements in products of elliptic curves~\cite{DePag}. It can be regarded as the ``toric" counterpart of the poset of faces of a linear arrangement\footnote{The broadening from face posets to face categories is necessary since the CW-complex $\rcw(\AA)$ is not necessarily regular, see Appendix~\ref{sec:ACGA}.}. 

Notice that, by passing to the universal cover of the torus, a toric arrangement can be seen as a quotient of an infinite, periodic arrangement of hyperplanes by the action of the deck transformation group. 

\begin{figure}[h]
\centering
\includegraphics[scale=0.5]{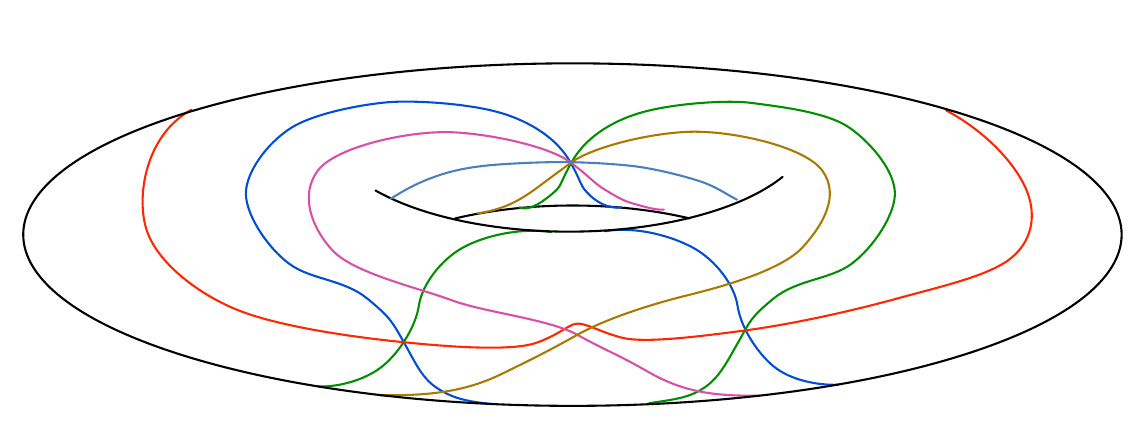}
\caption{
The quotient of the poset of covectors of the pseudoarrangement in Figure~\ref{fig2} is the face category of a (pseudo)arrangement in the $2$-dimensional torus (e.g., obtained by identifying opposite sides of the shaded rectangle), whose cells are counted by the Tutte polynomial of the group action on the underlying semimatroid.
}\label{fig:toricpseudo}
\end{figure}

The current impulse towards the combinatorial study of toric arrangements already led to substantial algebraic-combinatorial developments such as arithmetic Tutte polynomials and arithmetic matroids~\cite{MoDa}. However, the only available results about the structure of face categories to date are an explicit description in the case of toric Weyl arrangements by means of ``labelled necklaces"~\cite{AgPe}. 

\begin{itemize}
	\item We obtain an abstract characterization of the face category of toric arrangements as the quotient of the poset of covectors of an affine, infinite oriented matroid by a suitable class of group actions. This can be seen as an ``oriented" version of the theory of group actions on semimatroids~\cite{DeluRiedel} designed to describe toric arrangements on an ``unoriented", matroidal level.
	\item Accordingly, we obtain a notion of pseudoarrangements in the torus whose topology and geometry is amenable to treatment via the existing combinatorial toolkit. We leave the relationship to Pagaria's orientable arithmetic matroids to future research, see \S\ref{tomsp}.
\end{itemize}

\subsection{Acknowledgements} Emanuele Delucchi has been partially supported by the Swiss National Science Foundation professorship grant  PP00P2\_150552/1. Kolja Knauer was supported by the Spanish State Research Agency through grants RYC-2017-22701, PID2022-137283NB-C22 and the Severo Ochoa and María de Maeztu Program for Centers and Units of Excellence in R\&D (CEX2020-001084-M) by the French \emph{Agence nationale de la recherche} through project ANR-17-CE40-0015.

\section{Affine oriented matroids (AOM)}\label{sec:AOMs}
In the non-finite context it is essential to view AOMs with an intrinsic axiomatization instead of as halfspaces of oriented matroids as described in~\S\ref{intro:hyperplanes}. The goal of this section is to present the covector axiomatization of finite AOMs due to Karlander~\cite{Kar-92} whose proof was corrected recently by Baum and Zhu~\cite{Bau-16}. We state this axiomatization for arbitrary cardinalities and bring it into a simplified form, which puts AOMs into the context of (complexes) of oriented matroids, (C)OMs~\cite{Ban-18}. Moreover, we show that notions of minors and parallelism generalize straightforwardly to the infinite setting.
Indeed, for the purpose of the  present section no assumption on the ground set $E$ has to be made.

\begin{definition}\label{def:ursprung}
A {\em sign vector} (on a set $E$) is an element of $\signs^E$. A {\em system of sign vectors} is any subset $\covset \subseteq \signs^E$. We say system of sign vectors ``on $E$'', and write $(E,\covset)$, if specification is needed. Every system of sign vectors carries a natural partial order:
$$
X\leq Y \quad \textrm{ if and only if } X(e) \leq Y(e) \textrm{ for all }e\in E
$$
where we define $0 < +$, $0 < -$, $+$ and $-$ incomparable. 
The poset $(\covset,\leq)$ will be denoted by  $\covpos(\covset)$.
\end{definition}

We introduce some further standard notions, see e.g.~\cite{bjvestwhzi-93}.
The \emph{support} of a sign vector $X$ is  $\underline{X}:=\{e \in E\mid X(e)\neq 0\}$. The {\em zero set} of a sign vector $X$ is the complement of its support, i.e., $\ze{X}:=\{e\in E \mid X(e)=0\}$. Moreover, the \emph{separator} of two sign vectors $X,Y$ is $S(X,Y):=\{e\in \underline{X}\cap\underline{Y}\mid X(e)\neq Y(e)\}$ and the \emph{composition} of $X$ and $Y$ is the sign vector given by 
$$
X\circ Y(e):=\left\{\begin{array}{ll}
X(e) & \textrm{ if } e \in \underline{X}\\
Y(e) &\textrm{ otherwise.}
\end{array}\right.\textrm{ for all }e\in E.
$$

We now recall some notations that we take from the specific treatment of the affine case given in \cite{Bau-16}. Let $X,Y$ be any two sign vectors on $E$, $e\in E$, and $\covset$ a given system of sign vectors on $E$. Define 
$$I^=_e(X,Y;\covset):=\{Z\in\covset \mid Z(e) = 0, \forall f \notin S(X, Y ) : Z(f) = X(f) \}$$ and $$I_e(X,Y;\covset):=\{Z\in\covset \mid Z(e) = 0, \forall f \notin S(X, Y ) : Z(f) = X(f)\circ Y(f) \}.$$ Moreover, set 
\begin{displaymath}
I^=(X,Y;\covset):=\bigcup_{e\in S(X,Y)}I^=_e(X,Y;\covset)
\quad
\textrm{ and }
\quad
I(X,Y;\covset):=\bigcup_{e\in S(X,Y)}I_e(X,Y;\covset).
\end{displaymath} 
We will omit reference to $\covset$, writing simply $I(X,Y)$, $I^=(X,Y)$, etc., if no confusion can arise. The letter $I$ is established for the above sets, because these sets can be seen as intervals of sorts, see Figure~\ref{fig:annotated}.
Furthermore, write 
$$\mr{Asym}(\covset):=\{X\in\covset\mid-X\notin\covset\}$$ and let $X\oplus Y$ be the sign vector obtained from ``adding'' the signs of the sum of $X$ and $Y$ seen as integer vectors, i.e., 

$$
X\oplus Y(e):=\left\{\begin{array}{ll}
0 & \textrm{ if } e \in S(X,Y)\\
X\circ Y(e) &\textrm{ otherwise.}
\end{array}\right. \textrm{ for all }e\in E.
$$

We can now set
\begin{equation*}\mathcal{P}^=_{\mr{asym}}(\covset):=\{X\oplus (-Y)\mid X,Y\in \mr{Asym}(\covset),\, \underline{X}=\underline{Y},\, I^=(X,-Y;\covset)=I^=(-X,Y;\covset)=\emptyset\},
\end{equation*}
\begin{equation*}
\mathcal{P}(\covset):=\{X\oplus (-Y)\mid X,Y\in\covset,\, I(X,-Y;\covset)=I(-X,Y;\covset)=\emptyset\}.
\end{equation*}

\medskip

We are now able to state the first definition.

\begin{definition}[AOM, following~\cite{Kar-92,Bau-16}]\label{def:orig}
 A pair $(E,\covset)$ is the system of covectors of an affine oriented matroid if and only if
 \begin{itemize}
 \item[(C)] $\covset\circ\covset\subseteq\covset$, \hfill (composition)
  \item[(FS)] $\covset\circ(- \covset)\subseteq\covset$, \hfill (face symmetry)
  \item[(SE$^=$)] $X,Y\in\covset,\underline{X}=\underline{Y}\implies \forall e\in S(X,Y):I^=_e(X,Y)\neq\emptyset$,  \strut \\ \strut \hfill (strong elimination equal support)
  \item[(P$^=_{\mr{asym}}$)] $\mathcal{P}^=_{\mr{asym}}(\covset)\circ\covset\subseteq\covset$.   \hfill (peripheral composition equal support)
 \end{itemize}
 
 Then, the associated $\covpos(\covset)$ is called the {\em poset of covectors} of the given AOM.

\end{definition}

\begin{remark}\label{finiteok}
By \cite[Theorem 1.2]{Bau-16}, finite AOMs (i.e., AOMs $(E,\covset)$ with $\vert E \vert<\infty$) are exactly {affine oriented matroids} in the sense, e.g., of \cite{bjvestwhzi-93}.
\end{remark}

We propose the following simpler and (seemingly) stronger axiomatization.

\begin{proposition}[AOM]\label{prop:AOMCOM}
 A pair $(E,\covset)$ is the system of covectors of an affine oriented matroid if and only if
 \begin{itemize}
  \item[(FS)] $\covset\circ(- \covset)\subseteq\covset$,
  \item[(SE)] $X,Y\in\covset\implies \forall e\in S(X,Y):I_e(X,Y)\neq\emptyset$, \hfill (strong elimination)
  \item[(P)] $\mathcal{P}(\covset)\circ\covset\subseteq\covset$. \hfill (peripheral elimination)
 \end{itemize}
\end{proposition}
\begin{proof}

 First we collect some straightforward observations. For every $X,Y \in\{\pm, 0\}^E$, we have
 \begin{enumerate}
  \item\label{i:separator}  $S(X,Y)=S(X\circ Y,Y\circ X)$;
  \item\label{i:support}  $\underline{X\circ Y} = \underline{Y\circ X}$;
  \item\label{i:comp} if $\underline{X} = \underline{Y}$, then $X=X\circ Y$;
  \item\label{i:sum} $X\oplus Y=(X\circ Y)\oplus (Y\circ X)$.
 \end{enumerate}

We now move to prove the stated equivalence, in several steps.

\begin{itemize}
\item[--](FS)$\Rightarrow$(C), hence (C) can be removed from Definition~\ref{def:orig}.
\item[] {\em Proof.} It is enough to notice that $X\circ Y= (X\circ -X)\circ Y= X\circ (-X\circ Y)= X\circ -(X\circ -Y)$.
 
 \item[--] (SE)$\Leftrightarrow$(SE$^=$).
\item[] {\em Proof.} Clearly, (SE) implies (SE$^=$). Conversely, we get that with (C) the axiom (SE$^=$) implies (SE). Indeed, for $X,Y \in  \covset$ by~\ref{i:separator} we have $I_e(X,Y)=I_e(X\circ Y,Y\circ X)$ and both sets are defined for the same set of elements $e$. 
 By~\eqref{i:support} and~\eqref{i:comp} we have $X\circ Y(f)=(X\circ Y)\circ(Y\circ X)(f)$, which gives $I_e(X\circ Y,Y\circ X)=I^=_e(X\circ Y,Y\circ X)$. Thus, $\forall e\in S(X\circ Y,Y\circ X):I^=_e(X\circ Y,Y\circ X)\neq\emptyset$ implies $\forall e\in S(X,Y):I_e(X,Y)\neq\emptyset$.

\item[--] (P)$\Rightarrow$(P$^=_{\mr{asym}}$).
\item[] {\em Proof.} Since by~\eqref{i:comp} $I_e(X,Y)=I^=_e(X,Y)$ for sign vectors of equal support, we can write $\mathcal{P}^=_{\mr{asym}}(\covset)$ as $$\{X\oplus (-Y)\mid X,Y\in \mr{Asym}(\covset), \underline{X}=\underline{Y}, I(X,-Y)=I(-X,Y)=\emptyset\},$$ which gives $\mathcal{P}^=_{\mr{asym}}(\covset)\subseteq \mathcal{P}(\covset)$. 
 \item[--] Under (SE),  (P$^=_{\mr{asym}}$) $\Rightarrow$ (P).
\item[] {\em Proof.} First observe that (P$^=_{\mr{asym}}$) $\implies \{X\oplus (-Y)\mid X,Y\in \covset, \underline{X}=\underline{Y}, I(X,-Y)=I(-X,Y)=\emptyset\}\circ\covset\subseteq\covset$, i.e., we can drop the asymmetry condition.
 Indeed, suppose $X,-X, Y \in \covset$. Now, by (SE) $I(-X,Y) \neq \emptyset$ except if $S(-X,Y)=\emptyset$. But if $\underline{X}=\underline{Y}$ then $S(-X,Y)=\emptyset$ implies $Y=-X$, and $X\oplus (-Y)=X\oplus X=X\in \covset$. Thus, only trivially fulfilled conditions are added. The symmetric argument works for the case $X,Y,-Y \in \covset$.
 
 We proceed by showing that $$\{X\oplus (-Y)\mid X,Y\in \covset, \underline{X}=\underline{Y}, I(X,-Y)=I(-X,Y)=\emptyset\}\supseteq \mathcal{P}(\covset).$$ 
 Let $X\oplus (-Y)\in\mathcal{P}(\covset)$ and consider the vectors $X\circ (-Y)$ and $Y\circ (-X)$. We can compute $-(Y\circ (-X))=-Y\circ X$ and 
$$\underline{X\circ (-Y)}\stackrel{~\eqref{i:support}} = \underline{-Y\circ X}= \underline{-(Y\circ (-X))}=\underline{Y\circ (-X)},$$ 
where the last equality follows from $\underline{Z}=\underline{-Z}$. By~\eqref{i:separator} we get 
$$S(X,-Y)=S(X\circ (-Y),-Y\circ X)\text{ and }S(-X,Y)=S(-X\circ Y,Y\circ (-X)).$$
Which implies the equality of the elimination sets
$$I(X,-Y)=I(X\circ (-Y),-Y\circ X)
\text{ and }
I(-X,Y)=I(-X\circ Y,Y\circ (-X)).$$
Finally,~\eqref{i:sum} gives $X\oplus (-Y)=(X\circ (-Y))\oplus (-Y\circ X)$. Together we obtain that $X\oplus (-Y)$ is contained in the set on the left-hand side. This concludes the proof.
\end{itemize}

 \end{proof}

\begin{examplenew}
    In Figure~\ref{fig:annotated} we illustrate the operations involved in the covector axioms of AOMs in Proposition~\ref{prop:AOMCOM} on the example of the realizable AOM from Figure~\ref{fig:arrangement}. First, we choose two points $X,Y$ identified with the corresponding sign vectors and add the auxiliary (dashed) line $\ell$ defined by the two points. The fact that a point on $\ell$ close to $X$ towards $Y$ as well as away from $Y$ is also in the arrangement, respectively, corresponds to the covectors $X\circ Y\in\covset$ and $X\circ -Y\in\covset$, respectively. The intersection points of $\ell$ with the hyperplanes $H_2$ and $H_3$, respectively, correspond to the elements of $I(X,Y;\covset)$.

    Now, $Z,W$ are sign vectors corresponding to two maximal cells bounded by the parallel hyperplanes $H_3,H_4$. The sign vector $Z\oplus -W$ can be seen as the intersection point of the hyperplanes at infinity, i.e., with the auxiliary equator. The fact that $(Z\oplus -W)\circ \covset\subseteq\covset$ can be interpreted by saying that all points "close to $(Z\oplus -W)$" towards existing cells of the arrangement also form part of the arrangement. Hence, the name \emph{peripheral} composition.
\end{examplenew}

\begin{figure}
\centering
\includegraphics[scale=1]{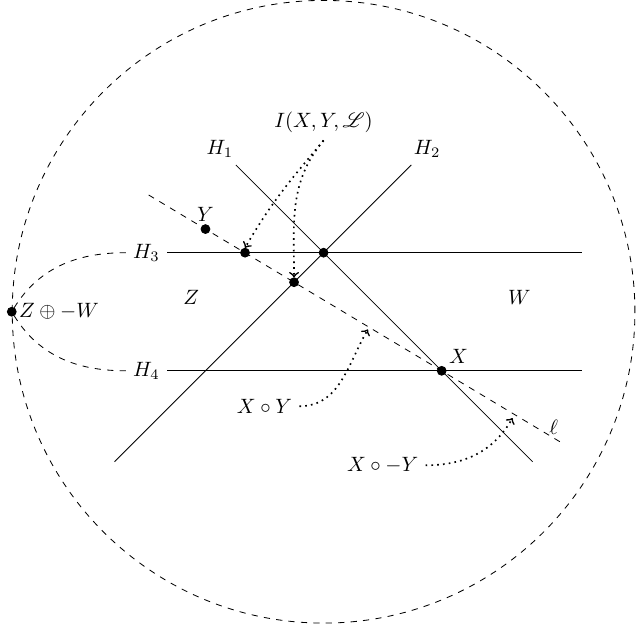}
\caption{Covector axioms illustrated in the example from Figure~\ref{fig:arrangement}.}\label{fig:annotated}
\end{figure}

Let us compare the axiomatization from Proposition~\ref{prop:AOMCOM} with axiomatizations for COMs and OMs, even if those are usually given for finite ground sets. Following~\cite{Ban-18} a COM is a system of sign vectors satisfying (FS) and (SE). Thus, from Proposition~\ref{prop:AOMCOM} one immediately deduces:

\begin{corollary}
 Every AOM is a COM.
\end{corollary}

The axiomatization of OMs given in~\cite{Ban-18} (see also Definition~\ref{def:OM}) is as a system of sign vectors satisfying (FS) and (SE) and 
\begin{itemize}
 \item[(0)] the all-zeroes vector $\mathbf{0}$ is in $\covset$. \hfill (zero-vector)
\end{itemize}

\begin{corollary}
 Every OM is an AOM.
\end{corollary}
\begin{proof}
 Note that (0) together with (FS) implies $X\in\covset\implies -X\in\covset$. This, together with (SE), yields that $I(X,-Y;\covset)=S(X,-Y)$ for all $X,Y\in\covset$. Hence, $\mathcal{P}(\covset)\subseteq\covset\circ -\covset\subseteq \covset$. Since (FS) implies $\covset\circ \covset\subseteq \covset$, (P) is fulfilled trivially . 
\end{proof}

We now proceed to define reorientations.

\begin{definition}\label{def:reorientation}
Let $\covset$ be a family of sign vectors. A \emph{reorientation} of $\covset$ is any set
	$$
	\covset^{(\tau)}:=\{ \tau \cdot X \mid X\in \covset \}
	$$ for a given $\tau\in \{+1,-1\}^E$, where multiplication is intended componentwise, i.e.,  $(\tau\cdot X)(e):=\tau(e)X(e)$.
	\end{definition}

\begin{remark}
It is straightforward to see that  every reorientation of an AOM is an AOM.
\end{remark}

\subsection{Minors}\label{ssec:minors}

The notion of minors is crucial in the study of OMs and COMs. Let us define the necessary ingredients here.
Let $(E,\covset)$ be any system of sign vectors. 

\begin{definition}\label{def:minors}
For any $A\subseteq E$ define the \emph{contraction} of $A$ in $\covset$ as
$$
\covset/A := \{X_{\vert E\setminus A } \mid X\in \covset, \,\, X(A)=\{0\}\},
$$
(notice that this is nonempty if and only if $A\in \central(\covset)$), and the \emph{deletion} of $A$ from $\covset$ as
$$
\covset\setminus A := \{X_{\vert E\setminus A } \mid X\in \covset\}.
$$
Moreover, we call \emph{restriction to $A$} the set $\covset [A]:=\covset\setminus (E\setminus A)$.

A system of sign vectors $(E',\covset')$ is a \emph{minor} of another system of sign vectors $(E,\covset)$ if there are disjoint sets $A,B\subseteq E$ such that $(E',\covset')=(E\setminus A\setminus B,\covset\setminus A/B)$.
\end{definition}

As an example consider the AOM in Figure~\ref{fig:arrangement}. Contracting the element $3$ yields $\{(-,+,+), (0,0,+), (+,-,+)\}$. Thus can be seen as considering only the cells on $H_3$ and removing the third coordinate. Deleting $3$, corresponds to removing $H_3$ from the arrangement. 

\begin{remark} \label{rem:injection_contraction}
Notice that there is a canonical order preserving injection
$$
\iota_A: \covset/A \hookrightarrow \covset,\quad \quad
\iota_A(X)(e):=\left\{
\begin{array}{ll}
X(e) & e\not \in A,\\
0 & e\in A.
\end{array}
\right.
$$
\end{remark}
\begin{remark}
One can show following~\cite{Ban-18}, that for countable families $\{A_i\}_{i\geq 1}$ and $\{B_i\}_{i\geq 1}$ of sets we have 
 $$\covset\setminus A_1/ B_1\setminus A_2/ B_2\ldots = 
 \covset
 \setminus \bigcup_{i\geq1}A_i 
 \left/ \bigcup_{i\geq1}B_i\right.,$$
i.e., the operations of contraction and deletion commute. 
We do not investigate further the case of uncountable families of sets.
\end{remark}

\begin{lemma}\label{lem:pdel} Let $(E,\covset)$ satisfy (SE) and let $A\subseteq E$. Then,
$$\mathcal P(\covset \setminus A)\subseteq \mathcal{P}(\covset)\setminus A.$$
\end{lemma}
\begin{proof}
Let $X\oplus -Y\in \mathcal{P}(\covset \setminus A)$, with $X,Y\in \covset\setminus A$ and $I_{\covset \setminus A}(X,-Y)=I_{\covset \setminus A}(X,-Y)=\emptyset$. 

We prove that there are $\hat{X}, \hat{Y}\in\covset$ such that $\hat{X}\setminus A=X$, $ \hat{Y}\setminus A=Y$ and $\hat{X}\oplus -\hat{Y}\in \mathcal{P}(\covset)$. Let $\hat{X}, \hat{Y}\in\covset$ such that $\hat{X}\setminus A=X, \hat{Y}\setminus A=Y$ and suppose that $\hat{X}\oplus -\hat{Y}\notin \mathcal{P}(\covset)$. Thus, without loss of generality there is a $Z\in I_{\covset}(\hat{X},-\hat{Y})$. 

Note that for $f\in S(\hat{X},-\hat{Y})\setminus A$ we have $Z(f)\neq 0$, since otherwise $Z\setminus A\in I_{\covset \setminus A}(X,-Y)$ with respect to $f$. Hence, $\ze{Z}=A\cup(\ze{X}\cap \ze{Y})$. 

Furthermore note that $Z(f)=\hat{X}_f$, since otherwise if $Z(f)=-\hat{X}(f)$ we can apply strong elimination to $Z$ and $\hat{X}$ with respect to $f$ and obtain a $\widetilde{Z}\in I_{\covset}(\hat{X},-\hat{Y})$ with $\widetilde{Z}(f)=0$. This contradicts the above.

We conclude that for all $g\in E$ we have $Z(g)=\begin{cases}
                       0 & \textrm{if }  g\in A\\
                       \hat{X}\circ -\hat{Y} &  \mathrm{otherwise.}
                      \end{cases}$ 
                      
Next, we show that $Z\oplus -\hat{Y}\in \mathcal{P}(\covset)$. 
Clearly, we have $Z, -\hat{Y}\in\covset$. Furthermore, we have that $I(Z,-\hat{Y})\subseteq \bigcup_{f\notin A}I_f(\hat{X},-\hat{Y})$ and $I(-Z,\hat{Y})\subseteq \bigcup_{f\notin A}I_f(-\hat{X},\hat{Y})$.
However, as argued above, $\bigcup_{f\notin A}I_f(\hat{X},-\hat{Y})=\bigcup_{f\notin A}I_f(-\hat{X},\hat{Y})=\emptyset$. This concludes the proof of this last claim.

Since we now know exactly how $Z$ arises from $\hat{X}$ and $\hat{Y} $ it is straightforward to check that $(Z\oplus -\hat{Y})\setminus A=X\oplus -Y$. This concludes the proof.
\end{proof}

\begin{lemma}\label{lem:pcont} Let $(E,\covset)$ any set of sign vectors and $A\subseteq E$. Then, $\mathcal P(\covset / A) \subseteq \mathcal P ( \covset) / A$.
\end{lemma}
\begin{proof}
Let $X\oplus -Y\in P(\covset / A)$, with $X,Y\in \covset/ A$ and $I_{\covset / A}(X,-Y)=I_{\covset / A}(X,-Y)=\emptyset$. Thus, that there are $\hat{X},\hat{Y}\in\covset$ with $A\subseteq \ze{\hat{X}}\cap \ze{\hat{Y}}$ that otherwise coincide with $X$ and $Y$, respectively. In particular, $I_{\covset / A}(X,-Y)\cong I_{\covset}(\hat{X},-\hat{Y})$, since all covectors in $I_{\covset}(\hat{X},-\hat{Y})$ are $0$ on $A$. The same holds for $I_{\covset}(\hat{X},-\hat{Y})$ and we have $I_{\covset}(\hat{X},-\hat{Y})=I_{\covset}(\hat{X},-\hat{Y})=\emptyset$. 

This means that $X\oplus Y =(\hat{X}\oplus \hat{Y})/A\in \mathcal P ( \covset) / A$.
\end{proof}

We take the following from~\cite{Ban-18} and we review its proof in order to ensure that it does not rely on  finiteness assumptions.
\begin{lemma}\label{lem:minorclosed}
COMs are closed under minors, i.e., the properties {\normalfont (FS)} and {\normalfont (SE)} are closed under deletion and contraction.
\end{lemma}

\begin{proof}
We first prove the statement for deletion. To see (FS) let $X\setminus A,Y\setminus A\in \covset\setminus A$. 
Then $X\circ (- Y)\in \covset$ and $(X\circ (- Y))\setminus A= X\setminus A\circ (- Y\setminus A)\in\covset\setminus A$. To check (SE) let $X\setminus A,Y\setminus A\in \covset\setminus A$ and $e$ an element separating $X\setminus A$ and $Y\setminus A$.  
Then there is $Z\in\covset$ with $Z(e)=0$ and $Z(f)=X\circ Y(f)$ for all $f\in E\setminus S(X,Y)$. Clearly,
$Z\setminus A\in \covset\setminus A$ satisfies (SE) with respect to $X\setminus A,Y\setminus A$.

Now, we prove  the statement for contraction. Let $X\setminus A,Y\setminus A\in \covset/A$, i.e.,
$\underline{X}\cap A=\underline{Y}\cap A=\varnothing$. Hence $\underline{X\circ (- Y)}\cap A=\varnothing$ and
therefore $X\setminus A\circ (- Y\setminus A)\in \covset/A$, proving (FS). Towards proving (SE), let $X\setminus A,Y\setminus A\in \covset/ A$ and $e$ an element separating $X\setminus A$
and $Y\setminus A$. 
Then there is $Z\in\covset$ with $Z(e)=0$ and $Z(f)=X\circ Y(f)$ for all $f\in E\setminus S(X,Y)$.
In particular, if $X(f)=Y(f)=0$, then $Z(f)=0$. Therefore, $Z\setminus A\in \covset/A$ and it satisfies (SE).
\end{proof}

\begin{theorem}\label{thm:minorclosed}
 AOMs are closed under minors.
\end{theorem}
\begin{proof}
 Let $A,B\subseteq E$ be disjoint  and let $(E, \covset)$ be an AOM. Consider the minor $(E\setminus A\setminus B, \covset\setminus A/B)$. We prove that this is an AOM. For axioms (FS) and (SE) this follows directly from Lemma~\ref{lem:minorclosed}. Since AOMs satisfy (SE) we can apply Lemma~\ref{lem:pdel} and compute $\mathcal P(\covset \setminus A)\circ \covset \setminus A\subseteq \mathcal{P}(\covset)\setminus A\circ \covset \setminus A=(\mathcal{P}(\covset)\circ \covset)\setminus A\subseteq \covset\setminus A.$ Similarly, using Lemma~\ref{lem:pcont} we get $\mathcal P(\covset / A)\circ \covset / A\subseteq \mathcal{P}(\covset)/ A\circ \covset / A=(\mathcal{P}(\covset)\circ \covset)/ A\subseteq \covset/ A.$
\end{proof}

\begin{examplenew}
    In the realizable setting deletion and contraction of elements correspond to removing a hyperplane or restricting to a hyperplane. E.g., in the example from Figure~\ref{fig:arrangement}, deleting $2$ would result in the arrangement where $H_2$ is removed. Contracting $2$ would result in an arrangement inside $H_2\cong \mathbb{R}^1$ yielding $5$ covectors.
\end{examplenew}

Theorem~\ref{thm:minorclosed} implies immediately the following corollary.

\begin{corollary}
Finite restrictions of AOMs are finite AOMs.
\end{corollary}

\subsection{Parallelism in AOMs}\label{ssec:par}

\begin{definition}\label{defpar}
	Given two elements $e,f \in E$, we say that $e$ and $f$ are parallel, written $e \parallel f$, if there is no $X\in \covset$ with $e,f\in \ze{X}$.
\end{definition}

\begin{examplenew} In Figure \ref{fig:basis_inf} the pseudolines labeled $a_{-1},\ldots,a_{3}$ correspond to parallel elements in the associated FAOM.
\end{examplenew}

\begin{figure}[h]
\includegraphics[scale=1]{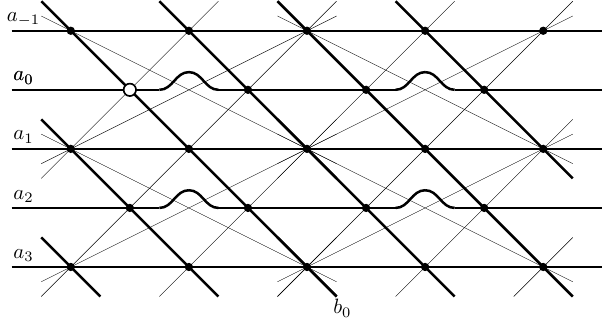}
\caption{An example of an infinite (but periodic) pseudoline arrangement. The point highlighted by a hollow bullet corresponds to the covector $X_B$ described in Example \ref{ex:xbwhite}. }\label{fig:basis_inf}
\end{figure}

Note that a different notion of parallelism in systems of sign vectors appears in the literature, that we call here \emph{equivalence} in order to avoid confusion, and that is defined by $e\sim f$ if $X(e)=X(f)$ for all $X\in \covset$ or $X(e)=-X(f)$ for all $X\in \covset$. Note that this is an equivalence relation on $E$. Another notion is that of \emph{redundant} elements, i.e., $e$ is redundant if $X(e)=Y(e)$ for all  $X\in \covset$. 

\begin{definition}\label{def:simple}
 An AOM is called \emph{simple} if all equivalence classes with respect to $\sim$ are trivial and there are no {redundant elements}.
Every AOM can be reduced to a simple one by deleting all redundant elements and all but one elements of each class of equivalent elements. The resulting AOM is a minor, that is unique up to isomorphism, called the \emph{simplification}.
\end{definition}

\begin{remark}\label{remsimplicity}
Let $\covset'$ be the simplification of an AOM $\covset$. We have $F(\covset')\cong F(\covset)$ and $\covpos(\covset')\cong \covpos(\covset)$.
\end{remark}

\begin{corollary}\label{PER}
In every simple $AOM$ with ground set $E$, the reflexive closure of parallelism is an equivalence relation on $E$. We call $\parclass{e}$ the parallelism class of $e\in E$.
\end{corollary}
\begin{proof} Symmetry of $\parallel$ being evident from the definition, we have to check transitivity. 
By way of contradiction consider three elements with $e\parallel f$, $e\parallel g$, but $f$ and $g$ not parallel. 
The restriction $\covset[\{e,f,g\}]$ is a finite AOM and so $\LL:=\ze{\covset[\{e,f,g\}]}$ is the geometric semilattice of flats of a (finite) semimatroid, see Lemma~\ref{lem:undersemifinite}.
Now since there is a covector $X$ such that $\ze{X}=\{f,g\}$, we can find $x\in\LL$ with $f,g\in x$.
Now since $\covset$ is simple, so is its restriction to $\{e,f,g\}$. In particular, all of $\{e\},\{f\},\{g\}$ are atoms of $\LL$  and $x$ has rank at least $2$ in $\LL$, so that $\{f\},\{g\}$ is an independent set of atoms. Now Axiom (GSL2) in Definition~\ref{df:GS} ensures that one among the joins $\{e\} \vee \{f\}$ and $\{e\}\vee\{g\}$ exists in $\LL$. But this implies existence of a covector $Y\in \covset$ with either $\{e,f\}\subseteq \ze{Y}$ or $\{e,g\}\subseteq \ze{Y}$, contradicting our parallelism assumption. 
\end{proof}

\begin{remark} Our notion of parallelism could be redefined so as to allow for Corollary~\ref{PER} to work also in non-simple AOMs, e.g., by adding to Definition~\ref{defpar} the requirement that $e,f$ are not loops and that $\vert\{e,f\}\cap \ze{X}\vert\neq 1$ for all $X\in \covset$ (the latter condition implying that $e,f$ are not parallel ``in matroid sense'' in the underlying semimatroid). Our choice of definition is enough since we will need it in order to study posets of covectors of AOMs, for which considering the ``simple'' case is no restriction of generality (see Remark~\ref{remsimplicity}).
\end{remark}

\begin{remark}\label{constantsign}
	Notice that $e\parallel f$ if and only if, for all $X,Y\in \covset$,  $X(f)=Y(f)=0$ implies $X(e)=Y(e)\neq 0$. This sign we can then denote by $\sigma_f(e)$.
\end{remark}
\begin{proof}
	The condition in the Remark's statement directly implies $e\parallel f$. For the other implication, suppose $e\parallel f $ and consider $X,Y\in \covset$ with $X(f)=Y(f)=0$. Clearly $X(e)\neq 0$ and $Y(e)\neq 0$ otherwise parallelism is immediately violated. It remains to prove $X(e)=Y(e)$. Indeed, if $X(e)=-Y(e)$ then $e\in S(X,Y)$ and we can eliminate $e$ obtaining $Z$ such that $Z(e)=0$ and $Z(f)=X(f)\circ Y(f)=0$, a contradiction to parallelism.
\end{proof}

\begin{definition}
Following Huntington~\cite{Huntington}, we say that a ternary relation $[\cdot,\cdot,\cdot]$ on a set $X$ is a {\em betweenness relation} if it satisfies the following requirements for all $a,b,c,x\in X$.
\begin{enumerate}
\item[(BR1)] $[a,b,c]$ implies that $a,b,c$ are distinct.
\item[(BR2)] $[\omega(a),\omega(b),\omega(c)]$ holds for some permutation $\omega$ of $\{a,b,c\}$.
\item[(BR3)] $[a,b,c]$ implies $[c,b,a]$
\item[(BR4)] $[a,b,c]$ and $[a,c,b]$ are mutually exclusive
\item[(BR5)] $[a,b,c]$ implies at least one of $[a,b,x]$ and $[x,b,c]$, whenever $x\not\in \{a,b,c\}$.
\end{enumerate}
\end{definition}

\begin{proposition}\label{prop:betweenness} Let $\pi\subseteq E$ denote a parallelism class of a simple AOM. The ternary relation on $\pi$ defined by 
\begin{center}
	\begin{minipage}{.2\textwidth}
		\begin{flushright}
		$[f,g,h]\Leftrightarrow$
		\end{flushright}
	\end{minipage}
	\begin{minipage}{.75\textwidth}
		\begin{flushleft}
		$f,g,h$ are pairwise distinct and, for all $X,Y\in \covset$,\\  $X(f)=0$ and $Y(h)=0$ imply $X(g)=-Y(g)$,
		\end{flushleft}
	\end{minipage}
\end{center}
 is a betweenness relation and is invariant under reorientation.
\end{proposition}
\begin{proof}
Properties (BR1) and (BR3) hold trivially. In order to prove the others let $a,b,c$ be distinct elements of $\pi$, recall the definition of the signs $\sigma_f(e)$ from Remark~\ref{constantsign} and for $x\in \{a,b,c\}$ write $$\sigma(x):=\prod_{y\in \{a,b,c\}\setminus \{x\}}\sigma_y(x).$$ Notice that  $[a,b,c]$ is equivalent to $\sigma_a(b)=-\sigma_c(b)$, hence to $\sigma(b)=-$. 
\begin{itemize}
\item[(BR2)] It is enough to find $x\in \{a,b,c\}$ with $\sigma(x) = -$. 
Since the AOM is simple and $a,b,c$ are parallel, in the restriction of the AOM to the set $\{a,b,c\}$ we can pick covectors $X$, $X'$, $Y$ with $\{a\} = \ze{X}$, $\{b\}= \ze{X'}$, $\{c\}= \ze{Y}$.
Now assume by way of contradiction that $\sigma(x)=+$ for all $x\in \{a,b,c\}$. Then $X(c)=X'(c)$, $X(b)=Y(b)$, $X'(a)=Y(a)$ and we can compute
$$
(X\oplus -X')(a)=-Y(a), \quad (X\oplus -X')(b)=Y(b), \quad (X\oplus -X')(c) = 0.
$$
Moreover, $I(X,-X') = I(-X,X') =\emptyset$ (Notice that $I(X,-X') = I_{c}(X,-X')$, hence any $W\in I(X,-X')$ must have $W(c)=0$, $W(a)=-Y(a)$ and thus would witness $\sigma(a)=\sigma_b(a)\sigma_c(a)=X'(a)(-Y(a))=X'(a)(-X'(a))=-$, a contradiction. For $I(X',-X)$ the reasoning is analogous and contradicts $\sigma(b)=+$.)   Thus, by (P) the family $\covset$ should contain
the covector $W':=(X\oplus -X') \circ Y$, which satisfies $W'(a)=-Y(a)$ and $W'(c)=0$ and
 would witness again $\sigma(a)=X'(a)(-Y(a))=-$, a contradiction.
\item[(BR4)] Let us argue again in the restriction to $\{a,b,c\}$ and pick covectors $X$, $X'$ and $Y$ as in the proof of (BR2).  If both $[a,b,c]$ and $[a,c,b]$ hold, then $\sigma(b)=\sigma(c)=-$ and so  $X(b)=-Y(b)$ and $X(c)=-X'(c)$. In particular, by (SE) the set $I(X,Y)=I_b(X,Y)$ contains some $W$ with $W(b)=0$ and $W(c)=X(c)\circ Y(c)=X(c)$, but the latter is opposite to $X'(c)$, a contradiction to $\sigma_b(c)$ being well-defined.
\item[(BR5)] Recall that $[a,b,c]$ means $\sigma_a(b)=-\sigma_c(b)$. Then, for every $x\not\in\{a,b,c\}$ either $\sigma_x(b)=\sigma_a(b)$, in which case $[x,b,c]$, or $\sigma_x(b)=\sigma_c(b)$, in which case $[a,b,x]$.
\end{itemize}
Invariance under reorientation is apparent from the definition of $[\cdot,\cdot,\cdot]$.
\end{proof}

\begin{corollary}\label{totalorder}
	Let $\pi\subseteq E$ be a parallelism class of the given AOM. Then there is a total order $<_\pi$ on $\pi$, unique up to order reversal,  such that $a<_\pi b <_\pi c$ if and only if $[a,b,c]$. In particular, this ordering is independent on the reorientation of the AOM. Moreover, there is a reorientation of $\pi$ such that
	\begin{center}
		for all $x,y\in \pi$, $x<_\pi y$ if and only if $\sigma_y(x)=+$,
	\end{center}
			where $\sigma_y(x)$ is defined in Remark \ref{constantsign}.
\end{corollary}
\begin{proof}
If $\pi$ has less than $3$ elements, the claim is trivial. Otherwise, recall the betweenness relation on $\pi$ defined in Proposition~\ref{prop:betweenness} and choose two distinct elements $e,f\in \pi$. In~\cite[\S 3.1]{Huntington} it is proved that the condition ``$a<_\pi b <_\pi c$ if and only if $[a,b,c]$'' determines a pair of opposite total orderings on $\pi$, and thus letting $e<_\pi f$ fully determines a total ordering of $\pi$ with the desired properties. 

The desired reorientation is obtained by reorienting  $e,f$ so that $\sigma_f(e) = + $ and $\sigma_e(f)=-$, as well as  reorienting every other $x\in \pi$ so that $\sigma_e(x)=-$ if and only if $e<_\pi x$.
\end{proof}

Notice that the total order $<_\pi$ obtained in Corollary~\ref{totalorder} is unique up to order reversal. In particular, the following definition is well-posed (where we assume, after possibly reversing the order, that if an extremum exists, it is a minimum).

\begin{definition}\label{def:totalorder} Write $1_\pi$ resp. $0_\pi$ for the unique maximal (resp. minimal) element of $<_\pi$ when they exist. We will assume, after possibly reversing the order, that if an extremum exists, then a minimum exists. 
Corollary~\ref{totalorder} allows us then to define the following partition of the ground set of an AOM, independently from the reorientation.
\begin{align*}
E^{01}:=& \{e \in E \mid \textrm{both }1_{\pi(e)} \textrm{ and } 0_{\pi(e)} \textrm{ exist}\},\\
E^{0*}:=& \{e \in E \mid 0_{\pi(e)} \textrm{ exists}\} \setminus E^{01}\\
E^{**}:=&  E\setminus (E^{01} \cup E^{0*}) 
\end{align*}
\end{definition}

\begin{examplenew}
Figure \ref{examplesab} illustrates different types of elements of the ground sets of FAOMS represented by affine line arrangements.
\end{examplenew}

\begin{figure}[h!]
\centering
\begin{tikzpicture}
\foreach \x in {0,1,2,3,4,5}
    \draw (.52*\x-1.4,-.3) -- (.52*\x-1.4,3);
\foreach \x in {-1,0,1,2,3,4}    
    \node (\x) at (.52*\x-.8,-.7) {$h_{\x}$};
    \node (ph) at (2,-.7) {$\cdots$};
    \node (phl) at (-2,-.7) {$\cdots$};    
\foreach \y in {0}
    \draw (-2,.45*\y+.9) -- (1.7,.45*\y+.9);
\foreach \y in {0}    
    \node (\y) at (-2.3,.45*\y+.9) {$l$};
\node (a) at (0,-1.5) {(a)};
\end{tikzpicture}
\quad
\begin{tikzpicture}
\foreach \x in {0,1,2,3,4,5}
    \draw (.52*\x-1.4,-.3) -- (.52*\x-1.4,3);
\foreach \x in {-1,0,1,2,3,4}    
    \node (\x) at (.52*\x-.8,-.7) {$h_{\x}$};
    \node (ph) at (2,-.7) {$\cdots$};
    \node (phl) at (-2,-.7) {$\cdots$};    
\foreach \y in {0,1,2,3}
    \draw (-1.7,.45*\y+.9) -- (2,.45*\y+.9);
\foreach \y in {0,1,2,3}    
    \node (\y) at (-2,.45*\y+.9) {$l_{\y}$};
\node[anchor=center] (pvl) at (-.1,2.9) {$\vdots$};
\node[anchor=center] (pv) at (-2,3) {$\vdots$};
\node (b) at (0,-1.5) {(b)};
\end{tikzpicture}
\quad
\begin{tikzpicture}
\foreach \x in {0,1,2,3,4,5}
    \draw (.52*\x-1.4,-.3) -- (.52*\x-1.4,3);
\foreach \x in {-1,0,1,2,3,4}    
    \node (\x) at (.52*\x-.8,-.7) {$h_{\x}$};
    \node (ph) at (2,-.7) {$\cdots$};
    \node (phl) at (-2,-.7) {$\cdots$};    
\foreach \y in {-2,-1,0,1,2,3}
    \draw (-1.7,.45*\y+.9) -- (2,.45*\y+.9);
\foreach \y in {-2,-1,0,1,2,3}    
    \node (\y) at (-2,.45*\y+.9) {$l_{\y}$};
\node[anchor=center] (pvl) at (0,2.9) {$\vdots$};
\node[anchor=center] (pv) at (-2.5,3) {$\vdots$};
\node (b) at (0,-1.5) {(c)};
\end{tikzpicture}
\caption{Three examples of infinite line arrangements. The FAOM associated to example (a) has $E^{01}=\{l\}$, $E^{\ast\ast} = \{h_i\}_{i\in \mathbb Z}$, $E^{0\ast}=\emptyset$; the one associated to example (b) has $E^{01}=\emptyset$, $E^{\ast\ast} = \{h_i\}_{i\in \mathbb Z}$, $E^{0\ast}=\{l_j\}_{j\in \mathbb N}$; the one of example (c) has $E^{01}=E^{0\ast} = \emptyset$, $E^{\ast\ast}=\{l_j\}_{j\in \mathbb N} \cup \{h_j\}_{j\in \mathbb N}$. } 
\label{examplesab}
\end{figure}

\section{Finitary affine oriented matroids (FAOM)}
\label{sec:FAOM}
We now move to a more restrictive definition, especially in order to approach a topological study of covector posets and to connect our theory with that of the unoriented ``affine" version of matroids, i.e., semimatroid theory.

\begin{definition}[FAOM]
 A \emph{Finitary} Affine Oriented Matroid
 is a pair $(E,\covset)$ that is an AOM (i.e., it satisfies (FS), (SE), and (P)) that furthermore fulfills: 
 \begin{itemize}
  \item[(S)] $X,Y\in\covset\implies |S(X,Y)|<\infty$ \hfill(finite separators),
  \item[(Z)] $X\in\covset\implies |\ze{X}|<\infty$ \hfill(finite zero sets),
  \item[(I)] $\vert \covpos(\covset)_{\leq X}\vert < \infty$ \hfill(finite intervals).
 \end{itemize}
\end{definition}

\begin{remark} \label{thepowerofZ}Axiom (Z) implies that $\covset$ is closed under infinite composition, since $|\ze{X\circ Y}|<|\ze{X}|$, unless $X\circ Y=X$. Thus, any result of an infinite composition can be expressed as a finite composition. Moreover, (Z) also implies that $\vert \covpos(\covset)_{\geq X}\vert < \infty$ for every $X\in \covset$.
\end{remark}

\begin{remark} 
	Axiom (Z) might be weakened to 
	\begin{itemize}
	\item[(Z')] $X\in\covset\implies |\ze{X}\setminus\bigcap_{Y\in\covset}\ze{Y}|<\infty$. 
	\end{itemize}
	Most of the statements and proofs remain valid with some technical adjustments. We choose the stronger axiom in order to fit the ``finitary" nature of the extant literature on matroids. For instance, (Z') would allow for infinitely many elements of rank $0$ in the underlying semimatroid. Again, as we are interested in the structure of $\covpos(\covset)$, omitting loops does not restrict generality.
\end{remark}

The next lemma follows from the more general property that faces of COMs are OMs.

\begin{lemma}[{See Lemma 4 in \cite{Ban-18}}]\label{locOM}
	Let $\covset$ be the set of covectors of an FAOM and let $X\in \covset$. Then the poset $\covpos(\covset)_{\geq X}$ is isomorphic to the poset of covectors of an oriented matroid. More precisely, $\mathscr O:=\{Y_{\vert \ze{X}} \mid Y\geq X\} = \covset[\ze{X}]$ is the set of covectors of an oriented matroid on the ground set $\ze{X}$.
\end{lemma}

\subsection{Topes, convex sequences and rank}\label{subsec:topes}
 \newcommand{\topes}[2]{\mathscr T(#1,#2)}
  \newcommand{\tope}{\mathscr T}
  
 \newcommand{\convset}{\mathscr C}
  \newcommand{\res}{\operatorname{res}}

\begin{definition}
 The set of \emph{topes} $\tope$ of simple FAOM is constituted by the elements of $\covset$ of full support, i.e., with $X_e\neq 0$ for all $e\in E$. 
 \end{definition}
 
 \begin{remark}\label{topexist}
 Note that a tope exists in a simple FAOM, since for every $e\in E$, there is an $X^e\in\covset$ with $X^e(e)\neq 0$ and we can obtain a tope by composing all $\{X^e\mid e\in E\}$ by Remark~\ref{thepowerofZ}.
 Moreover, every $X\in\covset$ is below some tope in in the face poset -- just take any $X\circ T$, with $T\in\tope$.
 \end{remark}
 
 \begin{definition}
 The \emph{tope graph} $G_{\covset}$ is the simple graph with the set $\tope$ as its vertices and where a pair of vertices $T$, $T'$ form an edge if and only if $|S(T,T')|=1$. In other words, $G_{\covset}$ is an induced subgraph of the \emph{hypercube} $Q_E$.
 \end{definition}

Note that (I) implies that all topes have a finite number of neighbors, since the edges of a tope $T$ are in $\covpos(\covset)_{\leq T}$. Together with (S), this implies the following statement.

\begin{remark}\label{lifeisgood}
 In the tope graph $G_{\covset}$ of a simple FAOM , every vertex has finite degree and the distance of any two vertices is finite.
\end{remark}

\begin{remark}
Remark \ref{lifeisgood} entails in particular that the distance in $G_{\covset}$ between any two topes $T,T'\in\tope$ satisfies $d(T,T')=|S(T,T')|$. In particular, $G_{\covset}$ is an isometric subgraph of $Q_E$, i.e., a \emph{partial cube} (of finite degree and finite distances), see~\cite[Proposition 2]{Ban-18}. Moreover,  $\covset$ is uniquely determined by $\tope$ and up to isomorphism by $G_{\covset}$, see~\cite[Corollary 4.10]{Kna-17}. 
\end{remark}

\def\balt{\mathscr B}

A subset of vertices of a graph $G$ is called {\em convex} if it contains all vertices of every shortest path between any two if its vertices. The \emph{convex hull} of a subset of vertices of $G$ is the smallest convex subgraph containing it, if it exists. A subset $\convset\subseteq \tope$ is called \emph{convex} if it is convex as a set of vertices of $G_\covset$.
 
 \begin{remark}\label{rem:convex}
It is well-known, see e.g.~\cite{Alb-16}, that convex subsets in partial cubes coincide with intersections of halfspaces, this is $\convset\subseteq \tope$ is convex if and only if there is a sign vector $X$ (not necessarily in $\covset$), such that $\convset=\{T\in\tope\mid T\geq X\}$. Moreover, for a finite subset $\balt$ its convex hull $\mathrm{conv}(\balt)$ can be represented by 

$$
X_{\balt}(e):=\left\{\begin{array}{ll}
0 & \textrm{ if } e \in S(T,T') \textrm{ for some } T,T'\in\balt\\
\bigcirc_{T\in\balt} T(e) &\textrm{ otherwise.}
\end{array}\right.
$$

 \end{remark}
 
%
\begin{lemma}\label{lemseq}
 	Let $\covset$ be a simple FAOM and $C\in \tope$. There is an increasing sequence 
	$$
	\convset_1\subsetneq \convset_2\subsetneq \ldots
	$$
	of {\bf finite} convex subsets of $\covset$ such that $C\in \convset_1$ and $\covset = \cup_i\convset_i$.
 \end{lemma}
 \begin{proof}
 
 We consider the tope graph $G_{\covset}$ of $\covset$. Take the sequence of balls $\balt_i$ of radius $i$ around $C$, starting with $\balt_0=\{C\}$. If we define $\convset_i=\mathrm{conv}(\balt_i)$ as the smallest convex subgraph of $G_\covset$ that contains $\balt_i$, which is finite because 
  all degrees are finite. In particular, $\balt_i$ is a subgraph of a finite partial cube and by Remark \ref{rem:convex} $\convset_i$ is a subgraph of the same cube, hence finite.  
  Moreover, by definition the subgraph induced by $\convset_i$ is convex and since distances are finite the sequence eventually exhausts the entire graph $G_\covset$.
 \end{proof}

\begin{examplenew} \label{ex:matrioshka}
In Figure \ref{fig:deltas} an example of a sequence $\mathscr C_1,\mathscr C_2,\ldots$ associated to the base tope $C$ is shaded in decreasing hue of gray.
\end{examplenew}

\begin{figure}[h]
    \begin{tikzpicture}[x=3.5em,y=3.5em]
\foreach \x in {0,1,2,3,4,5}
    \draw (.52*\x-1.4,-.3) -- (.52*\x-1.4,3);
\foreach \x in {-1,0,1,2,3,4}    
    \node (\x) at (.52*\x-.8,-.7) {$h_{\x}$};
    \node (ph) at (2,-.7) {$\cdots$};
    \node (phl) at (-2,-.7) {$\cdots$};    
\foreach \x in {0,1,2,3,4,5}    
    \node (\x) at (.52*\x-1.3,0) {$\to$};    
\foreach \y in {0,1,2,3}
    {
    \draw (-2,.45*\y+.9) -- (2.5+\y,.45*\y+.9);
    }
\foreach \y in {1,2,3}
    {    
    \fill[gray,opacity=0.2] (-2,.45*\y+.9) -- (2.5+\y,.45*\y+.9) -- (2.5+\y,0.3) -- (-2,0.3) -- (-2,.45*\y+.9);    
    \node (C\y) at (2.2+\y,.45*\y+.5) {$\mathscr C_{\y}$};
    }
\foreach \y in {0,1,2,3}    
    \node (\y) at (-2.5,.45*\y+.9) {$l_{\y}$};
\foreach \y in {0,1,2,3}    
    \node (\y) at (-1.8,.45*\y+1) {$\uparrow$};    
\node[anchor=center] (pvl) at (0,2.9) {$\vdots$};
\node[anchor=center] (pv) at (-2.5,3) {$\vdots$};
\fill[pattern=north west lines, pattern color=black] (-0.1,1.1) -- (-0.1,1.1) -- (-0.1,1.1) -- (-0.1,1.1) -- (-0.1,1.1);
\node[anchor=center] (B) at (-0.1,1.1) {$C$};
\node[anchor=center] (X) at (-0.55,1.5) {$X$};
\draw[ultra thick] (1.04-1.4,.9+.9) -- (1.04-1.4,.45+.9);
\end{tikzpicture}
\caption{Picture for Example \ref{ex:matrioshka} and Example \ref{ex:delta}. }\label{fig:deltas}
\end{figure}

 Lemma \ref{lemseq} implies immediately that FAOMs can only have countably many topes. This, together with $(I)$, yields the following corollary.

\begin{corollary}\label{cor:countable}
 	The set of covectors of a FAOM has countable cardinality.
\end{corollary}

 For $C\in\tope$ we define $\topes{\covset}{C}$ to be the \emph{tope poset based at $C$}, i.e, for topes $T,T'\in\tope$ we have $T\leq T'$ if and only if $S(C,T)\subseteq S(C,T')$.

 \begin{lemma}\label{lem:finconv}
	Let $\covset$ be a simple FAOM, $\convset\subseteq \tope$ convex and $C\in\convset$. Then $\convset$ is a lower ideal of the poset $\topes{\covset}{C}$. Moreover, if $\vert \convset \vert < \infty$, then 
	there is a finite subset $E_{\convset}\subseteq E$ such that the restriction map $\res_\convset: \covset\to\covset[E_\convset]$ restricts to order isomorphisms between
	\begin{itemize}
	\item[(1)] 
	$\covpos(\covset)_{\leq {\convset}}$ and $ \covpos(\covset[E_\convset])_{\leq \res_\convset(\convset)}$, as well as,
	\item[(2)] the induced subposets $\convset\subseteq \topes{\covset}{C} $ and $\res_\convset(\convset)\subseteq\topes{\covset[E_\convset]}{\res_{\convset}(C)}$.
	\end{itemize}
 \end{lemma}
 
 \begin{proof}
 Let $\convset\subseteq \tope$ be convex and $C\in \convset$. Then by definition the set $\convset$ is convex in $G_{\tope}$.  By Remark~\ref{rem:convex}, there is a sign vector $X_\convset$ such that $\convset=\{T''\in\tope\mid T''\geq X_\convset\}$. Let $T\leq T'$ with respect to $\topes{\covset}{C}$ and $T'\in\convset$. Since $S(C,T)\subseteq S(C,T')$ and $C,T'\geq X_\convset$ also $T\geq X_\convset$. Thus, $\convset$ is an order ideal of $\topes{\covset}{C}$.

 Now let $\convset$ be finite, then since distances are finite, separators of elements of $\convset$ are finite and $\ze{X_{\convset}}$ is finite by Remark~\ref{rem:convex}. 
 We define $E_\convset$ as the union $\bigcup_{Y\leq T\in \convset}\ze{Y}$. Note that this is a finite set by (I) and (Z) contains $\ze{X_{\convset}}$.
 
The representation of $\convset$ as intersection of halfspaces $\{T''\in\tope\mid T''\geq X_{\convset}\}$ yields (2). In particular, since all members of $\convset$ are identical on the complement of $E_{\convset}$, $\res_\convset$  induces an injection from $\convset$ whose restriction to $\res_\convset(\convset)$ is then bijection. 
To see (1) note that for two covectors below $\convset$ with an element $e$ in their separator, there are also two topes $T,T'\in\convset$ with $e\in S(T,T')$ and $e\in \ze{X_{\convset}}$. Thus, all members of $\covpos(\covset)_{\leq {\convset}}$ are identical on the complement of $E_{\convset}$. This proves injectivity.
Moreover, all elements of $\covpos(\covset)\setminus\covpos(\covset)_{\leq {\convset}}$ have an element $e\in E_\convset$ that is in the separator with all elements of $\convset$. Thus, their restriction cannot be in $ \covpos(\covset[E_\convset])_{\leq \res_\convset(\convset)}$. Since $\res_\convset$ from $\covpos(\covset)$ to $ \covpos(\covset[E_\convset])$ is surjective by definition, this proves surjectivity.

 
 
 \end{proof}

 \begin{corollary}\label{Franked}
 	The poset $\wtb{\covpos(\covset)}$ (obtained from  $\covpos(\covset)$ by adding a global minimum and a global maximum, cf.\ Appendix \ref{a:posets}) is graded of finite length. 
 \end{corollary}
 \begin{proof}
 	We have to prove that any two maximal chains in $\wtb{\covpos(\covset)}$ have the same, finite length. Let $\omega,\omega'$ be two such chains  and write $X:=\max(\omega\setminus \{\top\})$, $X':=\max(\omega'\setminus \{\top\})$. Let $i$ be such that $X,X'\in \mathscr C_i$ as in Lemma~\ref{lemseq}. Then by Lemma~\ref{lem:finconv} both $\omega$ and $\omega'$ are maximal chains in the poset  $\wtb{\covpos(\covset[E_{\mathscr C_i}])}$ that is graded of finite length by~\cite[Theorem 4.5.3]{bjvestwhzi-93}.
 \end{proof}
 
 \begin{remark}
 From Corollary~\ref{Franked} follows immediately that $\covpos(\covset)$ is ranked of finite length.
 \end{remark}
 

\begin{corollary}\label{cor:rcw}
	The poset $\covpos(\covset)$ is the poset of cells of a  regular CW-complex that we call $\rcw(\covset)$. The dimension of $\rcw(\covset)$ is the length of $\covpos(\covset)$.
\end{corollary}
\begin{proof}
	By~\cite[Proposition 4.7.23]{bjvestwhzi-93}, it is enough to check that, for every $X\in \covpos(\covset)$, the interval $I_X:=\covpos(\covset)_{<X}$ is homeomorphic to a $\rk(X)$-sphere. Now by Lemma~\ref{lem:finconv}   $I_X$ is an order ideal in $\covpos(\covset[E_{\mathscr C}])$, for some finite, convex set of topes $\mathscr C$ (whose existence is ensured by Lemma~\ref{lemseq}). In particular, $I_X$ is an order ideal in the ``bounded complex'' of the finite affine oriented matroid $\covpos(\covset[E_{\mathscr C}])$ (see~\cite[Definition 4.5.1]{bjvestwhzi-93}). With~\cite[Discussion before 4.5.7, \S 4.3]{bjvestwhzi-93} the claim follows.
	\end{proof}

\subsection{Shellings} \label{ssec:shellings}
As a stepping stone towards determining the topology of covector posets, we prove their shellability. A brief introduction to shellability as well as to the parts of the theory we need here are sketched in the Appendix (\S\ref{app:sc}, \S\ref{app:oc}), where we also point to some literature for further background.

 \begin{proposition}\label{prop:ShCo}
	Let $\covset$ be the set of covectors of an FAOM.
 Then, the poset $\wtb{\covpos(\covset)}$ admits a recursive coatom ordering without critical chains. 
	In particular, both the regular CW-complex $\rcw(\covset)$ and its barycentric subdivision, the simplicial complex $\Delta(\covpos(\covset))$, are shellable and contractible.
 \end{proposition}
 \begin{proof}
 	Let $C$ be any tope of $\covset$ and choose a sequence $\convset_1\subsetneq \convset_2\subsetneq \ldots$ according to Lemma~\ref{lemseq}. Since for all $i$ the set $\convset_i$ is a lower ideal of $\topes{\covset}{C}$, it is possible to  choose a linear extension $\prec$ of $\topes{\covset}{C}$  such that $T_i\prec T_{j}$ if $T_i\in \convset_i$, $T_j\in \convset_j$ and $ i<j$. Then $\prec$ is a total order, and it is a well-ordering because for any $\mathscr X \subseteq \topes{\covset}{C}$ the set $\{i\mid \convset_i\cap \mathscr X\neq \emptyset\}\subseteq \mathbb N$ has a minimum, say $i_0$, and the set $\convset_{i_0}\cap \mathscr X$ is finite, hence has a $\prec$-minimum $T$, that is also the minimum of $\mathscr X$. 
	
	We show that $\prec$ defines a recursive coatom ordering on the poset $\wtb{\covpos(\covset)}$. The set of coatoms of $\wtb{\covpos}(\covset)$ is exactly the set of topes. So let $T$ be a tope, and let $i$ be such that $T\in \convset_i $. For $T',T'' \in \res_{\convset_i}(\convset_i)$ let $T'\prec_i T''$ if and only if  $\res_{\convset_i}^{-1}(T')\prec\res_{\convset_i}^{-1}(T'')$.
	Since $\res_{\convset_i}(\convset_i)$ is a lower ideal in $\topes{\covset[E_{\convset_i}]}{\res_{\convset_i}(C)}$, the order $\prec_i$ can be extended to a linear extension of $\topes{\covset[E_{\convset_i}]}{C}$ where the elements of $\res_{\convset_i}(\convset_i) $ come first. By~\cite[Proposition 4.5.6]{bjvestwhzi-93}, this defines a recursive coatom ordering $\prec_i$ of  $\wtb{\covpos}(\covset[E_{\convset_i}])$. For every tope $T'$ of $\covset[E_{\convset_i}]$ we let $Q_{T'}^i$ be the associated distinguished set of coatoms of $\covpos(\covset[E_{\convset_i}])_{\leq T'}$, uniquely determined by $\prec_i$. 
	Now, for any $T\in\topes{\covset}{C}$ we can set $Q_T:= \res_{\convset_i}^{-1}Q^{i}_{\res_{\convset_i}(T)}$ where $i$ is such that $T\in \convset_i$ (\footnote{indeed for every $j<i$ and every $T\in \convset_j$, $\res_{\convset_j}^{-1}Q^{j}_{\res_{\convset_j}(T)}= \res_{\convset_i}^{-1}Q^{i}_{\res_{\convset_i}(T)}$, since $\prec_j$ equals the restriction of $\prec_i$ to $\convset_j$.}).

	Thus, by Remark~\ref{shellingrem}, $\Vert {\covpos}(\covset) \Vert$ is shellable. Now if some chain $\omega\subseteq {\covpos}(\covset)$ is critical, then it is critical also in the shelling of ${\covpos}(\covset[E_{\convset_i}])$, where $i$ is such that $\max\omega\in \convset_i$. But we know~\cite[Theorem 4.5.7]{bjvestwhzi-93} that ${\covpos}(\covset[E_{\convset_i}])$ is contractible, hence no shelling of the latter poset has critical chains. Therefore the obtained shelling of $\Vert {\covpos}(\covset) \Vert$ has no critical cells either, and this complex is contractible as well. The claim about $\rcw(\covset)$ follows with~\cite[Lemma 4.7.18]{bjvestwhzi-93}.
\end{proof}

\subsection{Topology of covector posets}
\label{ssec:topcov}
We study the topology of the order complexes of posets of covectors of FAOMs. For basic terminology and notations about combinatorial topology we refer again to the Appendix~\ref{sec:simplicial}.

\begin{lemma}\label{lem:diamond}
	Let $\covset$ be the set of covectors of an FAOM. Let $\omega$ be a maximal chain in ${\covpos(\covset)}$, let $X\in \omega$ and write $\omega':=\omega\setminus\{X\}$. 
	Let $\mathscr Y$ be the set of all $Y\in \covset$ such that $\omega'\cup\{Y\}$ is a chain in $\covpos(\covset)$. Then 
	\begin{itemize}
	\item[(1)] $\vert \mathscr Y \vert \leq 2$, and
	\item[(2)] the boundary of $\Vert\covpos(\covset) \Vert$ is generated by all chains of the form $\omega'$ with $\vert \mathscr Y \vert =1$.
	\end{itemize}
\end{lemma}
\begin{proof}
	For (1) first notice that by (Z) we have $\vert \mathscr Y \vert <\infty$. By Lemma~\ref{lemseq} we can find a finite convex set $\mathscr C$ such that $\omega\cup\mathscr Y \subseteq \covset[E_{\mathscr C}]$.
	Now the claim follows with Lemma~\ref{lem:finconv} from the corresponding property in  the finite affine oriented matroid $\covset[E_{\mathscr C}]$ (see~\cite[Theorem 4.1.14]{bjvestwhzi-93}, and recall that the  poset of covectors of a finite AOM is a filter inside that of a finite oriented matroid).
	
	Now we prove (2).  Consider any chain $\gamma$ in $\covpos$. Recall that shellability of $\covpos(\covset)$ implies shellability of links, that $\covpos(\covset)$ being ranked implies that the link of $\gamma$ is pure of dimension $k:=d-\dim\Vert\gamma\Vert$,  and that (1) implies that every ridge of $\operatorname{Lk}(\gamma)$ is contained is at most two facets of the link (see \S\ref{app:sc}). Moreover, by (Z) and (I) the complex $\operatorname{Lk}(\gamma)$ is finite. Now~\cite[Proposition 4.7.22]{bjvestwhzi-93} applies, implying that $\operatorname{Lk}(\gamma)$ is either a sphere or a closed ball, the second case entering exactly if there is one ridge that is contained in only one facet. Equivalently (e.g., by~\cite[p.\ 7]{Rourke-Sanderson}) $\Vert\gamma\Vert$ is in the boundary of $\Vert\covpos(\covset) \Vert$ if and only if  $\gamma\subseteq \omega \setminus\{X\}$ for a maximal chain $\omega\subseteq \covpos(\covset)$ and $X\in \omega$ such that $\vert \mathscr Y \vert =1$.
\end{proof}

\begin{theorem}\label{thm:PLball}
	Let $\covset$ be the set of covectors of an FAOM. Then $\Vert {\covpos}(\covset)\Vert$ is a shellable, contractible PL $d$-manifold whose boundary is described in Lemma~\ref{lem:diamond}.(2). Moreover,
	\begin{itemize}
	\item[(1)] If $\covset$ is finite, then $\Vert {\covpos}(\covset)\Vert$ is a PL-ball.
	\item[(2)] If $\Vert {\covpos}(\covset)\Vert$ has no boundary, then it is PL-homeomorphic to $\mathbb R^{\rk\covset}$.
	\end{itemize}
\end{theorem}
\begin{proof}
	The complex $\Vert {\covpos}(\covset)\Vert$ is shellable and contractible by Proposition~\ref{prop:ShCo}. 
	Moreover, it is a PL-manifold  (Lemma~\ref{lem:diamond}.(1)) with the stated boundary (Lemma~\ref{lem:diamond}.(2)). For the itemized claims: (1) is~\cite[Theorem 4.5.7.(i)]{bjvestwhzi-93}, and (2) follows from ~\cite[Theorem 1.5.(ii)]{BjoInf}, since axioms (Z), (I) imply that $\Vert {\covpos}(\covset)\Vert$  is finitary. 
\end{proof}

\begin{examplenew} When the FAOM arises from an arrangement in Euclidean space, the cell complex $\Vert \covpos (\covset)\Vert$ is isomorphic to the barycentric subdivision of the dual complex of the induced stratification of the ambient space. t
With this we can see that if $\covset$ is the FAOM of the arrangement in  Figure \ref{fig:arrangement} then $\Vert \covpos (\covset)\Vert$ is homeomorphic to a $2$-disk. On the other hand, if $\covset$ is the FAOM of the arrangement in  Figure \ref{examplesab}.(b)  then $\Vert \covpos (\covset)\Vert$ is a tiling of the plane by squares, i.e., it is homeomorphic to $\mathbb R^2$. Figure \ref{examplesab}.(a) shows an example outside the scope of Theorem \ref{thm:PLball}: there, $\Vert \covpos (\covset)\Vert$ is the poset of cells of an infinite row of juxtaposed squares. See Figure \ref{fig:cells}.
\end{examplenew}

\begin{figure}[h]
\centering
\begin{tikzpicture}[x=10em,y=12em]
\node (T) at (-1,-.3) {
    \includegraphics[scale=.8]{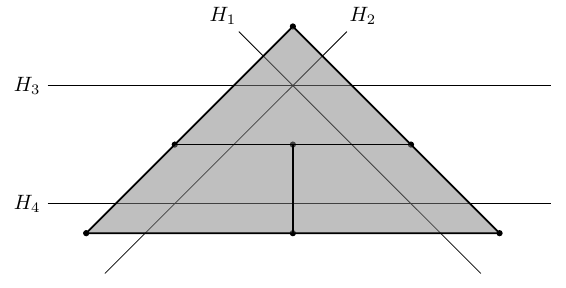}
    };
\node (S) at (-1,-1) {
    \includegraphics[scale=1]{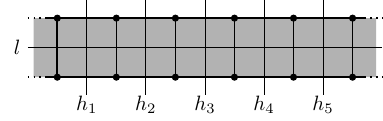}
    };
\node (B) at (1,-.6) {
    \includegraphics[scale=1]{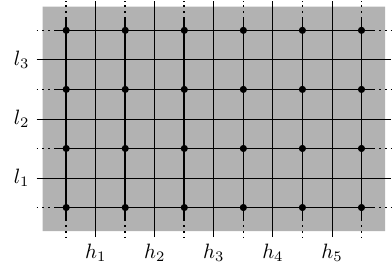}
    };
\end{tikzpicture}
\caption{Three examples of polyhedral complexes realizing $\Vert \covpos (\covset)\Vert$. From top left clockwise, they correspond to the FAOMs in Figure \ref{fig:arrangement},  Figure \ref{examplesab}.(c) and  Figure \ref{examplesab}.(a). The resulting cell complexes (whose barycentric subdivision is $\Vert \covpos(\covset)\Vert$) are shaded gray: they are homeomorphic to a closed disk, $\mathbb R^2$ and $\mathbb R^1\times [0,1]$, respectively.}\label{fig:cells}
\end{figure}


\subsection{The underlying semimatroid}
\label{ssec:us}

We show that zero sets of covectors of an FAOM define a semimatroid on the same ground set. We point to the Appendix~\ref{sec:semimatroids} for some basic notions and notations of semimatroid theory that we will be using in this section. 

\begin{definition}\label{def:KL}

Given a system of sign vectors $(E,\covset)$ define
$$
\flatset(\covset) := \{\ze{X} \mid X\in \covset\}\quad\quad
\central(\covset) := \bigcup_{A\in \flatset(\covset)} 2^A.
$$
 Elements of $\flatset(\covset)$ are called {\em flats}, elements of $\central(\covset)$ ``central sets'' of $\covset$.
For $A,B\in \flatset(\covset)$, let $A\leq B :\Leftrightarrow A\subseteq B$ and define
$$
\flatpos(\covset) := (\flatset(\covset), \leq).
$$
\end{definition}
\begin{remark} \label{rem:zeromap}
Notice that $X\leq Y$ implies $\ze{X} \supseteq \ze{Y}$, where $\leq$ is the partial order of $\covpos$. Therefore, taking zero sets induces an order reversing poset map $\ze{\cdot}: \covpos(\covset)\to \flatpos(\covset)$.
\end{remark}

\begin{examplenew} Figure \ref{fig:twoposets} depicts the poset $\covpos(\covset)$ and $\flatset(\covset)$ where $\covset$ is the set of covectors of the line arrangement in Figure \ref{fig:arrangement}.


\begin{figure}[h]
\begin{tikzpicture}[x=2em,y=3em]
\foreach \x in {2,3,4,5,6,7,8,9,10}
    \node[circle,minimum size=4pt,inner sep=0,fill] (t\x) at (\x,0) {};
\foreach \x in {1,2,3,4,5,6,7,8,9,10,11}
    \node[circle,minimum size=4pt,inner sep=0,fill] (l\x) at (\x,-1) {};   
\foreach \x in {3,6,9}
    \node[circle,minimum size=4pt,inner sep=0,fill] (p\x) at (\x,-2) {};  
\node[anchor=north] (P) at (p3) {$P$};    
\node[anchor=north] (Q) at (p6) {$Q$};    
\node[anchor=north] (R) at (p9) {$R$};
\foreach \x in {1,2,3,4,5,6}
\draw (p3.center) -- (l\x);
\foreach \x in {5,7,8,9}
\draw (p6.center) -- (l\x);
\foreach \x in {6,9,10,11}
\draw (p9.center) -- (l\x);
\draw (l1.center) -- (t2.center) -- (l2.center) -- (t3.center) -- (l3.center) -- (t4.center) -- (l4.center) -- (t5.center) -- (l5.center) -- (t6.center) -- (l6.center) -- (t7.center) -- (l1.center);
\draw (t4.center) -- (l7.center) -- (t8.center) -- (l8.center) -- (t9.center) -- (l9.center) -- (t6.center) -- (l5.center) -- (t5.center);
\draw (t6.center) -- (l9.center) -- (t9.center) -- (l10.center) -- (t10.center) -- (l11.center) -- (t7.center) -- (l6.center) -- (t6.center);
\end{tikzpicture}
\quad
\begin{tikzpicture}[x=2em,y=3em]
\foreach \x in {2,4,6}
    \node[circle,minimum size=4pt,inner sep=0,fill] (t\x) at (\x,0) {};
\foreach \x in {1,2,3,4}
    {
    \node[circle,minimum size=4pt,inner sep=0,fill] (l\x) at (2*\x-1,-1) {};
    }
\foreach \x in {2,3,4}
    {    
    \node[anchor=west] (H\x) at (l\x) {$H_{\x}$};
    }
    \node[anchor=east] (H1) at (l1) {$H_{1}$};    
\foreach \x in {3}
    \node[circle,minimum size=4pt,inner sep=0,fill] (p\x) at (\x,-2) {};  
\node[anchor=south] (P) at (t2) {$P$};    
\node[anchor=south] (Q) at (t4) {$Q$};    
\node[anchor=south] (R) at (t6) {$R$};
\foreach \x in {1,2,3,4}
\draw (p3.center) -- (l\x);
\foreach \x in {1,2,3}
\draw (t2.center) -- (l\x);
\foreach \x in {2,4}
\draw (t4.center) -- (l\x);
\foreach \x in {1,4}
\draw (t6.center) -- (l\x);
\node[anchor=north] (X) at (p3) {
$\mathbb R^2$
};
\end{tikzpicture}
\caption{The poset $\covpos(\covset)$ (l.-h.s.) and the poset $\flatset(\covset)$ (r.-h.s.) for the FAOM $\covset$ corresponding to the arrangement of Figure \ref{fig:arrangement}.}
\label{fig:twoposets}
\end{figure}
\end{examplenew}

\begin{lemma}\label{lem:undersemifinite}
	Let $(E,\covset)$ be a finite AOM. Then $\flatpos(\covset)\subseteq 2^E$ is the geometric semilattice of flats of a finitary semimatroid $\SS(\covset)$. This underlying semimatroid is unique up to isomorphism. 
\end{lemma}
\begin{proof}
	By Remark \ref{finiteok},  $(E,\covset)$ is a finite affine oriented matroid in the sense of~\cite{bjvestwhzi-93}. In this case, $\covset$ can be embedded into the set  $\onemore{\covset}$ of covectors of an oriented matroid  on the set $E\cup\{e\}$ where $e$ is a new element that is not a loop, so that there is a unique atom $\overline e$  of $\flatpos(\onemore{\covset})$ containing $\{e\}$. Moreover, e.g. by~\cite[Section 10.1]{bjvestwhzi-93}, 
	 $$
	 \flatpos(\covset)=\flatpos(\onemore{\covset})\setminus \flatpos(\onemore{\covset})_{\geq \overline{e}}.
	 $$
	 Since $\flatpos(\onemore{\covset})$ is a geometric lattice~\cite[Proposition 4.1.13]{bjvestwhzi-93}, by~\cite[Theorem 3.2]{WW} $\flatpos(\covset)$ is a geometric semilattice.

\end{proof}

\begin{lemma}\label{flatrestriction} Let $\covset$ be the set of covectors of an AOM on the ground set $E$.
	\begin{itemize}
	\item[(1)] For every $F_1,F_2\in \flatset(\covset)$, $F_1\cap F_2\in \flatset(\covset)$.
	\item[(2)] For every $A\subseteq E$,
	$\flatset(\covset[A]) = \{ G\cap A \mid G\in \flatset(\covset)\}.$
	\item[(3)] For all $A\in \flatset(\covset)$, $\flatpos(\covset[A])=\flatpos(\covset)_{\leq A}$.
	\item[(4)] For all $A\subseteq E$, the assignment $X\mapsto \iota(X):=\min_{\leq}\{G\in \flatpos(\covset) \mid G\cap A = X \}$ defines an order preserving embedding $\iota: \flatpos(\covset[A])\hookrightarrow\flatpos(\covset)$. 
	\item[(5)] For all $A\subseteq E$ and all $F_1,F_2\in \flatpos(\covset[A])$, if the join $F_1\vee F_2$ exists in $\flatpos(\covset[A])$, then $\iota(F_1)\vee\iota(F_2)$ exists in $\flatpos(\covset)$.
	\end{itemize}
\end{lemma}
\begin{proof}
	In order to check (1) let $X_1,X_2\in\covset$ such that $F_i=\ze{X_i}$ for $i=1,2$. Then axiom (C) ensures $X_1\circ X_2\in \covset$. Thus $F_1\cap F_2=\ze{X_1\circ X_2}\in \flatset(\covset)$.

	Now, by definition, $F \in \flatset(\covset[A])$ if and only if there is $X\in \covset$ with 
	$F=\ze{X_{\vert A}}=\ze{X}\cap A$, 
	i.e., if and only if $F=G\cap A$ for some $G\in \flatset(\covset)$ (indeed, $G\in\flatset(\covset)$ if and only if there is $X\in \covset$ with $\ze{X}=G$). This proves (2). 
	Item (3) follows from (2) since, by (1), $\flatpos(\covset)_{\leq A} = \{G\cap A \mid G\in \flatset(\covset)\}$.

	The assignment in (4) is well-defined by (1) and clearly determines an order preserving function.
	 Its injectivity follows from the existence of the right-sided inverse $G\mapsto G\cap A$ (this is well-defined by (2)). 

	 For item (5) let $F_1,F_2$ have a join $F_1\vee F_2$ in $\flatpos(\covset[A])$. Then $\iota(F_1\vee F_2)\geq \iota (F_i)$ for $i=1,2$, hence the set of upper bounds of $\iota(F_1)$ and $\iota(F_2)$ is nonempty and, by (1), has a unique minimal element. 
\end{proof}

\begin{theorem-definition}\label{thm:undersemi}
	Let $(E,\covset)$ be an AOM satisfying (Z). Then $\flatpos(\covset)\subseteq 2^E$ is the geometric semilattice of flats of a finitary semimatroid $\SS(\covset)$. This underlying semimatroid is unique up to isomorphism. 
\end{theorem-definition}
\begin{proof}
	By definition $\flatpos(\covset)$ is partially ordered by inclusion. By (Z) $\flatpos(\covset)$ is a chain-finite subposet of $\PF(E)$. By Theorem~\ref{thm:gsl}, it is enough to prove that $\flatpos(\covset)$ is a geometric semilattice whose order ideals are lattices of flats of matroids.	 
	If $\vert E \vert <\infty$, the claim is proved as Lemma~\ref{lem:undersemifinite}. The proof for general $E$ follows from the fact that the obstructions to 
	the claim can be detected in a finite restriction. 
	
	More precisely,
	 first notice that $\flatpos(\covset)$ is bounded below, with $\bot_{\flatpos(\covset)}=\{e\in E \mid \forall X\in \covset\mid X(e)=0\}$, which is well-defined because topes exist (by Remark \ref{topexist}). Now checking that $\flatpos(\covset)$ is graded and that order ideals are lattices of flats of matroids (in particular, then it satisfies (GSL1)) amounts to checking statements about order ideals of the type $\flatpos(\covset)_{\leq U}$ which, by Lemma~\ref{flatrestriction}.(3), are isomorphic to order ideals of (finite) geometric semilattices $\flatpos(\covset[U])$, where the claims hold. 

The poset $\flatpos(\covset)$ is a meet-semilattice by Lemma~\ref{flatrestriction}.(1). In order to check (GSL2) let $A_1,\ldots,A_k$ be an independent set of atoms in $\flatpos(\covset)$ that joins to some $U=\vee A_i\in \flatpos(\covset)$ and let $Y\in \flatpos(\covset)$ have rank less than $U$. Notice that $A_1,\ldots,A_k$ is an independent set of atoms in $\flatpos(\covset[U\cup Y])$: indeed for all $I\subseteq [k]$ the inclusion $\vee_{i\in I} A_i \subseteq U$ implies $\vee_{i\in I}A_i=\iota(\vee_{i\in I}A_i)\in\flatpos(\covset[U\cup Y])$ by Lemma~\ref{flatrestriction}.(2,4).  Thus, by Lemma~\ref{lem:undersemifinite} in the finite restriction $\flatpos(\covset[U\cup Y])$ there is some $i$ such that $A_i\vee Y$ exists in $\flatpos(\covset[U\cup Y])$, and so by Lemma~\ref{flatrestriction}.(5) this join also exists in $\flatpos(\covset)$.
\end{proof}

\begin{corollary-definition}\label{centralunder} 
For every AOM $(E,\covset)$ satisfying (Z), the set $\CC(\covset)$ from Definition \ref{def:KL} is the set of central sets of the underlying semimatroid, i.e.,  $\CC(\covset)=\CC(\SS(\covset))$.
\end{corollary-definition}
\begin{proof}
	Immediate from Theorem~\ref{thm:gsl}.
\end{proof}

\begin{corollary}
	Let $(E,\covset)$ be an AOM satisfying (Z) and let $B$ be any basis of the semimatroid $\SS(\covset)$. For every $b\in B$ choose $b'\in \pi(b)$. Then $B':=\{b'\mid b\in B\}$ is a basis of $\SS(\covset)$.
\end{corollary}
\begin{proof}
	Induction on the rank of $\SS(\covset)$. If the rank is $0$ or $1$, there is nothing to prove. Let then $\SS(\covset)$ have rank $r>1$ and let  $B$, $B'$ be as in the claim. Choose $b_0\in B$. By induction hypothesis applied to $\covset[\bigcup_{b\neq b_0}\pi(b)]$, $B'\setminus b_0'$ is an independent set of $\SS(\covset)$ of rank $r-1$. Now apply (CR2) to the sets $B$  and  $B'\setminus b_0'$. Since every element of $B$ except $b_0$ is parallel to some element of $B'\setminus b_0'$, the only way for (CR2) to hold is that $B'$ is central in $\SS(\covset)$ (cf.\ \S\ref{sec:semimatroids} for terminology) and has rank $r$. 
\end{proof}

\begin{corollary}\label{corBo}
	Let $(E,\covset)$ be an AOM satisfying (Z) and let $B$ be any basis of the semimatroid $\SS(\covset)$. Then there is a unique $X_B\in \max\covpos (\covset)$ with $B\subseteq \ze{X_B}$. 
\end{corollary}
\begin{examplenew}\label{ex:xbwhite} In Figure \ref{fig:basis_inf} the elements $a_0$ and $b_0$ form a basis $B=\{a_0,b_0\}$ of the FAOM represented by the given pseudoline arrangement. The covector $X_B$ corresponds to the vertex highlighted in white.
\end{examplenew}
\begin{proof}[Proof of {Corollary \ref{corBo}}]
	Since $B$ is a basis, $\cl( B)$ exists and is a maximal element in $\flatpos(\covset)$. In particular, there is $X_B\in \max\covpos(\covset)$ with $B\subseteq \ze{X_B}$. In order to prove uniqueness consider any $Y\in \covset$ with $B\subseteq \ze{Y}$. If $Y\neq X$, there is some $e\in E$ with $Y(e)=-X(e)\neq 0$. Elimination of $e$ from $X$ and $Y$ would give a $Z\gneq X$ in $\covpos(\covset)$, contradicting maximality of $X$. 
\end{proof}

\subsection{Rank} We briefly compare the different notions of rank that have appeared so far and set some notation for the remainder of the paper.

\begin{proposition}
	The order reversing map $\ze{\cdot}: \covpos(\covset) \to \flatpos(\covset)$ from Remark~\ref{rem:zeromap} is rank-preserving.
\end{proposition}

\begin{proof}
In analogy with the proof of~\cite[Proposition 4.1.13.(ii)]{bjvestwhzi-93}, it is enough to consider two $X,Y\in\covset$ with $X\lessdot Y$ in $\covpos(\covset)$ and to prove $\ze{X}\gtrdot \ze{Y}$ in $\flatpos(\covset)$. Now, for such $X,Y$ clearly $\ze{X}\supsetneq \ze{Y}$. If there is a covector $Z\in \covset$ with $\ze{X}\supsetneq \ze{Z}\supsetneq \ze{Y}$, then we can in fact choose this covector to be in the upper interval $\covpos(\covset)_{\geq X}$ (take for instance $X\circ Z$), contradicting~\cite[Proposition 4.1.13.(ii)]{bjvestwhzi-93} for the OM $\covset[\ze{X}]$ (cf.\ Lemma~\ref{locOM}).
\end{proof}

\begin{definition}
	For any given FAOM $\covset$, we will henceforth write $\rk$ for both the rank function of $\covpos(\covset)$ and for the rank function of its underlying semimatroid. We write $\rk(\covset)=\rk(\SS(\covset))$ for the rank of either (i.e., the length of $\covpos(\covset)$ and $\flatpos(\covset)$).
\end{definition}

\begin{corollary} \label{minorrank}
Let $\covset$ be an FAOM and let $A\subseteq E$. Then, $\rk(\covset[A])=\rk(A)$ and $\rk(\covset/A)=\rk(\covset)-\rk(A)$.
\end{corollary}
\begin{proof} These identities are known on the semimatroid level, see, e.g., \cite[\S1.1]{DeluRiedel}
\end{proof}

\subsection{Order type of parallelism classes}
\label{ssec:ordertype}

In this section we show that  the isomorphism type of the natural total orderings of parallelism classes of an FAOM  (described in Corollary~\ref{totalorder})   is restricted. Recall the notation and conventions of Definition~\ref{def:totalorder}.

\begin{lemma}\label{ordertype} Suppose $\covset$ is the set of covectors of an AOM satisfying (S) and let $e\in E$.  Then $(\pi(e),\leq_{\pi(e)})$ is order isomorphic to $\mathbb Z$ if $e\in E^{**}$, to $\mathbb N$ if $e\in E^{0*}$, and to the segment $\{0,1,\ldots,n_{\vert\pi(e)\vert}\}\subseteq \mathbb N$ if $e\in E^{01}$.
%
\end{lemma}
\begin{examplenew} Recall the FAOMs of Figure \ref{examplesab}. In case (a) we have that $\vert \pi(l)\vert =1$. In both cases $\leq_{\pi(h_0)}$ has the order type of $\mathbb Z$. In case (b) the order type of $\leq_{\pi(l_0)}$ is $\mathbb N$
\end{examplenew}
\begin{proof}[Proof of Lemma \ref{ordertype}] Write $\pi$ for $\pi(e)$ and $<$ for $<_{\pi(e)}$. Recall that by Corollary~\ref{totalorder} we have that $\pi$ is a total order.

Given $x<y$ in $\pi$ choose covectors $X$ and $Y$  with $x\in \ze{X}$ and $y\in \ze{Y}$. Now any $z\in \pi$ with  $x< z< y$ must have $X(z)=-Y(z)\neq 0$ and thus $z\in S(X,Y)$. By (S) there are at most finitely many such $z$'s. Now, if $e\in E^{01}$ the claim follows by taking $x=\bot_\pi$, $y=\top_\pi$. Otherwise, for every $x\in \pi$ there is some $y\in \pi$ with $y>x$ and so the assignment
	$$
	s(x):=\min\{z\in \pi \mid x < z \leq y\}
	$$
	determines a well-defined ``successor'' function $s:\pi\to\pi$ (the set on the r.h.-s.\ is finite by the previous discussion, nonempty since it contains at least $y$, and independent on the choice of $y$ since  $s(x)$ is an immediate successor of $x$ and $\leq$ is a total order). Analogously, for every $y\in \pi$, $y\neq \bot_\pi$, we can find an element $s^{-1}(y):=\max\{z\in \pi \mid \bot_\pi \leq z < y\}$ so that $s^{-1}(s(x)) = x$ for all $x\in \pi$ and thus $s$ is injective.
	\begin{itemize}
		\item[]{\em Claim.} For every $x\in \pi$, the function $f: \mathbb N\to \pi_{\geq x}$, $n\mapsto s^n(x)$ is bijective.
		\item[]{\em Proof.} Injectivity of $s$ implies injectivity of $f$. In order to prove surjectivity, let $y\in \pi_{\geq x}$. As above, by (S) there is a finite number, say $k$, of  $z\in \pi$, $x<z\leq y$, therefore $s^k(x)=y$.
 	\end{itemize}
	If $e\in E^{0*}$, in the claim above we can choose $x=\bot_\pi$ and we obtain an order isomorphism between $(\pi,\leq _{\pi})$ and $\mathbb N$   with the natural order.
	If $e\in E^{**}$, the Claim above gives an order isomorphism between $\mathbb N$ and $\pi_{\geq e}$  and an order antiisomorphism between $\mathbb N$ and $\pi_{\leq e}$, that combine to an order isomorphism between $(\pi,\leq_{\pi})$ and $\mathbb Z$ with the standard ordering, where $e$ is mapped to $0$.
\end{proof}

\begin{corollary}\label{corsepdelta}
	Suppose that  $\covset$ is an AOM satisfying (S) and let $X\in \covset$. Then {there is a reorientation of $\covset$} such that for every parallelism class $\pi$, there is a unique element $\separ{X}{\pi} \in \pi$ such that, for every $e\in \pi$,
	$$
	e <_{\pi} \separ{X}{\pi} \Rightarrow X(e)=+,\quad\quad 	e >_{\pi} \separ{X}{\pi} \Rightarrow X(e)=-, \quad\quad X(\separ{X}{\pi})\in\{0,-\}
	$$
\end{corollary}
\begin{proof} {Consider the reorientation of $\covset$ given in Corollary~\ref{totalorder}}. 
First note that if $X(\pi)\subseteq\{0,-\}$ then $\separ{X}{\pi}:=\hat{0}_\pi$ will do. Otherwise, we prove that the maximum
	$$
	m_\pi(X):=\max_{\leq_\pi} \{e\in \pi \mid X(e) = +\}
	$$
	exists. For this, choose an $f\in \pi$ and any $Y\in \covset$ with $f\in \ze{Y}$. Then by Corollary~\ref{totalorder} we know that $Y(e)=+$ if $e<_{\pi}f$ and $Y(e)=-$ if $e>_{\pi}f$. Now assume by way of contradiction that the stated maximum $m_\pi(X)$ does not exist. Then, either $X(e)=-$ for all $e\in\pi$ (a case that we excluded at the beginning), or there are infinitely many $e>_{\pi}f$ with $X(e)=+$, violating (S) between $X$ and $Y$.

	Now, the element $\separ{X}{\pi}:=s(m_\pi(X))$, the successor of $\separ{X}{\pi}$, satisfies the claim.
\end{proof}

\begin{examplenew}\label{ex:delta}
    In Figure \ref{fig:deltas} with the pictured reorientation and choosing the ordering $l_i<_{\pi(l_0)}l_j$, resp $h_i<_{\pi(h_0)}h_j$ if and only if $i<j$, for the covector $X$ corresponding to the bold segment we have $\delta_X(\pi(h_0))=h_1$ and $\delta_X(\pi(l_0))=l_1$. 
    For the tope $C$ we have $\delta_C(\pi(h_0))=h_1$ and $\delta_C(\pi(l_0))=l_0$.
\end{examplenew}

\addtocontents{toc}{\protect\setcounter{tocdepth}{1}}

\section{Basis frames and embeddings into Euclidean space}
\label{sec:frames}

This section contains some fundamentals that will be needed at a later stage, and can be skipped in a first reading. 
The goal is to study the homeomorphism type of covector posets of FAOMs by comparing them with those of restrictions to unions of parallelism classes of elements of a basis.

\subsection{Covectors of restrictions to basis frames} Throughout this section suppose that $\covset$ is the set of covectors of an AOM on the ground set $E$ satisfying (S) and (Z). In particular, (Z) ensures that there is a well-defined underlying  semimatroid $\SS(\covset)$ (see Corollary \ref{centralunder}), while (S) implies that the canonical ordering of parallelism classes has the order type of subsets of $\mathbb Z$ (see Lemma \ref{ordertype}).

\begin{definition}[Basis frame]
	Let $B$ be a basis of the underlying semimatroid $\SS(\covset)$. The associated {\em basis frame} is
	$$
	\Bext := \bigcup_{b\in B} \pi(b) \subseteq E,
	$$
	the union of all parallelism classes of elements of $B$.
\end{definition}

\begin{examplenew} \label{ex:basisframe}
Recall from Example \ref{ex:xbwhite} that
the pseudolines labeled $a_0, b_0$ in Figure \ref{fig:basis_inf} form a basis $B=\{a_0,b_0\}$ of the associated FAOM. The pseudolines corresponding to the basis frame $\widetilde{B}$ for are drawn in a bold stroke in Figure \ref{fig:basis_inf}.
\end{examplenew}

Throughout this section let $B$ be a basis of the semimatroid $\SS(\covset)$. We start by describing an explicit model of the restriction $\covset[\Bext]$.  Without loss of generality, suppose that $(E,\covset)$ is reoriented so to satisfy Corollary~\ref{totalorder}. 
	For every parallelism class $\pi$ of $\covset$ we fix an order isomorphism
\begin{equation}\label{eq:defot}
	j_\pi:\pi\to 
	\left\{\begin{array}{ll}
	\mathbb Z &\textrm{ if } \pi\subseteq E^{**}\\
	\mathbb N &\textrm{ if } \pi\subseteq E^{0*}\\
	\{0,1,\ldots, {n_\pi-1}\} &\textrm{ if } \pi\subseteq E^{01}\textrm{ where } n_\pi:=\vert \pi \vert\\
	\end{array}\right.
\end{equation}

and an index set
\begin{equation}\label{eq:indexset}
	\mathbb I_\pi := \left\{\begin{array}{ll}
	\frac{1}{2}\mathbb Z &\textrm{ if } \pi\subseteq E^{**}\\
	\{-\frac{1}{2}\}\uplus\frac{1}{2}\mathbb N &\textrm{ if } \pi\subseteq E^{0*}\\
	\{-\frac{1}{2},0,\frac{1}{2},1,\ldots,{n_\pi}, , \frac{2n_\pi+1}{2}\} &\textrm{ if } \pi\subseteq E^{01}\textrm{ where } n_\pi:=\vert \pi \vert\\
	\end{array}\right.
\end{equation}

Now consider the product

\begin{equation}
\mathbb I (B):= \prod_{b\in B} \mathbb I_{\pi(b)}
\end{equation}
For every $i\in \mathbb I(B)$ we define a sign vector $X_i\in \signs^{\Bext}$ as
\begin{equation}
X_i(e):=
\left\{
\begin{array}{ll}
	+ & \textrm{ if } i({\pi(e)}) > j_{\pi(e)}(e) \\
	0 & \textrm{ if } i({\pi(e)}) = j_{\pi(e)}(e) \\
	- & \textrm{ if } i({\pi(e)}) < j_{\pi(e)}(e) \\		
\end{array}
\right.
\textrm{ for every }e\in \Bext.
\end{equation}
Moreover, our chosen reorientation is the one that yields Corollary~\ref{corsepdelta} and so for every $X\in  \signs^{\Bext}$  we can define a vector $i_X\in \frac{1}{2}\mathbb Z^{\pi(B)}$ as
\begin{equation}
i_X (\pi):=
\left\{
\begin{array}{ll}
j_{\pi}(\separ{X}{\pi}) & \textrm{ if } X(\separ{X}{\pi})=0\\
j_{\pi}(\separ{X}{\pi}) -\frac{1}{2} & \textrm{ if } X(\separ{X}{\pi})=-
\end{array}
\right.
\end{equation}

\begin{lemma}\label{invcheck} With the definitions above, the following hold.
\begin{itemize}
	\item[(1)] For every $X\in  \signs^{\Bext}$, $X_{i_X}=X$.
	\item[(2)] For every $i\in \frac{1}{2}\mathbb Z^{\pi(B)}$, $i_{X_i}=i$.
\end{itemize}
\end{lemma}

\begin{proof}
We check both identities elementwise. 
\begin{itemize}
	\item[(1)]
	Fix $e\in \Bext$. By definition, $X_{i_X}(e)= + $ if and only if
	$ i_X({\pi(e)}) > j_{\pi(e)}(e)$, i.e.,  $j_{\pi}(\separ{X}{\pi}) > j_{\pi(e)}(e)$ or, equivalently, $\separ{X}{\pi} > e$, and by Corollary~\ref{corsepdelta} the latter means $X(e)=+$. The cases $X_{i_X}(e)=0$ and $X_{I_X}(e)= - $ are analogous. 
	\item[(2)] Notice first that $j_\pi (\separ{X_i}{\pi})=\lceil i(\pi)\rceil$ for all $i$, with $j_\pi (\separ{X_i}{\pi})= i(\pi)$ if and only if $X_i(\separ{X_i}{\pi})=0$. Now we compute with the definition, for every parallelism class $\pi$:
	$$i_{X_i}(\pi)= \left\{
	\begin{array}{ll}
	\lceil i(\pi)\rceil = i(\pi) & \textrm{ if } X_i(\separ{X_i}{\pi})=0\\
	\lceil i(\pi)\rceil -\frac{1}{2} = i(\pi) & \textrm{ if } X_i(\separ{X_i}{\pi})=-
	\end{array}
	\right. $$
	\end{itemize}
\end{proof}

\begin{definition}\label{def:orderi}
	Define a partial order $\preceq$ on $\frac{1}{2}\mathbb Z$ by setting $p\prec q$ if and only if $p\in \mathbb Z$, $q=\pm\frac{1}{2}$. (See Figure \ref{figfence}.)
	This restricts to a partial order on $\mathbb I_\pi$, for every parallelism class $\pi$, and induces a partial order on $\mathbb I(B)$ by taking Cartesian product of the orderings of all $\mathbb I_\pi$.
	This ordering on $\mathbb I(B)$ we also call $\preceq$, and it can be explicitly described as
	$$
	i\preceq i' \textrm { if and only if } i_{\pi(b)}\preceq i'_{\pi(b)} \textrm{ for all }b\in B.
	$$
\end{definition}

\begin{figure}[h]
\begin{tikzpicture}
\foreach \x in {1,3,5}
    \node (t\x) at (\x,1) {};
\foreach \x in {0,2,4,6}
    \node (b\x) at (\x,0) {};    
\node[anchor=south] (t1l) at (t1) {$-\frac{1}{2}$};
\node[anchor=south] (t3l) at (t3) {$\frac{1}{2}$};
\node[anchor=south] (t5l) at (t5) {$\frac{3}{2}$};
\node[anchor=north] (b0l) at (b0) {$-1$};
\node[anchor=north] (b2l) at (b2) {$0$};
\node[anchor=north] (b4l) at (b4) {$1$};
\node[anchor=north] (b6l) at (b6) {$2$};
\draw (b0.center) -- (t1.center) -- (b2.center) -- (t3.center) -- (b4.center) -- (t5.center) -- (b6.center);
\draw[dotted] (b0.center) -- (-.5,.5);
\draw[dotted] (b6.center) -- (6.5,.5);
\end{tikzpicture}
\caption{The partial order defined in Definition \ref{def:orderi}.}\label{figfence}
\end{figure}

\begin{proposition}\label{prop:LIbij}
	The assignment
	$$
	\mathbb I (B) \to \covset[\Bext],\quad\quad i\mapsto X_i
	$$
	is a well-defined bijection with inverse $X\mapsto i_X$.
\end{proposition}
\begin{proof}
	We first prove that the assignment is well-defined. By Corollary~\ref{corBo}, if every coordinate of $i$ is an integer then $X_i\in \covset[\Bext]$. Otherwise,  there is an $i^*\preceq i$ with all integer coordinates, and we have $X_{i^*}\in \covset[\Bext]$, with $X_{i^*}(e)=X_{i}(e)$ if $e\not\in \ze{X_{i^*}}$. Now $B^*:=\ze{X_{i^*}}$ is a basis of $\covset[\Bext]$ and thus $\covset[B^*]$  is a Boolean (finite) oriented matroid, i.e., $\covset[B^*]=\signs^{B^*}$ and, in particular, 
	 there is $Y\in\covset[\Bext]$ with $Y_{\vert B^*}=(X_i)_{\vert B^*}$. Since $X_i=X_{i^*} \circ Y$, by (C) we conclude 
	$ X_i\in \covset[\Bext]$ as desired. This proves that the function is well-defined. The stated inverse is clearly well-defined, and Lemma~\ref{invcheck}.(2) shows that it is, in fact, an inverse.
	
\end{proof}

We use this bijection in order to decompose the covector poset into a Cartesian product.

\begin{theorem}\label{corIF}
	The function $i\mapsto X_i$ defines an isomorphism of posets 
	$$(\mathbb I(B),\preceq) \cong \covpos(\covset[\Bext]).$$ 
\end{theorem}

\begin{proof}
	We prove that both $i\mapsto X_i$ and $X\mapsto i_X$ are order preserving.
	
	First suppose $i\preceq i' $ in $\mathbb I(B)$. We prove  $X_i\leq X_{i'}$ componentwise.
	Consider a parallelism class $\pi$. If $i_\pi = i'_\pi$, then obviously $X_i(e)=X_{i'}(e)$ for all $e\in \pi$. Otherwise, (1) $i_\pi\in \mathbb Z$ and (2) $i'_\pi = i_\pi\pm \frac{1}{2}$. This means that, with $e:=j_\pi^{-1}(i_\pi)$, we have (1) $X_i(e)=0$  and (2) $X_{i'}(e)\neq 0$, with $X_{i'}(f)=X_i(f)$ for all $f\in \pi\setminus \{e\}$. Therefore $(X_i)_{\vert \pi} \leq (X_{i'})_{\vert \pi}$, and repeating the argument for every $\pi$ proves $X_i\leq X_{i'}$.
	
	Now suppose $X\leq Y$ in $\covpos(\covset)$, i.e.,  $X(e)\leq Y(e)$ for all $e\in E$. In particular, $X(e)=Y(e)$ whenever $X(e)\neq 0$. We prove the inequality $i_X\preceq i_Y$ componentwise. Fix a parallelism class $\pi$. If $X(e)\neq 0$ for all $e\in \pi$, then $i_X(\pi) = i_Y (\pi)$. Otherwise there is $e\in \pi$ with $X(e)=0$ (which means $i_X({\pi}) = j_{\pi}(e)$), and  by definition of parallelism, $X(f)\neq 0$ (hence $Y(f)=X(f)$) for all $f\in \pi\setminus\{e\}$. $Y(e)$ can now be $0$, $+$ or $-$ and, depending on those cases, we will have $i_Y ({\pi})= j_{\pi}(e)$, $i_Y ({\pi})= j_{\pi}(e) + \frac{1}{2}$, $i_Y({\pi})= j_{\pi}(e)-\frac{1}{2}$. In any case, $i_X ({\pi}) \preceq i_X({\pi})$.

\end{proof}

\begin{examplenew}
    If $\covset$ is the FAOM of the line arrangement of Figure \ref{fig:pseudoperiodic} and $B=\{a_0,b_0\}$ is, e.g., the basis described in Example \ref{ex:xbwhite}, then Theorem \ref{corIF} states the isomorphism between the FAOM $\covpos(\covset(\widetilde B))$ (represented by the bold-faced lines in Figure \ref{fig:basis_inf}) and   $(\mathbb I(B),\preceq)$, which is the FAOM of the arrangement in Figure \ref{examplesab}.(b). 
\end{examplenew}

\subsection{Embeddings into Euclidean space}
\label{ssec:embeddings}

Throughout this section let $\covset$ be the covector set of a simple FAOM of rank $d$ and let $B$ be a basis of the semimatroid $\SS(\covset)$ such that $B\subseteq E^{**}$. For every $\pi\in\pi(B)$ we fix an arbitrary order isomorphism $j_\pi:\pi\to \mathbb Z$ and, thus, an isomorphism $\covpos(\covset[\Bext])\to \mathbb I(B)$.

\begin{notation}
For brevity, in this section we  write $\fkz$ for $\covpos(\covset)$ and $\fkzb$ for $\covpos(\covset[\Bext])$.
\end{notation}

\begin{lemma} \label{homeo3}\label{L31}		
	There is an isomorphism of simplicial complexes $\Delta(\fkzb) \cong\Delta( \mathbb I(B))$ and, in particular, a homeomorphism 
	$\Vert \fkzb \Vert \cong\Vert \mathbb I(B) \Vert$.
\end{lemma}
\begin{proof} The poset-isomorphism from
	Theorem~\ref{corIF}  establishes a one-to-one correspondence between chains and, thus, an isomorphism of simplicial complexes. 
\end{proof}

\begin{lemma}\label{lem_concreteR}
	The natural inclusion $\mathbb I(B)\hookrightarrow \mathbb R^d$ extends affinely to a homeomorphism
	$$
	\Vert \mathbb I (B) \Vert \simeq \mathbb R^d.
	$$
\end{lemma}
\begin{proof}
	The elements of $ \mathbb I (B)$ are the barycenters of the cells of the cube complex $\mathcal Q$ of integer unit cubes, with set of vertices $\mathbb Z^d\subseteq \mathbb R^d$, and $i_1\preceq i_2 $ in $ \mathbb I (B)$ if and only if the cell with barycenter $i_1$ is contained in the cell with barycenter $i_2$. Therefore $\Vert \mathbb I (B) \Vert$ is realized as a geometric simplicial complex by the barycentric subdivision of $\mathcal Q$, see Remark~\ref{rem:faceposet}. In particular, the homeomorphism in the claim can be defined by affine extension of the inclusion of  $ \mathbb I (B)$ into $\mathbb R^d$.
\end{proof}

\begin{lemma}\label{lemresB}
	Let $\covset$ be the covector set of a simple FAOM and let $B$ be a basis of the semimatroid $\SS(\covset)$ such that $B\subseteq E^{**}$.  Consider the natural restriction map
	$$
	\phi: \covset \to \covset[\Bext],\quad X\mapsto X_{\vert\Bext},
	$$
	which induces an order preserving, surjective map $\fkz\to \fkzb$.
	For every $Y\in\fkzb$, the following hold.
	\begin{itemize}
	\item[(1)] 
	The preimage $\phi^{-1}(Y)$ is finite.
	\item[(2)] 
The poset $\fkz_{\leq \phi^{-1}(Y)}$ has length equal to $\rk(Y)$.
	\item[(3)] The order complex $\Vert \fkz_{\leq \phi^{-1}(Y)} \Vert$ is a subdivision of a $\rk(Y)$-dimensional PL-ball with boundary $\Vert \phi^{-1}(\fkzbsub_{<Y} )\Vert$.
\end{itemize}
\end{lemma}

\begin{proof} Let $\covset$, $B$ and $Y$ be as in the claim. Moreover, we assume that $\covset$ is reoriented as to satisfy Corollary~\ref{corsepdelta}. 
As all three items are claims regarding preimages of $\fkzbsub_{\leq Y}$, by passing to contractions and simplifying, it is enough to consider the case $Y\in \max \fkzb$, i.e., $Y$ is a tope of $\covset[\Bext]$. Then $\ze{Y}=\emptyset$ because $\covset$ is simple. Thus if $k \in \mathbb Z^{\pi(B)}$ is defined by $k_\pi:=j_\pi(\delta_Y(\pi))-1$ for all $\pi\in\pi(B)$, Corollary~\ref{corsepdelta} yields, for all $e\in E$,
	$$ Y(e)=+ \quad\textrm{ if }\quad j_{\pi(e)} (e) \leq k_\pi, \quad\quad  Y(e)=- \quad\textrm{ if }\quad j_{\pi(e)} (e) > k_\pi.$$
	In particular, since $B\subseteq E^{**}$, the minimal elements of $\fkzbsub_{\leq Y}$ are given as
	$$
	X_{k + \epsilon},\quad\quad\textrm{ where }
	\epsilon\textrm{ ranges in }\{0,1\}^{\pi(B)} 
	$$
	 and 
	 $Y= \bigcirc_{\epsilon\in\{0,1\}^{\pi(B)}} X_{k+\epsilon}$ (the composition taken in any order), as can be explicitly verified e.g., via the isomorphism of Theorem~\ref{corIF}. 
	 	
Now, for every $\epsilon\in\{0,1\}^{\pi(B)}$ let $B_\epsilon := \ze{X_{k + \epsilon}}$. Every $B_\epsilon$ is a basis of the semimatroid $\SS(\covset[\Bext])$, and hence also of $\SS(\covset)$, because $\Bext$ has maximal rank.
Thus by Corollary~\ref{corBo} for every $\epsilon$ there is a unique cocircuit $Z_\epsilon\in \min\covpos(\covset)$ with $\phi({Z_\epsilon}) = X_{k+\epsilon}$. See the top part of Figure \ref{FigBigLemma} for an illustration of the setup.

\begin{figure}[h]
\begin{tikzpicture}
\node [anchor=east] (big) at (-.5,0) {
\includegraphics[scale=.8]{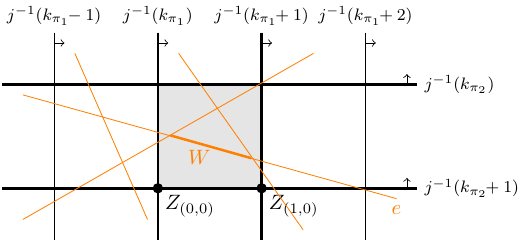}
};
\node [anchor=west] (B) at (.5,0) {
\includegraphics[scale=.8]{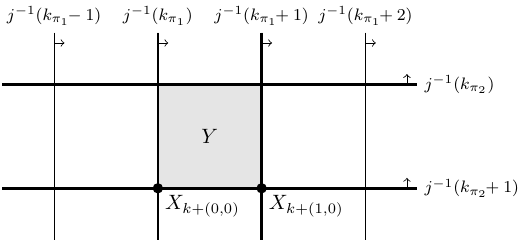}
};
\node [anchor=east] (big1) at (-.5,-6) {
\includegraphics[scale=.8]{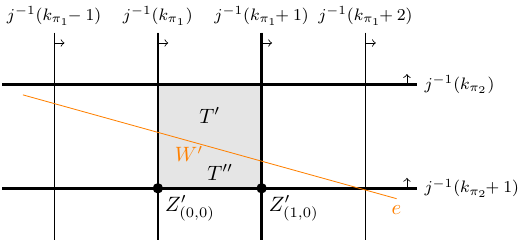}
};
\node [anchor=west] (Bb) at (.5,-6) {
\includegraphics[scale=.8]{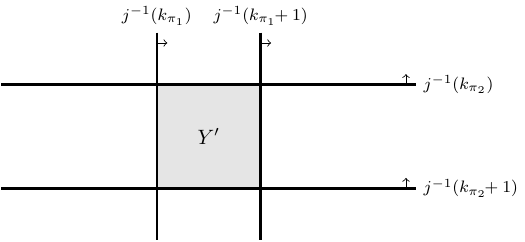}
};

\node [anchor=center] (L) at (-2.5,-2.2) {$\covset$};
\node [anchor=center] (LB) at (1.5,-2.2) {$\covset[\Bext]$};
\node [anchor=center] (LBe) at (-2.5,-3.8) {$\covset[\pareti\cup\{e\}]$};
\node [anchor=center] (Bbox) at (1.5,-3.8) {$\covset[\pareti]$};
\draw[ ->] (L.east) -- (LB.west);
\draw[ ->] (LBe.east) -- (Bbox.west);
\draw[ ->] (L.south) -- (LBe.north);
\draw[ ->] (LB.south) -- (Bbox.north);
\node [anchor=south] (phi) at (-.5,-2.2) {$\phi$};

\end{tikzpicture}
\caption{A picture for the setup of the proof of Lemma \ref{lemresB}.}
\label{FigBigLemma}
\end{figure}


	Let $\pareti=\bigcup_{\pi\in \pi(B) 
	}j_{\pi}^{-1}(\{k_\pi,k_\pi+1\})$.  
It is apparent that $\covset\left[\pareti\right]$ is the set of covectors of the arrangement given by the hyperplanes $x_{\pi}=k_{\pi}$ and $x_\pi=k_\pi+1$ in $\mathbb R^{\pi(B)}$ (``the facet-defining planes of a cube"). In this interpretation, the restriction $Y':=Y_{\vert \pareti}$  is the only bounded tope of $\covset\left[\pareti\right]$ (cf.\ Definition~\ref{bdfc}). 
 Now by Lemma~\ref{lem:bounded} we have that anything that maps to $Y'$ (in particular, $Y$ itself as well as every covector in $\covset_{\leq \phi^{-1}(Y)}$) is bounded in $\covset$.

\begin{itemize}
	\item[Ad (1).] We argue by induction on the rank of $\covset$. If the rank is $0$, then $B=\emptyset$ and the claim is trivial. Let then $\covset$ have rank $r>0$ and consider the following statement.
	\begin{itemize}
		\item[] {$(\dagger)$} For every $W\in\covset$ with $\phi(W)=Y$ and every $e\in\ze{W}$ we have that $e\in S(Z_{\epsilon'},Z_{\epsilon''})$ for some $\epsilon',\epsilon''\in\{0,1\}^{\pi(B)}$.
	\end{itemize}
	Let us first check that proving $(\dagger)$ is enough.
	First notice that the claim, together with the fact that there are finitely many $Z_\epsilon$, ensures immediately that there can be only finitely many $e\in E$ such that there is  $W\in\covset$ with $\phi(W)=Y$ and $\ze{W}=\{e\}$, 
	since otherwise axiom (S) would be contradicted for some pair of cocircuits of the form $Z_\epsilon$. Now, for every such $e\in E$, the induction hypothesis applied to the contraction $\covset[\Bext\cup\{e\}]/e$ shows that there are finitely many $W$ with $\phi(W)=Y$ and $\ze{W}=\{e\}$. This ensures that $\phi^{-1}(Y)$ has finitely many elements of rank $(r-1)$. 

By the second statement in Remark~\ref{thepowerofZ}, it follows that there can be at most finitely many topes in $\phi^{-1}(Y)$. 
	Moreover, (I) implies that there can be at most finitely many maximal elements below any given  $W\in\covset$. Thus $\phi^{-1}(Y)$ has finitely many elements. 
	\begin{itemize}
	\item[] {\em Proof of {$(\dagger)$}.} 
Let $W\in \covset$ be such that $\phi(W)=Y$, let $e\in\ze{W}$.
%
	%
	Then the image $W'$ of $W$ in $\covset\left[\pareti\cup\{e\} \right]$ is not maximal and we can choose $T'\in\covset\left[\pareti\cup\{e\} \right]$ maximal with $T'\geq W'$. 
	 Let then $T'':=W'\circ - T'$, which is a maximal element of $\covset\left[\pareti\cup\{e\} \right]_{\geq W'}$ by (FS). See the lower part of Figure \ref{FigBigLemma} for illustration.
	
	Then all of $W,T',T''$ map to $Y'$, and are therefore bounded. In particular, by~\cite[Proposition 3.7.2]{bjvestwhzi-93} we have that $T'$ and $T''$ can be obtained from $W'$ by composing with some cocircuits of $\covset\left[\pareti\cup\{e\} \right]$.	 Now the only cocircuits with nontrivial composition with $W'$ are those that are nonzero on $e$, hence exactly the images $Z'_\epsilon$ of the $Z_\epsilon$'s, for $\epsilon\in \{0,1\}^{\pi(B)}$. Since $T''(e)=-T'(e)\neq 0$, there must then be $\epsilon_1$ and $\epsilon_2$ with $Z'_{\epsilon_1}(e)=-Z'_{\epsilon_2}(e)\neq 0$, and hence $Z_{\epsilon_1}(e)=-Z_{\epsilon_2}(e)\neq 0$ as well.
	\end{itemize}
\end{itemize}
Since by Axiom (I) intervals in $\covset[\Bext]$ are finite, part (1) implies that $\covset_{\leq \phi^{-1}(Y)}$ is finite as well. 
	\def\cY{\mathscr P}
	Write for brevity $\cY:=\covpos(\covset)_{\leq \phi^{-1}(Y)}$. This is a convex set, 
	as it coincides with all 
	$X\in \covset$ such that
	$$(\ddagger)\left\{\begin{array}{ll}
		X(e)& \in \{0,+\}  \quad\textrm{ if } \quad e\in \Bext \quad \textrm{ and} \quad j_{\pi(e)}(e)=k_{\pi(e)},\\
		X(e)& \in\{0,-\} \quad\textrm{ if }\quad e\in \Bext \quad \textrm{ and} \quad j_{\pi(e)}(e)=k_{\pi(e)}+1.
	\end{array}\right.$$
	In particular, with Lemma~\ref{lem:finconv} we can find a finite $E'\subseteq E$ such that $\cY$ is an order ideal in in the poset of covectors of the finite affine oriented matroid $\covset[E']$.  We can then extend $\covset[E']$ to the set $\mathscr O$ of vectors of a (finite) oriented matroid with one additional element $g$ (see Remark~\ref{rem:ginf}). In $\mathscr O$, $\cY$ is a convex set defined by the above equations, in addition to $X(g)=+$. Such a convex set of covectors, with the usual partial order, is known to be a shellable poset~\cite[Proposition 4.3.6]{bjvestwhzi-93}. Moreover, since $\cY$ it is not the full poset of covectors of a contraction of $\mathscr O$ (e.g., because $\cY\supseteq\phi^{-1}(Y)$ and the latter contains at least one tope, since $\phi$ is order preserving), again by~\cite[Proposition 4.3.6]{bjvestwhzi-93} we conclude that
	\begin{center}
		$(\dagger\dagger)$ the order complex of $\cY$ is is a shellable PL-ball of dimension $r$.
	\end{center}
\begin{itemize}		
	\item[Ad (2).] The length of a poset equals the dimension of its order complex, hence (2) follows directly from $(\dagger\dagger)$.
	\item[Ad (3).]  The first part of claim (3) is exactly $(\dagger\dagger)$. From there, via Proposition~\ref{propthin}, we also deduce that $\wtb{\cY}$ is a subthin poset. We are thus left with showing that the maximal simplices in the boundary of $\Vert\cY\Vert$ are exactly given by the maximal chains in $\cY\cap \phi^{-1}(\covset[\Bext]_{<Y})$.
	
	To this end, again by Proposition~\ref{propthin}, it is enough to consider any $X\in \wtb{\cY}$ such that $[X,\top]$ has length $2$, and to prove that $[X,\top]$ has $3$ elements if and only if $\phi(X)<Y$. Notice that in this case, since $\mathscr O$ is simple, there is a single element $e\in E'$ with $\ze{X}=\{e\}$. Moreover, from the fact that $\mathscr O$ is an oriented matroid, $\covset[E']$ is an order filter in $\covpos(\mathscr O)$, and $X\in \covset[E']$, we deduce that there are exactly two elements $Z,Z'\in \covset[E']$ with $X\lessdot Z$, $X\lessdot Z'$, and these are given by $Z(f)=Z'(f)=X(f)$ for all $f\neq e$, and $Z(e)=-Z'(e)\neq 0$. Now, if $\phi(X)=Y$, then $Z,Z'\in \cY$ (for instance, in this case $\phi(Z)=\phi(Z')=Y$, thus $e\not\in \Bext$ and so $Z$ and $Z'$ both satisfy the definition of the convex set $\cY$ in $(\ddagger)$) and so $[X,\top]$ has $4$ elements. On the other hand, if $\phi(X)<Y$ then $e\in \Bext$ and thus only one among $Z$, $Z'$ agrees with $Y$ on $e$ and satisfies $(\ddagger)$, so that in this case $[X,\top]$ has $3$ elements. 
	\end{itemize}
\end{proof}

\begin{proposition}\label{homeo1}
	Let $\covpos$ be an FAOM 
	 and let $B$ be a basis of the semimatroid $\SS(\covset)$ such that $B\subseteq E^{**}$.
	The canonical restriction map $\covset \to \covset[\Bext]$ induces  a  homeomorphism  
	$$\psi : \Vert \covpos \Vert \to \Vert \fkzb\Vert $$
	that restricts to a homeomorphism $\Vert \covpos_{\leq\phi^{-1}{(Y)}} \Vert \cong \Vert \fkzbsub_{\leq Y}\Vert$ for all $Y\in \covset[\Bext]$.
\end{proposition}
\begin{proof}
We prove the statement by constructing the desired homeomorphism recursively in the rank of $\fkzb$. Let then $\fkzb^j$ be the 
 set of elements of $\fkzb$ of rank at most $j$. 
For every $Y\in\fkzb^0$ we have   $\vert \phi^{-1}(Y)\vert =1$, hence $\phi $ restricts to a homeomorphism $\psi^0$ between the discrete spaces $\Vert\phi^{-1}(\fkzb^0)\Vert$ and $\Vert \fkzb^0\Vert$ that again restricts to a homeomorphism between the one-point spaces $\fkzbsub_{\leq Y}$ and $\phi^{-1}(Y)$  for every $Y\in \fkzb^0$.  Let then $j>0$ and suppose we already constructed a homeomorphism $\psi^{j-1}: \Vert \fkz_{\leq \phi^{-1}(\fkzb^{j-1})}\Vert \to \Vert \fkzb^{j-1}\Vert$ that restricts to a homeomorphism $\Vert \fkz_{\leq \phi^{-1}{(Y)}} \Vert \cong \Vert \fkzbsub_{\leq Y}\Vert$ for all $Y\in \fkzb^{j-1}$. Consider any $Y\in \fkzb^j\setminus \fkzb^{j-1}$. Then by Lemma~\ref{lemresB} the poset $\fkzbsub_{\leq Y}$ is homeomorphic to a $j$-ball whose boundary $\Vert\covpos_{< Y}\Vert\subseteq \Vert\covpos^{j-1}\Vert$ is the homeomorphic image under $\psi^{j-1}$ of $\Vert\fkz_{< \phi^{-1}(Y)} \Vert$, a subcomplex of $\Vert \fkz_{\leq \phi^{-1}(\fkzb^{j-1})}\Vert$. We can now extend the boundary homeomorphism $\psi^{j-1}$ across $\Vert \fkz_{\leq \phi^{-1}(Y)} \Vert$  to a homeomorphism $\Vert \fkz_{\leq \phi^{-1}(Y)} \Vert \to \Vert \fkzbsub_{\leq Y} \Vert$ of $j$-balls. See, e.g.,~\cite[Chapter~1.4, Corollary~13]{Spanier}.
 Doing this inductively over the countable (by Corollary \ref{cor:countable}) set  $\fkzb^j\setminus \fkzb^{j-1}$ we can extend $\psi^{j-1}$ to the desired $\psi^{j}$.
\end{proof}

\section{Group actions}
\label{sec:gas}

If a group $G$ acts on a set $E$ by permutations, for every $g\in G$ and $e\in E$ we write $g(e)$ for the image of $e$ under the action of $g$. Moreover, for every sign vector $X\in\signs^E$ we can define a sign vector $g.X$ by setting $g.X(e):=X(g^{-1}(e))$ for all $e\in E$. This extends the action of $G$ on $E$ to an action on the set of sign vectors. For $\mathscr X\subseteq \signs^E $ write $g.\mathscr X:=\{g.X\mid X\in \mathscr X\}$.


\begin{definition}\label{def:sliding}
	A group $G$ acts on an AOM $(E,\covset)$ if $G$ acts by permutations on $E$ and $g.\covset=\covset$ for all $g\in G$. An action of $G$ on $\covset$ will be denoted by $\alpha: G\circlearrowright \covset$. The action of $G$ is called \emph{\pcp} if $g(e)\in \pi(e)$ for all $e\in E$.
\end{definition}

\begin{examplenew} The notion of \pcp actions is meant to model the action of a group of translations on an arrangement. For instance consider the pseudoline arrangement of 
Figure \ref{fig:basis_inf}. The action that corresponds to a vertical "translation" of the picture sending $a_i$ to $a_{i+1}$ is \pcp. On the other hand, the reflection with respect to a vertical line passing through, e.g., the point where $b_0$ meets $a_3$ is not \pcp, since it switches pairs of diagonal pseudolines passing through that point, and so any of those will not be parallel to its image.
\end{examplenew}

\begin{lemma}
	Let $(E,\covset)$ be an AOM and let $\tau\in \{+,-\}^E$. Every $G$-action on $(E,\covset)$ induces a $G$-action on the reorientation $\covset^{(\tau)}$. If the action on $\covset$ is \pcp, so is the induced action on $\covset^{(\tau)}$.
\end{lemma}
\begin{proof}
	Given any group action of $G$ on $\covset$ and any $h\in \{+,-\}^E$, $G$ acts on the reorientation $\covset^{(\tau)}$ by $g.(\tau\cdot X)=\tau\cdot g.X$, for all $g\in G$ (recall Definition~\ref{def:reorientation}). 
	The claim about \pcp actions follows because parallelism is preserved under reorientation, as $\ze{X}=\ze{\tau\cdot X}$ for all $X\in \covset$. 
\end{proof}

\begin{lemma}\label{lem:alphaminor}
	Let $(E,\covset)$ be an AOM with an action $\alpha$ of a group $G$. Let $A\subseteq E$.
	\begin{itemize}
	\item[(1)] The action $\alpha$ induces an action $\alpha_{/A}: \stab (A)\circlearrowright\covset/A$. 
	\item[(2)] If GA=A, then $\alpha$ restricts to an action $\alpha_{[A]}: G\circlearrowright \covset[A]$.
	\end{itemize}
	If $\alpha$ is \pcp, then so are $\alpha_{/A}$ and $\alpha_{[A]}$.
\end{lemma}
\begin{proof}
	The check of the claims is straightforward.
\end{proof}

\begin{lemma}\label{lem:ta}
	Every action of a group $G$ on  an FAOM $(E,\covset)$ induces an action of $G$ on the underlying semimatroid $\SS(\covset)$. If the action is \pcp, the action on $\SS(\covset)$ is translative. 
\end{lemma}
\begin{proof}
	Let $e\in E$ and $g\in G$. Then, the \pcp property implies $e\parallel g(e)$. In particular, with Corollary~\ref{centralunder} we have  $\{e,g(e)\}\not\in\CC(\covset)$ whenever $e\neq g(e)$, and the claim follows.
\end{proof}

Every action $\alpha: G\circlearrowright \covset$ induces an action $\alpha: G\circlearrowright \covpos(\covset)$ by poset automorphisms.  We will now begin the study of quotients of this poset. Recall that every poset can be seen as an acyclic category (see Appendix~\ref{sec:ACGA}).

\begin{definition}
 Let $\alpha: G\circlearrowright \covset$ be a group action on an AOM. Let
$$
q_\alpha: \covpos(\covset) \to \covpos(\covset)\qc G
$$
be the quotient functor in the category $\AC$. Explicitly, if we regard $G$ as a one-object category $\mathcal G$ with one morphism for each group element and the group operation as composition law, then $\alpha$ defines a diagram in $\AC$ over $\mathcal G$, and $q_\alpha$ is the universal co-cone over this diagram with target the diagram's colimit $\covpos(\covset)\qc G$. 
\end{definition}
\begin{remark}\label{remexplquot}
	Since, by Corollary~\ref{Franked}, $\covpos(\covset)$ is ranked,  the category $\covpos(\covset)\qc G$ can be described explicitly via~\cite[Lemma A.18]{DePag}. It has object set $\Ob{\covpos(\covset)\qc G}=\Ob{\covpos(\covset)}/G$, the set of orbits of objects. The morphisms of $\covpos(\covset)\qc G$ are orbits of morphisms of $\covpos(\covset)$: the orbit of $\phi:X\to Y$ is a morphism $G\phi:GX\to GY$ and composition between orbits $G\phi$ and $G\psi$ is defined as the orbit of the composition of $\phi$ and $\psi$, when it exists.
\end{remark}

\begin{lemma}\label{lem:contractionaction}
	Let $\alpha: G\circlearrowright \covset$ be a group action on an AOM.  The following hold for every $A\in \central(\covset)$ and $g\in G$.
	\begin{itemize}
	\item[(1)] 	$\covset/g(A) = g^{-1} (\covset/A)$ and hence $q_\alpha(\covset/A)=q_\alpha(\covset/g(A))$.
	\item[(2)]  If $g(A)\neq A$, then $\covset/A \cap \covset/g(A) =\emptyset$.
	\item[(3)] There is a canonical isomorphism of categories $q_\alpha (\covpos(\covset/A))\cong q_{\alpha_{/A}}(\covpos(\covset/A))$.
	\item[(4)] There is a canonical isomorphism of categories $q_\alpha (\covpos(\covset/A))\cong q_\alpha (\covpos(\covset/\cl(A)))$.
	\end{itemize}
\end{lemma}
\begin{proof}
	The first statement can be checked straightforwardly. The second follows from Lemma~\ref{lem:ta}, as any $X\in \covset/A \cap \covset/g(A)$ would have $\ze{X}\supseteq A\cup g(A)$, but translativity of the induced action on $\SS(\covset)$ implies that $A\cup g(A)$  can only be a central set if $g(A)=A$.
	
	For the third statement, we regard $\covpos(\covset)$ as an acyclic category and we denote by $X\leq Y$ the unique morphism between any two $X,Y\in \covpos(\covset/A)$, if one exists.
	Now the assignment
	\begin{equation}\label{eq:mormap}
	\Mor{q_\alpha(\covpos(\covset/A))} \to \Mor{q_{\alpha_{/A}}(\covpos(\covset/A))}, \quad G(X\leq Y)\mapsto \stab(A)(X\leq Y)
	\end{equation}
	is well-defined since any representative $(X'\leq Y')\in G(X\leq Y)$ has the form $(X'\leq Y')=g(X\leq Y)$ for some $g\in G$, and if $g\not\in\stab(A)$ then $g.X, g.Y\not \in \covpos(\covset/A)$ by part (2).
	The proof of item (3) is complete by noticing that \eqref{eq:mormap} describes the inverse of the natural functor $q_{\alpha_{/A}}(\covpos(\covset/A))\to q_{\alpha}(\covpos(\covset/A))$ induced by the inclusion $\stab(A)\subseteq G$.
	
	For item (4), notice that for translative actions we have $\stab(A)=\stab(\cl(A))$ (see, e.g.,~\cite[Lemma 8.1]{DeluRiedel}) and the isomorphism $\covset / A \cong \covset / \cl(A)$ (see Remark {\ref{remsimplicity}}) is equivariant. This induces a natural isomorphism of diagrams, and 	hence an isomorphism of quotients. Alternatively, (4) can be verified explicitly via Remark~\ref{remexplquot}. 
\end{proof}

\subsection{Actions on parallelism classes}

\begin{lemma}\label{lem:actionorbit}
	If the action of $G$ on $(E,\covset)$ is \pcp, then the restriction of the action on each parallelism class $\pi$ is an action by order isomorphisms of the ordered set $(\pi,<_\pi)$.
\end{lemma}

\begin{proof}
	Let $\pi$ be a parallelism class of $(E, \covset)$. Since the definition of $<_\pi$ is independent on reorientation, by passing to a suitable reorientation of $\covset$ (as in Corollary~\ref{totalorder}) and recalling Remark~\ref{constantsign} we can assume that $e<_\pi f$ if and only if $\sigma_{f}(e)=+$, for all $e,f\in \pi$. 
	Now let $g\in G$ and consider any two $X,Y$ with $X(g(f)) = Y(g(f))=0$. This means $g^{-1}.X (f) = g^{-1}.Y(f) =0$ and, since $g^{-1}.X,g^{-1}.Y\in g^{-1}.\covset=\covset$, the order relation $e<_\pi f$ implies $g^{-1}.X(e)=g^{-1}.Y(e)=\sigma_f(e)=+$. Since $X(g(e)) = g^{-1}.X(e)$ and $Y(g(e)) = g^{-1}.Y(e)$, we conclude $\sigma_{g(f)}(g(e))=X(g(e)) = Y(g(e))=+$ and so $g(e)<_\pi g(f)$.
	
	Now the assignment $e\mapsto g^{-1}(e)$ from $\pi$ to $\pi$ is order preserving as well, and it is an inverse to $e\mapsto g(e)$. The latter is thus an order isomorphism.
\end{proof}

\begin{proposition}\label{prop:jaction}
	Let $\pi$ be a parallelism class of an FAOM $(E,\covset)$ with a \pcp action of a group $G$ and recall  the  order isomorphism  $j_\pi$ from~\eqref{eq:defot}. 
	For every $g\in G$ there is $k(g)_\pi \in\mathbb Z$ such that $j_{\pi} (g(e))= k(g)_\pi+ j_{\pi} (e)$ for all $e\in \pi$. If $\pi\not\subseteq E^{*,*}$ then $k(g)_\pi=0$ for all $g$, and so $G$ acts trivially on $\pi$.
\end{proposition}

\begin{proof}
  Let $g\in G$ and $\pi$ a parallelism class.  By Lemma~\ref{lem:actionorbit} the map $g$ is an order isomorphism on $\pi$. If $\pi\subseteq E^{0,1}\cup E^{0,*}$ then this implies that $g$ acts as the identity on $\pi$. Otherwise, under the  order isomorphism  $j_\pi$ from $(\pi,<_\pi)$ to $\mathbb Z$ the map $g$ carries over to an order automorphism of $\mathbb Z$, and those are exactly the maps ``addition by a constant".
\end{proof}

\begin{remark}
	The function $k: G \to \mathbb Z^{\pi(E)}$ defined by $g\mapsto k(g)_{\ast}$ is a group homomorphism.
\end{remark}

\begin{corollary}\label{coractiononI}
	The isomorphism $\fkzb\to \mathbb I(B)$ from Theorem~\ref{corIF} is equivariant with respect to the induced $G$-action on the left, and the $G$-action on $\mathbb I(B)$ defined by $i\mapsto i+k(g)$ for every $g\in G$.
\end{corollary}
\begin{proof}
	The stated poset isomorphism maps any $X\in \fkzb$ to $i_X\in \mathbb I(B)$. Now, for every $g\in G$ and every parallelism class $\pi$ we have $\separ{g.X}{\pi}=g(\separ{X}{\pi})$ (recall Corollary~\ref{corsepdelta}) and therefore
	\begin{align*}
	i_{g.X} (\pi) &=
	\left\{
\begin{array}{ll}
j_{\pi}(g(\separ{X}{\pi})) & \textrm{ if } g.X(\separ{g.X}{\pi})=0\\
j_{\pi}(g(\separ{X}{\pi})) -\frac{1}{2} & \textrm{ if } g.X(\separ{g.X}{\pi})=-
\end{array}
\right.\\
	&=\left\{
\begin{array}{ll}
j_{\pi}(\separ{X}{\pi}) + k(g)_\pi  & \textrm{ if } X(\separ{X}{\pi})=0\\
 j_{\pi}(\separ{X}{\pi}) -\frac{1}{2} + k(g)_\pi& \textrm{ if } X(\separ{X}{\pi})=-
\end{array}
\right\}
=i_{X} (\pi) + k(g)_\pi.
	\end{align*}
\end{proof}

\subsection{Topological aspects}

We now consider the topology of quotients of covector posets of AOMs under a group action. A preliminary remark is in order:

\begin{definition} If a group $G$ acts on an FAOM $(E,\covset)$, let
	 $$Q_\alpha:\Vert \covpos(\covset) \Vert\to \Vert \covpos(\covset) \Vert / G  $$
	  denote the topological quotient map. 
\end{definition}

We will study in particular the case where the quotient map is a topological cover. This happens for instance if the group action is free.

\begin{definition}
	Call an action $\alpha: G\circlearrowright \covset$ {\em free} if the induced action on $\covset$ is free.
\end{definition}

\begin{lemma}\label{rem:quotiso} If the action $\alpha$ is free, then $q_\alpha$ is a regular covering of acyclic categories, 
$Q_\alpha = \Vert q_\alpha \Vert$ is a topological covering map and there is an isomorphism of cell complexes $$\Vert q_\alpha(\covpos(\covset)) \Vert \cong Q_\alpha(\Vert \covpos(\covset)\Vert).$$  
\end{lemma}

\begin{proof}
	By definition, any action of $G$ on $(E,\covset)$ induces an action by rank-preserving automorphisms on $\covpos(\covset)$. The Lemma then follows with~\cite[Lemma A.19, Corollary A.20]{DePag}.
\end{proof}

\begin{theorem}\label{tpa}
	Let $(E,\covset)$ be a nonempty FAOM with a distinguished basis $B\subseteq E^{*,*}$, and let  a free abelian group $G\cong \mathbb Z^{\vert B \vert}$ act on $(E,\covset)$ so that the induced action on $\covset[\Bext]$ is free. If the action is \pcp, then $\Vert \covpos(\covset)\qc G \Vert$ is homeomorphic to the $\vert  B \vert$-torus $(S^{1})^{\vert B \vert}$.
\end{theorem}
\begin{proof}
	First note that, by Remark~\ref{rem:quotiso}, there is a homeomorphism $\Vert \covpos(\covset) /G\Vert\simeq\Vert \covpos(\covset) \Vert/G$.
	Moreover, the restriction map $\covset \to \covset[\Bext]$ is $G$-equivariant, and so is the homeomorphism of  nerves that it induces (see Proposition~\ref{homeo1}). Therefore the quotients $\Vert \covpos(\covset) \Vert/G$ and $\Vert \fkzb \Vert/G$ are homeomorphic. Analogously we have a homeomorphism with $\Vert\mathbb I(B)\Vert/G$, where the action is the one described in Corollary~\ref{coractiononI}. 
	
	Now Lemma~\ref{lem_concreteR} gives a concrete realization of $\Vert \mathbb I(B)\Vert$ as $\mathbb R^{\vert B \vert}$, and on this realization the $G$-action described in Corollary~\ref{coractiononI} corresponds to the action of the subgroup $k(G)$ of the group of integer translations $\mathbb Z^{\vert B \vert}$. Since stabilizers of $k(G)$ are trivial (by the  freeness assumption), the subgroup $k(G)$ is full-rank, and so $\mathbb R^{\vert B \vert}/k(G)$ is homeomorphic to a ${\vert B \vert}$-torus, as was to be proved.
\end{proof}

\begin{corollary}\label{cor:tpa}
	Let $\alpha: G\circlearrowright \covset$, let $B$ be as in the statement of Theorem~\ref{tpa}, and let $A\in \central(\covset)$. Then, 
	$
	Q_\alpha(\Vert \covpos(\covset/A)\Vert)
	$
	is homeomorphic to a $(\vert B\vert - \rk(A))$-dimensional torus.
\end{corollary}
\begin{proof}
	We check that the action $\alpha_{/A}$ satisfies the hypotheses of Theorem~\ref{tpa}; the claim about the dimension follows with Corollary~\ref{minorrank}. The action is \pcp by Lemma~\ref{lem:alphaminor} and free, e.g., by Remark~\ref{rem:injection_contraction}. We claim that there is a basis $B'$ of $\covset/A$ with $B'\subseteq B$. In fact, by (CR2) in the semimatroid $\SS(\covset)$ any maximal independent subset $I\subseteq A$ can be completed to a basis $J$ of $\SS(\covset)$ using elements from $B$, and it is enough to take $B':=J\setminus I$. Now for every $b'\in B'$ and every $b''\in \pi(b')$  we have, again by  (CR2), that $A\cup\{b'\}\in \central(\covset)$. Moreover, $b'\not\in \cl(A)$ implies $A\cup\{b''\}\in \central(\covset)$ and $b''\not\in \cl(A)$. Therefore $b'' \parallel b'$ in $\covset/A$, and so the parallelism class of $b'$ in $\covset /A$ contains the parallelism class of $b'$ in $\covset$. In particular, $B' \in (E\setminus A)^{*,*}$. We are left to check that $\stab(A)$ acts freely on $(\covset/A)[\widetilde{B'}]$ which is straightforward.

\end{proof}

\subsection{Toric pseudoarrangements}\label{sec:tpa}
\newcommand{\sT}{\mathscr T}

Throughout this section let $\alpha: G\circlearrowright \covset$ denote a free and \pcp action of a finitely generated free abelian group $G$ on an FAOM $\covset$, and suppose that $\SS(\covset)$ has a basis $B\in E^{*,*}$.

\begin{definition}\label{defarr} 	
	Via the embedding $\iota_{\{e\}}$ of Remark~\ref{rem:injection_contraction} we can identify $\Vert \covpos(\covset/e)\Vert$ with the subcomplex
$$\HH_e:= \Vert \iota_{\{e\}}(\covpos(\covset/e))\Vert \subseteq \Vert \covpos(\covset) \Vert.$$ 

	Let $T:=\Vert \covpos(\covset)\Vert/G$ and, for every $a\in E/G$, let
	$$
	\sT_a := Q_\alpha \left(\HH_e\right),
	$$
	where $e$ is any representative of $a$ (this is well-defined by Lemma~\ref{lem:contractionaction}.(1)).
	We then let $$\AA_\alpha:=\{\sT_a \mid a\in E/G\}.$$ 
	We call any such arrangement of subspaces of the torus $T$ a {\em toric pseudoarrangement}. If $\SS(\covset)$ has no loops, $\AA_\alpha$ is called {\em proper}
\end{definition}

\begin{examplenew}
Figure \ref{fig:toricpseudo} shows the toric pseudoarrangement obtained by quotienting the pseudoarrangement of Figure \ref{fig:pseudoperiodic} by the action of $\mathbb Z^2$ by translations whose fundamental region is shaded in Figure \ref{fig:pseudoperiodic}.
\end{examplenew}

\begin{lemma}\label{lem:QCW}
	Every toric pseudoarrangement $\AA_\alpha$ as in Definition~\ref{defarr} defines a CW-complex structure $\rcw(\AA_\alpha)$ on $T$, with one cell for every object in $q_\alpha(\covpos(\covset))$.

	If $\AA_\alpha$ is proper, the union $\cup\AA_\alpha$   is the $(d-1)$-skeleton of $\rcw(\AA)$, here $d=\dim T=\rk(\covset)$. In particular, the complement of $\cup\AA_\alpha$ in $T$ is a union of open $d$-cells.
\end{lemma}
\begin{proof}
	 For every $GX\in \Ob{q_\alpha(\covpos(\covset))}$, by~\cite[Lemma A.19]{DePag} freeness of the action implies that the slice category below $GX$ (see~\cite[Definition A.3]{DePag}) 
	 is isomorphic to the principal order ideal $\covpos(\covset)_{\leq X}$ whose geometric realization is a $\rk(X)$-ball $B_X$ with boundary $\Vert \covpos(\covset)_{< X}\Vert$. Now~\cite[Lemma A.4]{DePag} implies that the interior of $B_X$ includes homeomorphically into an open subset $U_{GX}$ of $T$. We have $U_{GX}\cap U_{GY} = \emptyset$ whenever $GX\neq GY$, and clearly $T$ is the union of all $U_{GX}$. We claim that this is a CW-structure on $T$, and we prove it by checking the conditions given in~\cite[\S 38]{Munkres}.
	 
	 First, since nerves of categories are defined as CW-complexes, the total space $T$ is Hausdorff. The required continuous map of $B_X$ into $T$ carrying the interior of $B_X$ to $U_{GX}$ and the boundary to a finite union of cells is the canonical map from the realization of the slice below $GX$ to $T$ (see~\cite[Definition A.3]{DePag}, where this map is called $j_{GX}$). Moreover, the closure $\overline{U_{GX}}$ of $U_{GX}$ in $T$ is a subcomplex of the geometric realization $\Vert q_\alpha(\covpos(\covset))\Vert$ (i.e., the  the image of $j_{GX}$ in~\cite[Definition A.3]{DePag}), and since the latter is a CW-complex, it follows that any $V\subseteq T$ is closed  if $V\cap \overline{U_{GX}}$ is closed in $\overline{U_{GX}}$.
	
	From this description follows in particular that the $k$-skeleton of $\rcw(\AA_\alpha)$ is the union of all $U_{GX}$ with $\rk(X)\leq k$. Now,  $X$ has maximal rank if and only if it is a tope and, if $\covset$ has no loops, being a tope implies $\ze{X}=\emptyset$, hence $X\not\in\covpos(\covset/e)$ for all $e\in E$. Conversely, any $X$ that is not of maximal rank is contained in some $\HH_e$, thus $\cup\AA_\alpha$ is exactly the $(d-1)$-skeleton of $\rcw(\AA_\alpha)$, and the complement of $\cup\AA_\alpha$ in $T$ is the disjoint union of all $U_{GX}$ where $X$ is a tope. 
\end{proof}

The topological tameness expressed by Lemma~\ref{lem:QCW} allows us to apply Zaslavsky's theory of topological dissections in order to enumerate the open cells constituting the complement of a toric pseudoarrangement $\AA_\alpha$. The stepping stone is determining the poset of connected components of intersections of $\AA_\alpha$.

\begin{proposition}\label{prop_poslayers}
	The poset of connected components of intersections of $\AA_\alpha$ is isomorphic to the quotient poset $\flatpos(\covset)/G$. Moreover, every intersection is topologically a torus $(S^1)^{d-r}$ where $r$ is the rank of the corresponding element in  $\flatpos(\covset)/G$.
\end{proposition}

\begin{proof} Let $A:=\{a_1,\ldots,a_m\} \subseteq E/G$ and write $\CC=\CC(\covset)$ for the set of central sets of the underlying semimatroid (see Definition~\ref{centralunder}). Then
$$
Q_G^{-1}(\sT_{a_1} \cap\ldots\cap \sT_{a_m}) 
= \bigcup_{e_1\in a_1,\ldots,e_m\in a_m} \HH_{e_1}\cap\ldots \cap \HH_{e_m}
= \bigcup_{
\substack{e_1\in a_1,\ldots,e_m\in a_m \\ \{e_1,\ldots,e_m\}\in\CC}} 
\Vert\covpos(\covset/{\{e_1,\ldots,e_m\}})\Vert
.$$
The right-hand side union is disjoint by Lemma~\ref{lem:contractionaction}, and by Corollary~\ref{cor:tpa} each of its members maps under $Q_G$ to a torus 
 of dimension $d-\rk(\covset[A])$. Therefore, the connected components of $\sT_{a_1} \cap\ldots\cap \sT_{a_m}$ correspond to $G$-orbits of maximal elements of  $\CC(\covset[A])$, the index set of the r.-h.s.\ union. More precisely, the component $\Vert\covpos(\covset/{\{e_1,\ldots,e_m\}})\Vert$ corresponds to such an orbit $G\{e_1,\ldots,e_m\}$ and has dimension $\rk(\covset) - \rk(\{e_1,\ldots,e_m\})$. Let $X$ be the closure of $\{e_1,\ldots,e_m\}$ in $\SS(\covset)$. By Lemma~\ref{lem:contractionaction}.(4), we have  $\Vert\covpos(\covset/{\{e_1,\ldots,e_m\}})\Vert = \Vert\covpos(\covset/X))\Vert$ and $\rk(X)=\rk(\{e_1,\ldots,e_m\})$.
  Comparing with the definition of $\flatpos(\covset)/G$ gives the claim.  
\end{proof}

\begin{theorem}\label{MT:toric}
	Let $\alpha: G\circlearrowright \covset$ denote a free and \pcp action of a finitely generated free abelian group $G$ on an FAOM $\covset$, and suppose that $\SS(\covset)$ has a basis $B\in E^{*,*}$. We use $\alpha$ to denote the induced $G$-semimatroid $\alpha: G\circlearrowright \SS(\covset)$, and let $T_\alpha(x,y)$ be the associated Tutte polynomial (see Definition \ref{def:tutte}).
	Then $T_\alpha(x,y)$ computes the number of connected components of the arrangement's complement as follows:
	$$\left\vert \pi_0\left(T\setminus \cup\AA_\alpha\right)\right \vert = T_\alpha (1,0)$$
\end{theorem}
\begin{remark}
	In analogy with the case of enumeration of faces of arrangements \cite{Zold}, by passing to contractions and using the properties of Tutte polynomials of group actions derived in \cite{DeluRiedel}, Theorem \ref{MT:toric} leads to enumerating faces of  $\AA_\alpha$ of any dimension.
\end{remark}
\begin{proof}[Proof of Theorem \ref{MT:toric}]
	This is an application of Zaslavsky's dissection theory and of the theory of group actions on semimatroids. In fact, if $K$ is a finite $CW$-complex, $\{H_e\}_{e\in E}$ is any set of (proper) subcomplexes and $\{R_j\}_{j\in J}$ is the set of connected components of the complement $K\setminus \bigcup_{e\in E} H_e $, then~\cite[Lemma 1.1 and Theorem 1.2]{ZasDis} states that	 
	$$
	\sum_{j\in J} \kappa(R_j) = 	\sum_{p\in P } \mu(p)\kappa(p)
	$$
	where $P$ is the poset of connected components of intersections of the $H_e$s ordered by reverse inclusion, $\mu$ is the M\"obius function of $P$, and $\kappa$ is the ``combinatorial Euler number'' ($\kappa(X)$ equals Euler characteristic of $X$ if $X$ is compact, and otherwise it equals the Euler characteristic of the one-point compactification, see \cite[\S1]{ZasDis}).
	
	Now we can apply this to the CW-complex $K(\AA_\alpha)$ and the collection of subcomplexes $\AA_\alpha$. By Lemma~\ref{lem:QCW}, all $R_j$s are open $d$-cells and thus  
	$\sum_{j\in J} \kappa(R_j) =  (-1)^{d}\left\vert \pi_0\left(T\setminus \cup\AA_\alpha\right)\right \vert$. Moreover, by Proposition~\ref{prop_poslayers} the poset of connected components of intersections is $F(\SS)/G$ and every $GX\in F(\SS)/G$ is a torus $(S^1)^{d-r}$ where $r$ is the rank of $GX$ in $F(\SS)/G$. Thus, $\kappa(GX)\neq 0$ only if $GX$ is a maximal element of $F(\SS)/G$ and, in that case, $GX$ is a single point, whence $\kappa(GX)=1$. In summary, we can
	express the desired number of connected components as
 $$
 (-1)^{d}\left\vert \pi_0\left(T\setminus \cup\AA_\alpha\right)\right \vert = \sum_{GX\in \max F(\SS)/G} \mu_{F(\SS)/G}(GX)
 = \chi_{F(\SS)/G}(0),
  $$
  where the last equality is the definition of the characteristic polynomial (see \S\ref{a:posets}).
 
 Now, $\alpha$ is translative by Lemma~\ref{lem:ta}, and by Theorem~\ref{thm:TC} the characteristic polynomial of the poset $F(\covset)/G$ equals $(-1)^dT_\alpha (1-t,0)$, thus the claim follows.
\end{proof}

\section{Open questions}
\label{sec:OQ}

\subsection{Pseudoarrangements in Euclidean space} \label{oq_pseudo} 
Considerable work has been devoted to the study of arrangements of (and in) general manifolds - see, e.g., the introduction of \cite{EhReManifold} for a recent account. An open problem in this context is the one
 discussed in~\cite{FoZa}, namely to give a topological characterization of the arrangements that can appear as realizations of a (finite) AOM. Our work suggests the following reformulation.

We say that $Y$ is a pseudohyperplane in $X$ if there is a homeomorphism of $X$ with $\mathbb R^n$, for some $n$, that carries $Y$ to a coordinate hyperplane.  In particular, $X\setminus Y$ consists of two connected components, which we can label as the ``positive" and ``negative" side of $Y$ in $X$.
With this, call  {\em pseudoarrangement} in $\mathbb R^d$ any collection $\mathscr A=\{H_e\}_{e\in E}$ of pseudohyperplanes of $\mathbb R^n$ such that
(1) $\AA$ is finitary (every $p\in \mathbb R^d$  has a neighborhood that intersects only finitely many $H_e$), and chambers have finitely many walls; 
(2) if $Y=\bigcap_{e\in A}H_e$ is the intersection of some of the pseudohyperplanes and $e\in E$ is such that $H_e\cap Y\neq \emptyset$, then either $Y\subseteq H_e$ or $Y\cap H_e$ is a pseudohyperplane in $Y$;
(3) the intersection of any family of elements of $\mathscr A$, when nonempty, is {\em clean} in the sense of Bott (i.e., every point of $\mathbb R^d$ has an open neighborhood $U$ with a homeomorphism $U\to \mathbb R^n$ that sends the intersections of the $H_e$ with $U$ to a finite arrangement of linear hyperplanes, see \cite[Section 3]{EhReManifold} for a precise statement);
(4)  parallelism ($H_e\Vert H_f$ if and only if $H_e\cap H_f=\emptyset$) is an equivalence relation.

An ``oriented'' pseudoarrangement is a pseudoarrangement with a choice of a ``positive" and ``negative" side of every $H_e$. An oriented pseudoarrangement $\mathscr A$ gives rise to a set $\covset(\mathscr A)$ of sign vectors on $E$ as in \S\ref{intro:hyperplanes}. 

\begin{enumerate}[label=(Q\arabic*)]
	\item {\bf Conjecture:} For every oriented pseudoarrangement $\mathscr A$,  $\covset(\mathscr A)$ is the set of covectors of a unique simple FAOM. Conversely, every simple FAOM  arises this way.
\end{enumerate}
\begin{remark} \label{rem:FoZa}
Conditions (2) and (3) above define what Forge and Zaslavsky call an ``affine topoplane arrangement''. In the finite case, our conjecture  amounts to the converse of their \cite[Lemma 11]{FoZa}  and is proved in rank 2 as \cite[Theorem 13]{FoZa}. A positive answer to (Q1) would also solve the questions stated in \cite{FoZa} about topology of faces and structure of intersection semilattices of pseudoarrangements.
\end{remark}

\subsection{Toric oriented matroids and pseudoarrangements}\label{tomsp} 
The results of {Section~\ref{sec:gas}} suggest the categories $q_\alpha(\covpos(\covset))$ associated to a free and \pcp action on an FAOM $\covset$ as the counterpart of covector posets for toric arrangements, see Section~\ref{intro:toric} for an outline of this context. In order to fully develop an oriented matroid theory for arrangements on the torus we ask the following questions.

\begin{enumerate}[resume,label=(Q\arabic*)]
	\item Find an intrinsic axiomatic description of the class of acyclic categories that can be obtained as $q_\alpha(\covpos(\covset))$ for a free and \pcp action on an FAOM $\covset$. This framework should include for instance Aguiar and Petersen's posets of labeled necklaces \cite{AgPe}.
	\item Find a topological characterization of the pseudoarrangements on the torus that arise as $\AA_\alpha=\{\sT_a\}_{a\in E/G}$ in Section~\ref{sec:tpa}.
\end{enumerate}
In this context, Pagaria \cite{PagOM} proposed an algebraic notion of {\em orientable arithmetic matroid} and asked whether it can be interpreted in terms of pseudoarrangements in the torus. Every toric pseudoarrangement in the sense of Definition \ref{sec:tpa} has an associated matroid with multiplicity (via the induced group action on the underlying semimatroid), but the multiplicity does not have to be arithmetic  -- for example, the one associated to the pseudoarrangement in \cite[Fig. 11]{DeluRiedel} is not. On the other hand, for example the multiplicity matroid underlying to the non-stretchable infinite pseudoarrangement in \cite[Fig. 2, left-hand side]{DeluRiedel} is arithmetic, and indeed orientable -- take any orientation of the underlying uniform matroid $U_{2,5}$ -- (however, of course it has not the ``GCD'' property from \cite[\S 8]{PagOM}).

\begin{enumerate}[resume,label=(Q\arabic*)]
	\item Does every orientable arithmetic matroid as defined by Pagaria \cite{PagOM} arise from a toric pseudoarrangement in the sense of \S\ref{sec:tpa}? 
\end{enumerate}

\subsection{Further topological interpretation: Salvetti complexes} To every oriented matroid one can associate a cell complex that, in case the oriented matroid comes from a finite arrangement of hyperplanes, carries the homotopy type of the complement of the arrangement's complexification. This cell complex is usually called ``Salvetti complex", and its cohomology is described by the underlying matroid's Orlik-Solomon algebra~{\cite{GeRy}}. The construction can be carried out also for infinite affine arrangements and, indeed, for every FAOM $\covset$, yielding a regular CW-complex with poset of cells $M(\covset)$. Moreover, given a group action $\alpha: G\circlearrowright \covset$,  we can consider the category $\mathscr M(\alpha):=M(\covset)\qc G$. If $\alpha$ is the action by deck transformations on the periodic hyperplane arrangement that is obtained by lifting a toric arrangement to the universal cover of the $d$-torus, then $\mathscr M(\alpha)$ is the category studied in~\cite[\S2.5]{DaDeJEMS}. In particular, the Poincar\'e polynomial of $\Vert\mathscr M(\alpha)\Vert$ is $t^dT_\alpha(2+1/t,0)$ (recall Remark~\ref{rem:ArMat} and see~\cite[Theorem 3.5.(2)]{MoDa}), the space $\Vert\mathscr M(\alpha)\Vert$ is minimal~\cite[Corollary 6.10]{DaDeJEMS} and its cohomology algebra can be described by an Orlik-Solomon type presentation~\cite{a5}. Now let $\alpha$ be any free and \pcp action of a finitely generated free abelian group $G$ on a FAOM with a basis $B\subseteq E^{*,*}$.

\begin{enumerate}[resume,label=(Q\arabic*)]
	\item Is $t^{\vert B \vert}T_\alpha(2+1/t,0)$ the Poincar\'e polynomial of $\Vert\mathscr M(\alpha)\Vert$? Is  $\Vert\mathscr M(\alpha)\Vert$  minimal?
	\item Is there a presentation of the cohomology algebras $H^*(\Vert\mathscr M(\alpha)\Vert,\mathbb Q)$ and $H^*(\Vert\mathscr M(\alpha)\Vert,\mathbb Z)$ that extends the one for the case of toric arrangements given in~\cite{a5}?
\end{enumerate}

\subsection{Beyond FOAMs}
In many of our results one can observe that indeed not all axioms of FAOMs are needed. This leads to questions concerning the generalizability of the theory. We want to single out the study of finitary and general COMs and their group actions as a worthwhile endeavour. A natural problem to attack in this spirit is the following.
\begin{enumerate}[resume,label=(Q\arabic*)]
	\item Characterize tope graphs of finitary or general COMs.
\end{enumerate}
Note that some classes of tope graphs of finitary COMs have been studied already, e.g., median graphs~\cite{Ban-83,Imr-09,Mar-14} and more generally hypercellular graphs~\cite{Che-16}.

\appendix
\section{Appendix}

\subsection{Simplicial complexes and posets}\label{sec:simplicial}

We collect here some basic terminology and background material about combinatorial topology. We will only mention what is strictly necessary in order to make our treatment self-contained, and refer the reader to the  literature for more.

\subsubsection{Posets}\label{a:posets}
A partially ordered set (or poset for short) is a set $P$ with a reflexive, antisymmetric and transitive binary relation on $P$, usually denoted by $\leq$. 
The relation $\leq$ is a total order, and  $P$ is a totally ordered set if, for all $p,q\in P$ one of $p\leq q$ or $q\leq p$ holds. If $P$ has a unique maximal (resp.\ minimal) element, this is denoted $\top_P$ (resp.\ $\bot_P$) and $P$ is called ``bounded above'' (resp.\ ``bounded below''). A bounded poset is one that is both bounded above and below. Write $\overline{P}:=\{P\setminus \bot_P\}$ (so that $\overline{P}=P$ if $P$ is not bounded below), let $\wtop{P}:= P\uplus \top$, resp.\ $\wbot{P}$, denote the poset $P$ extended by a new, maximal (resp.\ minimal) element. Moreover, write $\wtb{P}$ for $\wbot{(\wtop{P})}$.

A subset $Q\subseteq P$ is called an {\em order ideal} if $x\in Q$ and $y\leq x$ imply $y\in Q$; dually, an {\em order filter} is any $Q\subseteq P$ such that $x\in Q$ and $y\geq x$ imply $y\in Q$. To every $A \subseteq P$ we associated the order ideal $P_{\leq A}:=\{x\in P\mid \exists a\in A :\: x\leq a\}$ and the order filter $P_{\geq A}:=\{x\in P\mid \exists a\in A:\: x\geq p\}$. The {\em meet}, resp.\ {\em join}, of $p,q\in P$ is defined as $p\vee q:= \top_{P_{\leq p}\cap P_{\leq q}}$, resp.\  $p\wedge q :=\bot_{P_{\geq p}\cap P_{\geq q}}$, if this exists. We call $P$ a {\em meet-semilattice}, resp.\ {\em join-semilattice} if $p\vee q$, resp.\ $p\wedge q$, exists for all $p,q\in P$. A {\em lattice} is a meet-semilattice that is also a join-semilattice.

A poset is {\em pure} if every maximal chain $p_0<p_1<\cdots < p_\ell$ in $P$ has the same length $\ell$; this number is called the {\em length} of $P$. If for a given $x\in P$ the poset $P_{\leq x}$ is pure, then its length is the {\em rank} of $x$. If the rank of every $x\in P$ is defined and finite, $P$ is called {\em ranked}. The poset $P$ is {\em graded} if it is bounded and ranked. Write $p\lessdot q$, and say ``$q$ covers $p$'', if $p\leq x < q$ implies $p=x$. A {\em saturated chain} is one of the form $p_0\lessdot p_1\lessdot \cdots \lessdot p_\ell$. If $P$ is bounded-below, the {\em atoms} of $P$ are the elements that cover $\bot_P$.

The characteristic polynomial of a bounded-below, ranked poset $P$ of finite length $\ell$ is 
$$
\chi_P(t):=\sum_{p\in P} \mu_P(p)t^{\ell-\operatorname{rank}(p)},
$$
where $\mu_P(p)$ is the M\"obius Function of $P$, defined recursively as $\mu_P(\bot)=1$ and, for all $p\in P$, $\sum_{x\leq p} \mu_P(x)=0$.

\subsubsection{Simplicial complexes}\label{app:sc}
An abstract simplicial complex on the vertex set $V$ is any family $\Sigma$ of finite subsets  of $V$ 
such that every subset of a member of the family is again a member of the family (i.e., $\sigma\in \Sigma$ and $\tau\subseteq \sigma$ implies $\tau\in \Sigma$). 

A geometric simplicial complex is a collection of  simplices (i.e., convex hulls of finite, affinely independent sets of points) in Euclidean space that is closed under taking faces of simplices, and such that the intersection of any two simplices in the collection is a face of both. The union of all simplices of a geometric simplicial complex is called the {\em underlying space} of the complex. The collection of all sets of vertices of simplices of a geometric simplicial complex is an abstract simplicial complex, which we say is ``realized" by the given geometric complex. Since any two geometric complexes realizing the same abstract complex have homeomorphic underlying spaces, it makes topological sense to talk about ``the" geometric realization of an abstract simplicial complex. The {\em dimension} of an abstract simplex $\sigma\in \Sigma$ is one less than its cardinality, which corresponds to the dimension of any geometric simplex realizing the abstract complex $2^\sigma$.
A geometric (resp.\ abstract) simplicial  complex is called {\em pure} if all its maximal simplices (resp.\ its maximal elements) have the same dimension (resp.\ cardinality). It is customary to call {\em facets} of a simplicial complex its nontrivial faces of maximal dimension. A facet's facet is a {\em ridge} of the given complex.

A powerful tool for the study of the topology of simplicial complexes is the theory of shellability introduced by Bj\"orner.  

\begin{definition}[See {\cite[Definition 1.1 and Remark 4.21]{BjoInf}}] An abstract simplicial complex  $\Sigma$ pure of dimension $d$ is called shellable the set of its maximal elements can be well-ordered in a way that, for every $\sigma$ that is not the initial element in the order, $\sigma \cap \left(\bigcup_{\sigma'<\sigma} 2^{\sigma'} \right)$ is a pure, $(d-1)$-dimensional complex. Shellable complexes have very restricted homotopy types (see, e.g.,~\cite{BjoInf}, and Remark~\ref{shellingrem}).
\end{definition}

The {\em link} of a simplex $\sigma\in \Sigma$ is the abstract simplicial complex $\operatorname{Lk}(\sigma):=\{\tau\in \Sigma \mid \tau\cap \sigma=\emptyset,\, \tau\cup \sigma\in \Sigma\}$. If $\Sigma$ is pure of dimension $d$, then $\operatorname{Lk(\sigma)}$ is pure of dimension $d-\vert \sigma\vert$. Moreover, every shelling order for $\Sigma$ induces a shelling order for $\operatorname{Lk(\sigma)}$. 

\subsubsection{Order complexes of posets}\label{app:oc}
To every partially ordered set $P$ one can associate the family $\Delta(P)$ of all finite subsets of $P$ that are totally ordered as sub-posets of $P$ (such finite subsets are called {\em chains} of $P$). The collection $\Delta(P)$ is called the {\em order complex} of $P$, and it is naturally an abstract simplicial complex, with $P$ as vertex set. We write $\Vert P \Vert$ for the geometric realization of the order complex of $P$. 

\begin{remark} Notice that the geometric realization of a chain is a simplex with one vertex for each element of the chain. Therefore, the length of the chain corresponds to the dimension of the simplex. The length of a poset $P$ is the maximum length of a chain in $P$ (if this maximum exists), and thus corresponds to the dimension of $\Vert P \Vert$. If $P$ is a ranked poset, for every $p\in P$ the dimension of $\Vert P_{\leq p}\Vert$ equals the rank of $p$.
\end{remark}

There are several techniques available in order to determine whether the order complex of a poset is shellable. We recall the notion of a recursive coatom ordering.

\begin{definition}[{\cite[\S 7]{WW}}]
Let $P$ be a graded poset of finite length $\ell$ and let $\prec$ be a well-ordering of the set $U$ of coatoms of $P$ (i.e., the elements of $P$ covered by $\top$).  This ordering is a {\em recursive coatom ordering} for $P$ if either $\ell\leq 2$, or $\ell > 2$ and, for every coatom $a$ of $P$ there is a recursive coatom ordering of $P_{\leq a}$ in which the elements of the set $Q_a$ of coatoms of $P_{\leq a} \cap \left( \bigcup_{a'\prec a}P_{\leq a'}\right)$ come first. 
\end{definition} 

\begin{remark}
Notice that, for every $p\in P$, the set $Q_p$ is completely determined by the ordering $\prec$ (see discussion after~\cite[Definition 4.7.17]{bjvestwhzi-93}).
\end{remark}

\begin{remark}[Recursive coatom orderings and shellings]\label{shellingrem}
A recursive coatom ordering on a graded poset $P$ induces a total order on the chains of $P$ that is a  shelling order for $\Delta(P)$, and implies that $\Vert P \Vert$ is homotopy equivalent to a wedge of spheres. Those spheres are indexed by ``critical chains", i.e.,  maximal chains  $\omega\subseteq P$ such that for every $p\in \omega$ the set $\omega\setminus \{p\}$ is contained in some chain that comes before $\omega$ in the ordering. Notice that whether a maximal chain $\omega$ is critical or not only depends on the ordering of the coatoms up to $\max\omega$.
 \end{remark}

 \begin{remark}[Face posets of regular CW-complexes]\label{rem:faceposet} To every regular CW-complex $K$ we can associate the poset $\FF(K)$ of all (closed) cells of $K$, partially ordered by inclusion. Then, $\Vert \FF(K)\Vert$ is homeomorphic to $\FF(K)$ (in fact, the order complex of $\FF$ is the barycentric subdivision of $K$). 
 \end{remark}

Let $P$ be a graded poset. Following~\cite[\S 4.7]{bjvestwhzi-93} we call $P$ {\em thin} if every interval of length $2$ in $P$ has exactly $4$ elements. We call $P$ {\em subthin} if intervals of length $2$ have either $3$ or $4$ elements, where the first case enters only if the interval contains $\top$, and indeed does enter at least once. Recall that the poset of cells of a regular CW complex is the set of all (closed) cells, ordered by inclusion.
\begin{proposition}\label{propthin}
Let $P$ be a finite, graded poset of length $\ell+2$.
\begin{itemize}
\item[(1)] $\overline{P}$ is the poset of cells of a shellable, regular cell decomposition of the $\ell$-sphere if and only if $P$ is thin and admits a recursive coatom ordering.
\item[(2)] $\overline{P}$ is the poset of cells of a shellable, regular cell decomposition of the $\ell$-ball if and only if $P$ is subthin and admits a recursive coatom ordering. Moreover, the maximal simplices in the boundary of $\Vert P \Vert$ correspond exactly to saturated chains $\omega'\subseteq P$ of length $\ell-1$ and such that the interval $P_{\geq \max\omega'}$ has only $3$ elements.
\end{itemize}
\end{proposition}
\begin{proof} This is~\cite[Proposition 4.7.24]{bjvestwhzi-93}, the claim about the boundary in item (2) easily following from the fact, mentioned in the referred proof, that $\omega'$ indexes a boundary simplex if and only if it can be completed in exactly one way to a maximal chain of $\overline{P}$.  
\end{proof}

 \subsubsection{Group actions on posets and acyclic categories} \label{sec:ACGA}
 The geometric realization of the order complex of a poset can be generalized as follows. A small category $\mathcal C$ is called ``acyclic'' if the only invertible morphisms in $\mathcal C$ are endomorphisms, and the only endomorphisms are the identities.  \footnote{Two comprehensive references for this point of view are~\cite{kozlov} (in the context of computational algebraic topology) and~\cite{briha} (in the context of geometric group theory, and in particular without finiteness assumptions).} The geometric realization $\Vert \mathcal C \Vert$ of $\mathcal C$ is a regular CW-complex whose cells are indexed by composable chains of morphisms of $\mathcal C$~\cite[Appendix A1]{DePag}.  The category  $\mathcal C$ may arise from a poset $P$ by taking $P$ as the set of objects of $\mathcal C$ and declaring that there is at most a morphism between any two elements, and one morphism $p\to q$ exists when $p\leq q$: in this case the geometric realization of
 $\mathcal C$ is $\Vert P \Vert$. Conversely, to every acyclic category $\mathcal C$ is associated a poset $\underline{\mathcal C}$, i.e., the set of objects of $\mathcal C$ partially ordered by $x\leq y$ if and only if $\operatorname{Mor}_{\mathcal C}(x,y)\neq \emptyset$. We call $\mathcal C$ {\em ranked} if $\underline{\mathcal C}$ is.

  The category-theoretic point of view is useful when dealing with group actions, as we explain next. Let $\SmC$ denote the category of small categories. We call $\AC$ the full subcategory of $\SmC$ with objects all acyclic categories. Although $\AC$ is cocomplete, its colimits do not coincide with colimits taken in $\SmC$. In particular, given $\mathcal C\in \Ob{\AC}$ and an action of a group $G$ on $\mathcal C$, we write $\mathcal C \qc G$ for the quotient object in $\AC$. See~\cite[\S 5.1]{DaDeJEMS} for a study of colimits in $\AC$ and~\cite[Appendix A.4]{DePag} for an explicit treatment of quotients by group actions. In particular, if a group $G$ acts on an acyclic category $\mathcal C$ then it acts on the cell complex $\Vert \mathcal C\Vert$ and, under certain conditions there is an isomorphism of cell complexes $\Vert \mathcal C \qc G \Vert \simeq \Vert \mathcal C \Vert / G$ (see, e.g.,~\cite[Corollary A.20]{DePag}). 
 \begin{remark}[$P/G$ and $P\qc G$]\label{rem:QQ}
A poset $P$ can be thus considered as an object in the category $\POS$ of partially ordered sets and order preserving maps, or as an acyclic category, as described above. This distinction entails that on posets we may consider two different types of ``quotients''. Let a group $G$ act on a poset $P$. Then on the set $P/G$ of orbits of elements of $P$ we can consider a binary relation $\leq$ defined by $Gp\leq Gq$ if and only if $p\leq gq$ for some $g\in G$. Under some mild conditions, this defines a partial order relation  (e.g., when $P$ is of finite length, see~\cite[\S 2.1]{DaDeSRR}) and we can speak of the {\em quotient poset} $P/G$. It is important to underscore that in general  the poset $P/G$, viewed as an acyclic category, is different from the quotient category $P\qc G$.
\end{remark}

\subsection{Semimatroid theory}\label{sec:semimatroids}

We recall some basic definitions and set up some notations from the theory of semimatroids and geometric semilattices, and group actions thereon. We follow mainly~\cite{DeluRiedel}, to which we refer for a more thorough treatment as well as for a brief historical sketch of the subject.

\subsubsection{Finitary semimatroids}

\begin{definition}\label{def:FS} A \textit{finitary semimatroid} is a
  triple $\SS=(\ground,\CC,\rc)$ consisting of a (possibly infinite) set
  $\ground$, a non-empty simplicial complex \thinspace$\CC$ on $\ground$ and a function $\rc:\CC\rightarrow\nn$ satisfying the following conditions.
\begin{itemize}
  \item[(R1)] If $X\in\CC,$ then $0\leq \rc(X)\leq |X|.$
  \item[(R2)] If $X,Y\in\CC$ and $X\ssq Y,$ then $\rc(X)\leq\rc(Y).$
  \item[(R3)] If $X,Y\in\CC$ and $X\cup Y\in\CC,$ then $\rc(X)+\rc(Y)\geq\rc(X\cup Y)+\rc(X\cap Y).$
  \item[(CR1)] If $X,Y\in\CC$ and $\rc(X)=\rc(X\cap Y),$ then $X\cup Y\in\CC.$
  \item[(CR2)] If $X,Y\in\CC$ and $\rc(X)<\rc(Y),$ then $X\cup y\in\CC$ for some $y\in Y-X.$
\end{itemize}
\label{def:lrt}
\end{definition}

\begin{definition}\label{df:simple}
A {\em loop} of a semimatroid $\SS=(\ground,\CC,\rk)$ is any $s\in
\ground$ with $\rk(s)=0$. 
We call $\SS$ \textit{simple} if it has no loops and if  $\rk(x,y)=2$ for all $\{x,y\}\in\CC$ with $x\neq y.$
\end{definition}

\begin{definition}\label{def:latticeofflats} Let $\SS=(S,\CC,\rk)$ be a finitary semimatroid and
  $X\in\CC$. The \textit{closure of $X$ in $\CC$} is 
\begin{align*}
  \cl(X) :=\{x\in S\mid X\cup x\in\CC,\rk(X\cup x)=\rk(X)\}. 
\end{align*}

 A \textit{flat} of a finitary semimatroid $\SS$ is a set $X\in\CC$
 such that $\cl(X)=X.$ The set of flats of $\SS$ ordered by containment forms the \textit{poset of flats of} $\SS$.
  \label{defi:lc}
\end{definition}

\subsubsection{Geometric semilattices}

Geometric semilattices  offer a poset-theoretic cryptomorphism for semimatroids, just as geometric lattices do for matroids. However, the study of geometric semilattices goes back to work of Wachs and Walker~\cite{WW} that predates the introduction of semimatroids and does not restrict to the finite case.   We review the definition and prove a cryptomorphism, slightly improving on~\cite[Theorem E]{DeluRiedel}.
Recall that a set $A$ of atoms of a ranked, bounded-below meet-semilattice is called {\em independent} if $\vee A$ exists and its the poset-rank equals the cardinality of $A$. This is equivalent to saying that $A$ is minimal under all $A'$ such that $\vee A'=\vee A$.
\begin{definition}[See Theorem 2.1 in {\cite{WW}}]\label{df:GS}

  A  \textit{geometric semilattice} is a chain-finite ranked meet-semilattice $\LL$ 
 satisfying the following conditions.
\begin{itemize}
  \item[(GSL1)] Every (maximal) interval in $\mathcal L$ is a finite geometric lattice.
  \item[(GSL2)] For every independent set $A$ of atoms of $\LL$ and every $x\in\LL$ such that $\rl(x)<\rl(\vee\! A)$, there is $a\in A$ with $a\nleq x$ and such that $x\vee a$ exists. 
\end{itemize}\end{definition}

\begin{theorem} \label{thm:gsl}
 A poset $\LL$ is a  geometric semilattice if and only if it is isomorphic to the poset of flats of a finitary semimatroid. 
More precisely, there is a one-to-one correspondence between
\begin{itemize}
\item[(1)] Subposets $\LL\subseteq \PF(E)$  of the poset of finite subsets of a set $E$ ordered by inclusion, such that $\LL$ is a chain-finite ranked meet-semilattice with respect to set intersection,  such that (GSL2) holds and such that, for all $X\in \LL$,  $\LL_{\leq X}$ is the (geometric) lattice of flats of a matroid. 
\item[(2)] Finitary semimatroids $(E,\CC,\rl)$.
\end{itemize}
The correspondence is as follows: given $\LL$ as in (1) let $\CC:=\bigcup_{X\in \LL} 2^X$ and for every $A\in \CC$ define $\overline{A} := \min_{\LL}\{X\in \LL\mid X\supseteq A\}$ (this is well-defined because $\LL$ is a meet-semilattice). Let then $\rk(A):=\rl(\overline{A})$, where $\rl$ denotes the rank function of $\LL$. Then $(E,\CC,\rl)$ is a finitary semimatroid with poset of flats isomorphic to $\LL$. 
\label{thm:fsl}\end{theorem}
\begin{proof} 
The proof that the poset of flats of a semimatroid is a geometric semilattice can be found in ~\cite{DeluRiedel} (Notice also Remark 5.3. there).

Now let $\LL\subseteq\PF(E)$ be as in the claim and consider the triple $(E,\CC,\rl)$. The family $\CC$ is by definition a simplicial complex on $E$. Since $\LL$ is chain-finite, the function $\rl$, that measures length of maximal chains in intervals, takes value in $\mathbb N$, and so does $\rk$.
Now, (R1), (R2) and (R3) are statements about the value of $\rl$ on intervals of $\LL$ (specifically: (R1) on the interval $\LL_{\leq \overline{X}}$, (R2) on the interval $\LL_{\leq \overline{Y}}$, (R3) on the interval $\LL_{\leq \overline{X\cup Y}}$). By assumption these intervals are geometric lattices of flats of matroids, and hence the corresponding restriction of $\rk$ is a matroid rank function (on $\overline{X}$, resp.\ $\overline{Y}$, $\overline{X\cup Y}$) -- in particular, $\rk$ satisfies (R1), (R2), (R3), see, e.g.,~\cite{oxley}. 

For (CR1) Let $X,Y\in \CC$. Trivially $\overline{X\cap Y} \subseteq \overline{X}$, and so $\rk(X)=\rk(X\cap Y)$ implies $\overline{X\cap Y} = \overline{X}$. Now, $\overline{X\cap Y}\subseteq \overline{Y}$ and thus $\overline{X}=\overline{Y}$. In particular, $\LL$ contains an element (e.g., $\overline{X}$) that contains $X\cup Y$, whence $X\cup Y\in \CC$.

For (CR2), let $X,Y\in \CC$ with $\rk(X)<\rk(Y)$ and consider $A_Y:=\left\{ \overline{\{y\}} \mid y\in Y\right\}$. Then, $\vee A_Y=\overline{Y}$ (the inclusion $\supseteq$ is trivial, and since $ \overline{\{y\}}\subseteq \overline{Y}$ for all $y\in Y$ we have $\overline{Y}\supseteq \bigcup_{y\in Y} \overline{\{y\}}$  hence  $\vee A_Y\subseteq \overline{Y}$). Choose $A\subseteq A_Y$ minimal such that $\vee A = Y$. Then $A$ is independent.
Now (GSL2) implies that there is $a\in A$, $a\not\subseteq \overline{X}$, such that $\overline{X}\vee a$ exists. Choose $y\in Y$ such that $\overline{\{y\}}=a$. Then, $X\cup\{y\}\in \CC$, $y\in Y$ and $y\not \in X$ (the latter since $a\wedge \overline{X}=\hat{0}$, the set of loops, and since $y$ is not a loop $y\not \in \overline{X}$).
\end{proof}

\subsubsection{Group actions on semimatroids}\label{app:GAOS}

\begin{definition}
An action of a group $G$ on a semimatroid $\SS=(E,\CC,\rk)$ is an action of $G$ on $E$ by permutations that preserves $\CC$ and $\rk$. I.e., for every $g\in G$ and every $X\in \CC$ we have $gX\in \CC$ and $\rk(gX)=\rk(X)$.
	The action is called {\em translative} if, for every $e\in E$,
		$\{e,g(e)\}\in \CC$ implies $g(e) =e$.
\end{definition}

Given such an action, for every $X\in \CC$ we can define $\lfloor X\rfloor :=\{ Gx\mid x\in X\}\subseteq E/G$. With this notation, for each $A\subseteq E/G$ let us write

\def\rka{\rk_\alpha}
\begin{equation}\label{eq:r-m}
\rka(A):=\max\{\rk(X)\mid X\in \CC, \, \lfloor X \rfloor = A\},\quad
m_\alpha(A):=\vert \{X\in CC\mid \lfloor X \rfloor = A\}/G\vert
\end{equation}

\begin{definition}\label{def:tutte}	
Consider an action $\alpha: G\circlearrowright \SS$ such that $E/G$ is finite. The associated {\em Tutte polynomial} is
$$
T_\alpha(x,y):=\sum_{A\subseteq E/G} m_\alpha(A)(x-1)^{\rka(E/G) - \rka(A) }(y-1)^{\vert A \vert - \rka(A)}.
$$
\end{definition}

Recall that an action $\alpha: G\circlearrowright \SS$ induces an action of $G$ on the geometric semilattice $\LL(\SS)$ of closed sets. This action is by semilattice automorphisms, and these are in particular rank-preserving. It follows that the set $\LL/G$ of orbits of elements $\LL$ has a natural partial order: given $X,Y\in \LL$ let $GX\leq GY$ if $X\leq gY$ for some $g\in G$.

\begin{theorem}[{\cite[Theorem F]{DeluRiedel}}]\label{thm:TC} If $\alpha$ is translative and $\LL/G$ is finite, then
$$
\chi_{\LL/G}(t) = (-1)^{\rka(E/G)} T_\alpha (1-t,0).
$$
\end{theorem}

\begin{remark}\label{rem:ArMat}
	When $\alpha$ is the group action on the semimatroid associated to a toric arrangement (via the associated periodic hyperplane arrangement, see Section~\ref{intro:toric}), then $T_\alpha(x,y)$ is the arithmetic Tutte polynomial of the given toric arrangement, see~\cite{MoDa,DeluRiedel}.
\end{remark}

\subsection{Finite oriented matroids, regular and affine}\label{FOMs}

We summarize some basics of the  theory of finite (affine) oriented matroids. See ~\cite{bjvestwhzi-93} for more.

\newcommand{\allzero}{\mathbf{0}}
\begin{definition}\label{def:OM} Let $E$ be a finite set. A system of sign vectors $\mathscr O\subseteq \signs^E$ is the set of covectors of an oriented matroid (OM) if and only if it satisfies the following axioms.
\begin{itemize}
\item[($0$)] $\allzero:=(0,\ldots,0)\in \mathscr O$.
\item[(Sym)] If $X\in \mathscr O$ then $-X\in \mathscr O$.
\item[(C)] If $X,Y\in \mathscr O$ then $X\circ Y \in \mathscr O$.
  \item[(SE$^=$)] $X,Y\in\mathscr O,\underline{X}=\underline{Y}\implies \forall e\in S(X,Y):I^=_e(X,Y)\neq\emptyset$.
\end{itemize}
\end{definition}
It is easy to see that (also for arbitrary ground sets $E$) the above definition is equivalent to requiring ($0$), (FS), and (SE), see~\cite{Ban-18}.

A fundamental theorem in oriented matroid theory says that $\mathscr O$ is an OM if and only if it is the set of sign vectors of an arrangement of pseudospheres (i.e., a set of centrally symmetric, tame embeddings of codimension-$1$ spheres in a sphere, with some conditions on their intersections, see~\cite[\S 5.1]{bjvestwhzi-93}). In this interpretation, finite affine oriented matroids appear as sign vectors of the arrangement induced on one of the the open hemispheres obtained by removing one pseudosphere from the ambient sphere. 

\begin{proposition}[{\cite[Theorem 2.1]{Bau-16}}]\label{prop:coningOM}
	Let $E$ be a finite set. A subset $\covset\subseteq \signs^E$ is the set of covectors of a finite AOM if and only if there is an oriented matroid $\mathscr O  \subseteq \signs^{E\uplus \{g\}}$ such that $\covset=\{X\in \mathscr O \mid X(g)=+\}$. 
\end{proposition}

\begin{remark}
	Notice that in the setting of Proposition~\ref{prop:coningOM} we have that $\covpos(\covset)$  is an order filter in $\covpos(\mathscr O)$. In particular, it is ranked (see also~\cite[Proposition 4.5.3]{bjvestwhzi-93}).
\end{remark}

Borrowing some terminology from the theory of arrangements, we say that $\mathscr O$ is obtained by {\em coning} $\covset$ with respect to the new element $g$. It is noteworthy that the set of sign vectors in $\mathscr O$ is fully determined by $\covset$ as in the following remark.

\begin{remark}[{\cite[Lemma 2.2]{Bau-16}}]\label{rem:ginf}
	Let $\covset$ be  a finite AOM on the ground set $E$. Let $\mathscr O$ be the set of covectors of the finite OM obtained by coning $\covset$ with respect to an additional element $g\not\in E$. Then $\mathscr O\setminus \left(\covset\cup -\covset\right) = \{N\in \signs^{E\cup g} \mid N(g)=0, (\pm N\setminus g)\circ \covset \subseteq \covset\}$.
\end{remark}

\begin{definition}\label{bdfc}
A {\em bounded face} of a finite AOM $\covset$ that arises from an OM $\mathscr O$ is any $X\in \covset$ such that $\covpos(\covset)_{\leq X}= \covpos(\mathscr O)_{\leq X}$, compare~\cite[Definition 4.5.1]{bjvestwhzi-93}. 
\end{definition}

\begin{lemma}\label{lem:bounded}
	Let $\covset$ be  a finite AOM on the ground set $E$, let $A\subseteq E$ and let $Y'$ be a bounded face of $\covset[A]$. Then every $Y\in \covset$ such that $Y_{\vert A}=Y'$ is a bounded face of $\covset$.
\end{lemma}
\begin{proof}
	Let $\mathscr O$ be the set of covectors of the finite OM obtained by coning $\covset$ with respect to an additional element $g\not\in E$. Let $Y'$ be bounded in $\covset[A]$ and let $Y$ be such that $Y_{\vert A}=Y'$. By definition, $Y$ is bounded if and only if $\mathscr O_{\leq Y} \subseteq \covset$. In view of Remark~\ref{rem:ginf}, this means that there is no $X\in \mathscr O_{\leq Y}$ such that $(\pm X)\circ \covset \subseteq \covset$. By way of contradiction suppose that such an $X$ exists and consider $X':=X_{\vert A}$. Then, for every $Z\in \covset[A]$ we can choose $\overline{Z}\in\covset$ such that $\overline{Z}_{\vert A}=Z$ and then in particular, $(\pm X')\circ Z = \left((\pm X)\circ \overline{Z}\right)_{\vert A} \in \covset[A] $  (since $(\pm X)\circ \overline{Z}\in \covset$). Since obviously $X'\leq Y'$ in $\covset[A]$, again by Remark~\ref{rem:ginf} we conclude that $Y'$ is unbounded in $\covset[A]$, a contradiction.
\end{proof}

\bibliographystyle{my-siam}
{\footnotesize \bibliography{minors}}

\begin{thebibliography}{10}

\bibitem{AgPe}
{\sc M.~Aguiar and T.~K. Petersen}, {\em The {S}teinberg torus of a {W}eyl
  group as a module over the {C}oxeter complex}, J. Algebraic Combin., 42
  (2015), pp.~1135--1175.

\bibitem{Alb-16}
{\sc M.~{Albenque} and K.~{Knauer}}, {\em {Convexity in partial cubes: the hull
  number.}}, {Discrete Math.}, 339 (2016), pp.~866--876.

\bibitem{Ardila}
{\sc F.~Ardila}, {\em Semimatroids and their {T}utte polynomials}, Rev.
  Colombiana Mat., 41 (2007), pp.~39--66.

\bibitem{Ban-18}
{\sc H.-J. {Bandelt}, V.~{Chepoi}, and K.~{Knauer}}, {\em {COMs: complexes of
  oriented matroids.}}, {J. Comb. Theory Ser. A}, 156 (2018), pp.~195--237.

\bibitem{Ban-83}
{\sc H.-J. {Bandelt} and H.~M. {Mulder}}, {\em {Infinite median graphs,
  (0,2)-graphs, and hypercubes}}, {J. Graph Theory}, 7 (1983), pp.~487--497.

\bibitem{Bau-16}
{\sc A.~{Baum} and Y.~{Zhu}}, {\em {The axiomatization of affine oriented
  matroids reassessed.}}, {J. Geom.}, 109 (2018).

\bibitem{BjoInf}
{\sc A.~Bj\"{o}rner}, {\em Some combinatorial and algebraic properties of
  {C}oxeter complexes and {T}its buildings}, Adv. in Math., 52 (1984),
  pp.~173--212.

\bibitem{bjvestwhzi-93}
{\sc A.~Bj{\"o}rner, M.~Las~Vergnas, B.~Sturmfels, N.~White, and G.~M.
  Ziegler}, {\em Oriented matroids}, vol.~46 of Encyclopedia of Mathematics and
  its Applications, Cambridge University Press, Cambridge, second~ed., 1999.

\bibitem{briha}
{\sc M.~R. Bridson and A.~Haefliger}, {\em Metric spaces of non-positive
  curvature}, vol.~319 of Grundlehren der Mathematischen Wissenschaften
  [Fundamental Principles of Mathematical Sciences], Springer-Verlag, Berlin,
  1999.

\bibitem{a5}
{\sc F.~Callegaro, M.~D'Adderio, E.~Delucchi, L.~Migliorini, and R.~Pagaria},
  {\em Orlik-{S}olomon type presentations for the cohomology algebra of toric
  arrangements}, Trans. Amer. Math. Soc., 373 (2020), pp.~1909--1940.

\bibitem{Che-16}
{\sc V.~Chepoi, K.~Knauer, and T.~Marc}, {\em Hypercellular graphs: Partial
  cubes without ${Q}_3^-$ as partial cube minor}, Discrete Math., 343 (2020),
  p.~111678.

\bibitem{MoDa}
{\sc M.~D'Adderio and L.~Moci}, {\em Arithmetic matroids, the {T}utte
  polynomial and toric arrangements}, Adv. Math., 232 (2013), pp.~335--367.

\bibitem{DaDeSRR}
{\sc A.~D'Alì and E.~Delucchi}, {\em Stanley-reisner rings for symmetric
  simplicial complexes, g-semimatroids and abelian arrangements}, 2018.

\bibitem{DaDeJEMS}
{\sc G.~d'Antonio and E.~Delucchi}, {\em Minimality of toric arrangements}, J.
  Eur. Math. Soc. (JEMS), 17 (2015), pp.~483--521.

\bibitem{DeCP}
{\sc C.~De~Concini and C.~Procesi}, {\em Topics in hyperplane arrangements,
  polytopes and box-splines}, Universitext, Springer, New York, 2011.

\bibitem{DePag}
{\sc E.~Delucchi and R.~Pagaria}, {\em The homotopy type of elliptic
  arrangements}, arXiv:1911.02905, to appear in Algebraic \& Geometric
  Topology.

\bibitem{DeluRiedel}
{\sc E.~Delucchi and S.~Riedel}, {\em Group actions on semimatroids}, Adv. in
  Appl. Math., 95 (2018), pp.~199--270.

\bibitem{EhReManifold}
{\sc R.~Ehrenborg and M.~Readdy}, {\em Manifold arrangements}, J. Combin.
  Theory Ser. A, 125 (2014), pp.~214--239.

\bibitem{FoZa}
{\sc D.~Forge and T.~Zaslavsky}, {\em On the division of space by topological
  hyperplanes}, European J. Combin., 30 (2009), pp.~1835--1845.

\bibitem{GeRy}
{\sc I.~M. Gel\cprime~fand and G.~L. Rybnikov}, {\em Algebraic and topological
  invariants of oriented matroids}, Dokl. Akad. Nauk SSSR, 307 (1989),
  pp.~791--795.

\bibitem{Huntington}
{\sc E.~V. Huntington}, {\em Inter-relations among the four principal types of
  order}, Trans. Amer. Math. Soc., 38 (1935), pp.~1--9.

\bibitem{Imr-09}
{\sc W.~{Imrich} and S.~{Klav\v{z}ar}}, {\em {Transitive, locally finite median
  graphs with finite blocks}}, {Graphs Comb.}, 25 (2009), pp.~81--90.

\bibitem{Kar-92}
{\sc J.~Karlander}, {\em A characterization of affine sign vector systems}, PhD
  Thesis, Kungliga Tekniska H{\"o}gskolan Stockholm, 1992.

\bibitem{Kawahara}
{\sc Y.~Kawahara}, {\em On matroids and {O}rlik-{S}olomon algebras}, Ann.
  Comb., 8 (2004), pp.~63--80.

\bibitem{Kna-17}
{\sc K.~Knauer and T.~Marc}, {\em On tope graphs of complexes of oriented
  matroids}, Discrete Comput. Geom.,  (2019), pp.~377--417.

\bibitem{kozlov}
{\sc D.~Kozlov}, {\em Combinatorial algebraic topology}, vol.~21 of Algorithms
  and Computation in Mathematics, Springer, Berlin, 2008.

\bibitem{Mar-14}
{\sc T.~{Marc}}, {\em {Regular median graphs of linear growth}}, {Discrete
  Math.}, 324 (2014), pp.~1--3.

\bibitem{Munkres}
{\sc J.~R. Munkres}, {\em Elements of algebraic topology}, Addison-Wesley
  Publishing Company, Menlo Park, CA, 1984.

\bibitem{oxley}
{\sc J.~Oxley}, {\em Matroid theory}, vol.~21 of Oxford Graduate Texts in
  Mathematics, Oxford University Press, Oxford, second~ed., 2011.

\bibitem{PagOM}
{\sc R.~Pagaria}, {\em Orientable arithmetic matroids}, Discrete Math., 343
  (2020), pp.~111872, 8.

\bibitem{Rourke-Sanderson}
{\sc C.~P. Rourke and B.~J. Sanderson}, {\em Introduction to piecewise-linear
  topology}, Springer-Verlag, New York-Heidelberg, 1972.
\newblock Ergebnisse der Mathematik und ihrer Grenzgebiete, Band 69.

\bibitem{Spanier}
{\sc E.~H. {Spanier}}, {\em {Algebraic topology}}, Berlin: Springer-Verlag,
  1995.

\bibitem{WW}
{\sc M.~L. Wachs and J.~W. Walker}, {\em On geometric semilattices}, Order, 2
  (1986), pp.~367--385.

\bibitem{Zold}
{\sc T.~Zaslavsky}, {\em Facing up to arrangements: face-count formulas for
  partitions of space by hyperplanes}, Mem. Amer. Math. Soc., 1 (1975),
  pp.~vii+102.

\bibitem{ZasDis}
{\sc T.~Zaslavsky}, {\em A combinatorial analysis of topological dissections},
  Advances in Math., 25 (1977), pp.~267--285.

\end{thebibliography}

\subsection*{Data availability statement:} This manuscript has no associated data.
\end{document}